\documentclass[a4paper, reqno]{amsart}

\title[{Correspondences of categories for subregular $\mathcal W$-algebras and principal $\mathcal W$-superalgebras}]{Correspondences of categories for subregular $\mathcal W$-algebras and principal $\mathcal W$-superalgebras}

\author{Thomas Creutzig}
\address{Department of Mathematical and Statistical Sciences, University of Alberta, 632 CAB, Edmonton, Alberta, Canada T6G 2G1}
\email{creutzig@ualberta.ca}

\author{Naoki Genra}
\address{Department of Mathematical and Statistical Sciences, University of Alberta, 632 CAB, Edmonton, Alberta, Canada T6G 2G1}
\email{genra@ualberta.ca}

\author{Shigenori Nakatsuka}
\address{Kavli Institute for the Physics and Mathematics of the Universe (WPI), The University of Tokyo Institutes for Advanced Study, The University of Tokyo, Kashiwa, Chiba 277-8583, Japan}
\email{shigenori.nakatsuka@ipmu.jp}

\author{Ryo Sato}
\address{Institute of Mathematics, Academia Sinica, Taipei, Taiwan 10617}
\email{sato@gate.sinica.edu.tw}

\usepackage{latexsym}
\usepackage{amsmath}
\usepackage{amsfonts}
\usepackage{amssymb}
\usepackage{amsxtra}
\usepackage{mathrsfs}
\usepackage{eucal}
\usepackage[all]{xy}
\usepackage{color}
\usepackage{braket}
\usepackage{graphics, setspace}
\usepackage{etoolbox, textcomp}
\usepackage{comment}
\usepackage{xcolor}
\definecolor{rouge}{rgb}{0.85,0.1,.4}
\definecolor{bleu}{rgb}{0.1,0.2,0.9}
\definecolor{violet}{rgb}{0.7,0,0.8}
\usepackage[colorlinks=true,linkcolor=bleu,urlcolor=violet,citecolor=rouge]{hyperref}
\usepackage{cleveref}

\newtheorem{definition}{Definition}[section]
\newtheorem{proposition}[definition]{Proposition}
\newtheorem{theorem}[definition]{Theorem}
\newtheorem{corollary}[definition]{Corollary}
\newtheorem{lemma}[definition]{Lemma}
\newtheorem{remark}[definition]{Remark}

\newtheorem{assumption}[definition]{Assumption}
\newtheorem*{thmA}{Main Theorem 1}
\newtheorem*{thmB}{Main Theorem 2}

\numberwithin{equation}{section}

\newcommand{\N}{\mathbb{N}}
\newcommand{\Z}{\mathbb{Z}}
\newcommand{\C}{\mathbb{C}}
\newcommand{\End}{\operatorname{End}}
\newcommand{\Hom}{\operatorname{Hom}}

\newcommand{\Com}{\operatorname{Com}}
\newcommand{\id}{\operatorname{id}}
\newcommand{\Ker}{\operatorname{Ker}}

\newcommand{\ch}{\operatorname{ch}}
\newcommand{\tr}{\operatorname{tr}}

\newcommand{\E}{\mathcal{E}}
\newcommand{\F}{\mathcal{F}}
\newcommand{\G}{\mathcal{G}}
\newcommand{\K}{\mathcal{K}}
\newcommand{\Q}{\mathbb{Q}}
\newcommand{\g}{\mathfrak{g}}
\newcommand{\sll}{\mathfrak{sl}}

\newcommand{\Mod}{\text{-}\mathrm{mod}}
\newcommand{\Pic}{\operatorname{Pic}}
\newcommand{\irr}{\operatorname{Irr}}
\newcommand{\Rep}{\operatorname{Rep}}

\newcommand{\wtcat}[1]{\mathcal{W}^{#1}\text{-}\mathrm{mod}^{\sf{wt}}}
\newcommand{\modcat}[1]{\mathcal{W}^{#1}\text{-}\mathrm{mod}}
\newcommand{\coset}{\boldsymbol{\Omega}}
\newcommand{\relcoh}[1]{H^{\mathrm{{\sf rel}},#1}_\infty}
\newcommand{\Relcoh}{\mathbf{H}^{{\sf rel}}}

\newcommand{\W}{\mathcal{W}}
\newcommand{\sub}{f_{\mathrm{sub}}}
\newcommand{\prin}{f_{\mathrm{prin}}}

\newcommand{\subW}{\mathcal{W}_{\mathrm{\textsf{sb}}}(n, r)}
\newcommand{\prinW}{\mathcal{W}_{\mathrm{\textsf{pr}}}(r, n)}
\newcommand{\prinsW}{\mathcal{W}_{\mathrm{\textsf{spr}}}(n,r)}
\newcommand{\subWtwo}{\mathcal{W}_{\mathrm{\textsf{sb}}}(n, 2)}
\newcommand{\prinWtwo}{\mathcal{W}_{\mathrm{\textsf{spr}}}(n, 2)}
\newcommand{\Wplus}{\mathcal{W}^+}
\newcommand{\Wminus}{\mathcal{W}^-}
\newcommand{\Wpm}{\mathcal{W}^\pm}
\newcommand{\Wmp}{\mathcal{W}^\mp}
\newcommand{\spmod}{\mathbf{L}_{\mathrm{\textsf{spr}}}}
\newcommand{\sbmod}{\mathbf{L}_{\mathrm{\textsf{sb}}}}
\newcommand{\prmod}{\mathbf{L}_{\mathcal{W}}}
\newcommand{\SubW}{\mathcal{W}_{\mathrm{\textsf{sb}}}}
\newcommand{\PrinW}{\mathcal{W}_{\mathrm{\textsf{pr}}}}
\newcommand{\PrinsW}{\mathcal{W}_{\mathrm{\textsf{spr}}}}

\newcommand{\Hsub}{H^+}
\newcommand{\Hsup}{H^-}
\newcommand{\tHsub}{\widetilde{H}^+}

\newcommand{\Hpm}{H^\pm}
\newcommand{\Hmp}{H^\mp}
\newcommand{\tHpm}{\widetilde{H}^\pm}

\newcommand{\heisplus}{\pi^{\mathfrak{h}_+}}
\newcommand{\heisminus}{\pi^{\mathfrak{h}_-}}

\newcommand{\oboxtimes}{\operatorname{\boxtimes}}

\newcommand{\Proj}{\pi_{P/Q}}
\newcommand{\boson}[1]{\phi^{#1}}
\newcommand{\ssqrt}[1]{\operatorname{\sqrt{\smash[b]{#1}}}}

\newcommand{\gluing}[1]{\widetilde{\mathcal{V}}^{#1}}
\newcommand{\difexp}[1]{E\left(#1,z\right)}

\newcommand\doi[2]{\href{http://dx.doi.org/#1}{#2}}

\begin{document}
\maketitle

\markboth{Correspondences of categories for subregular $\mathcal W$-algebras and principal $\mathcal W$-superalgebras}{Correspondences of categories for subregular $\mathcal W$-algebras and principal $\mathcal W$-superalgebras}

\begin{abstract}
Based on the Kazama--Suzuki type coset construction and its inverse coset between the subregular $\W$-algebra for $\sll_n$ and the principal $\W$-superalgebra for $\sll_{1|n}$, we prove a weight-wise linear equivalence of their representation categories. 
Our main results are then improvements of these correspondences incorporating the monoidal structures. 
Firstly, in the rational case, we obtain the classification of simple modules and their fusion rules via simple current extensions from their Heisenberg cosets. 
Secondly, beyond the rational case, we use a certain kernel VOA together with a relative semi-infinite cohomology functor to get functors from categories of modules for the subregular $\W$-algebra for $\sll_n$ to categories of modules for the principal $\W$-superalgebra for $\sll_{1|n}$ and vice versa. 
We study these functors and in particular prove isomorphisms between the superspaces  of logarithmic intertwining operators.  As a corollary, we obtain a  correspondence of representation categories in the monoidal sense beyond the rational case as well.
\end{abstract}

\section{Introduction}

Let $\g$ be a simple Lie superalgebra, $f$ an even nilpotent element in $\g$ and $k$ a complex number. Then one associates via quantum Hamiltonian reduction to the affine vertex superalgebra of $\g$ at level $k$, $V^k(\g)$, the $\mathcal W$-superalgebra $\mathcal W^k(\g, f)$ \cite{FF1, KRW, KW}.
Historically, principal $\mathcal W$-algebras, that is $f$ principal nilpotent and $\g$ a Lie algebra, have received most attention. However more general $\mathcal W$-superalgebras and their cosets have recently received increased attention due to their relevance in quantum field theory \cite{GR} and moduli spaces of instantons on surfaces \cite{RSYZ}. These can be viewed as generalizations of the celebrated Alday--Gaiotto--Tachikawa correspondence \cite{AGT}. 
One of the major predictions of this new development is that cosets of certain $\mathcal W$-superalgebras obey isomorphisms, called triality \cite{GR}. This has been proven by Andrew Linshaw and one of us whenever at least one of the three involved $\mathcal W$-superalgebras is in fact a $\mathcal W$-algebra \cite{CL1, CL2}.
We are interested in using these new relations of different $\mathcal W$-algebras and superalgebras to get correspondences of their representation categories. Triality is a vast generalization of both Feigin--Frenkel duality and the coset realization of principal $\mathcal W$-algebras. We now quickly review this principal case. 

In the case of principal $\mathcal W$-algebras one has Feigin--Frenkel duality \cite{FF2}, that is $\mathcal W^k(\g, f_{\text{prin}}) \cong \mathcal W^\ell({}^L\g, f_{\text{prin}})$ with the levels being non-critical and related by
\begin{equation}\label{FF}
r^\vee (k + h^\vee)(\ell + {}^Lh^\vee) =1
\end{equation} 
and $r^\vee$ the lacety of $\g$, ${}^L\g$ the Lie algebra whose roots coincide with the coroots of $\g$ and $h^\vee, {}^Lh^\vee$ the dual Coxeter numbers. 
Moreover for simply-laced $\g$ one also has the coset realization of principal $\mathcal W$-algebras \cite{ACL2}, that is for $k + h^\vee$ not a non-positive rational number one has $\text{Com}(V^k(\g), V^{k-1}(\g) \otimes L_1(\g)) \cong  \mathcal W^\ell(\g, f_{\text{prin}})$ with $L_1(\g)$ the simple quotient of $V^1(\g)$. In this instance $k$ and $\ell$ are related by 
\begin{equation}\label{coset}
\frac{1}{k+h^\vee} + \frac{1}{\ell+h^\vee} = 1. 
\end{equation}
Let $P_+$ be the set of dominant weights of $\g$ and for $\lambda \in P_+$ denote by $V^k(\lambda)$ the highest-weight module of $V^k(\g)$ whose top-level is the irreducible highest-weight module $E_\lambda$  of $\g$ of highest-weight $\lambda$, see subsection \ref{Affine_Notation} for the precise definition. Let $\mathcal W^k(\lambda, f)$ be its image under quantum Hamiltonian reduction corresponding to the nilpotent element $f$. Then the coset Theorem can be restated as follows for generic $k$, 
\begin{equation}
V^{k-1}(\g) \otimes L_1(\g) \cong  \bigoplus_{\lambda \in P_+ \cap Q} V^k(\lambda) \otimes \mathcal W^\ell(\lambda, f_{\text{prin}})
\end{equation}
with $Q$ the root lattice of $\g$. 
We specialize to type $A$, that is $\g = \mathfrak{sl}_n$. Denote by $\omega_i$ the $i$-th fundamental weight. Let $L = \sqrt{n}\mathbb Z$ and define the map $s : P_+ \rightarrow L'/L$ by $s(\lambda) = \frac{i}{\sqrt{n}} + L$ if $\lambda = \varpi_i \mod Q$. With this notation the coset Theorem can be restated as a sum over all dominant weights, namely for generic $k$
\begin{equation}\label{ACL}
V^{k-1}( \mathfrak{sl}_n) \otimes \mathcal E^{\otimes n} \cong  \bigoplus_{\lambda \in P_+ } V^k(\lambda) \otimes \mathcal W^\ell(\lambda, f_{\text{prin}}) \otimes V_{s(\lambda)}
\end{equation}
with $\mathcal E$ the vertex superalgebra of a pair of free fermions and $V_L$ the lattice vertex algebra of $L$.
Decompositions of vertex algebras of this type together with the theory of vertex superalgebra extensions of \cite{CKM} can be very efficiently used to get 
nice functors between representation categories of different vertex algebras. Indeed there are fully faithful braided tensor functors from certain subcategories of the simple quotient $\mathcal W_\ell(\lambda, f_{\text{prin}})$ of $\mathcal W^\ell(\lambda, f_{\text{prin}})$ to certain categories of ordinary modules of $L_{k-1}(\g) \otimes L_1(\g)$ if $k-1$ is an admissible level \cite{CHY, C1} or $\g = \mathfrak{sl}_2$ at generic level $k$ \cite{CJORY}. Such correspondences of braided tensor categories have been conjectured in the context of the quantum geometric Langlands program \cite{AFO}. 
In this work we will study correspondences between representation categories of  subregular $\mathcal W$-algebra of $\mathfrak{sl}_n$ and  principal $\mathcal W$-superalgebra of $\mathfrak{sl}_{n|1}$. Our first method will also use the theory of vertex superalgebra extensions. This method is very efficient, however it only applies if one can ensure the existence of vertex tensor category structure in the sense of \cite{HLZ1}-\cite{HLZ8}. This is usually a very difficult problem and hence in addition we seek a different method.

In order to motivate our second approach we need to discuss generalizations of Feigin--Frenkel duality. 
Both Feigin--Frenkel duality and the coset realization of principal $\mathcal W$-algebras of type $A$ and $D$ are part of a large family of trialities relating cosets of $\mathcal W$-(super)algebras \cite{CL1, CL2}. We review the Feigin--Frenkel type duality in type $A$. 
Nilpotent elements of $\mathfrak{sl}_N$ are characterized by partitions of $N$. Let $N = n+m$ and $f_{n, m}$ the nilpotent element corresponding to the partition $(n, 1, \dots, 1)$ of $N$. Let $\psi = k+ n+m$ and for $n>0$ one sets $\mathcal W^\psi(n, m) := \mathcal W^k(\mathfrak{sl}_{n+m}, f_{n, m})$  and we only note that the $\mathcal W^\psi(0, m)$-algebra is defined slightly differently. $\mathcal W^\psi(n, m)$ has an affine subalgebra of $\mathfrak{gl}_m$ at level $k +n -1$ and the coset is denoted by $C^\psi(n, m)$. Next, consider $\mathfrak{sl}_{n|m}$. In this case nilpotent elements are characterized by pairs of partitions of $n$ and $m$. Let $f_{n|m}$ be the nilpotent element corresponding to the partition $(n|1, \dots, 1)$ of $(n|m)$ and set $\mathcal V^\phi(n, m) := \mathcal W^\ell(\mathfrak{sl}_{n|m}, f_{n|m})$ with $\phi = \ell + n-m$. Again the case $n=0$ is slightly differently and in fact is precisely \eqref{ACL}.
The superalgebra $\mathcal V^\phi(n, m)$ has an affine subalgebra of $\mathfrak{gl}_m$ at level $-\ell  -n  +1$ and the coset is denoted by $D^\phi(n, m)$ if $n\neq m$. We only note that the definition of $\mathcal V^\phi(n, n)$ and $D^\phi(n, n)$ are a little bit different since $\mathfrak{sl}_{n|n}$ is not a simple Lie superalgebra. The Feigin--Frenkel type duality is then \cite{CL1}
\begin{align}\label{isom of cosets}
    C^\psi(n -m, m) \cong D^{\psi^{-1}}(n, m)
\end{align}
for $n\geq m$. The case $m=0$ is precisely Feigin--Frenkel duality of type $A$. 
These results are a very good starting point for investigating the problem of finding a nice functor from certain representation categories of $\mathcal W^\psi(n - m, m)$ to categories of $\mathcal V^{\psi^{-1}}(n, m)$-modules. The idea is that one pairs $\mathcal W^\psi(n - m, m)$ with a certain kernel vertex operator algebra, such that $\mathcal V^{\psi^{-1}}(n, m)$ is realized as a cohomology. The candidate kernel vertex operator algebra is suggested from physics:
$S$-duality for four-dimensional $\mathcal N=4$ supersymmetric GL-twisted gauge theory is a physics pendant of the quantum geometric Langlands program and it especially conjectures (and proves for $\mathfrak{sl}_2$) the existence of certain junction vertex operator algebras, also called quantum geometric Langlands kernel VOAs \cite{CG}, see also  \cite{CGL, FG} for subsequent studies. In the case of type $A$, the central object is 
\begin{equation}\label{kernel}
A[\mathfrak{sl}_N, \psi] :=   \bigoplus_{\lambda \in P^+ } V^k(\lambda) \otimes V^\ell(\lambda)\otimes V_{s(\lambda)}
\end{equation}
with $k, \ell$ related by \eqref{coset} and $\psi = k+N$. 
This has conjecturally a simple vertex superalgebra structure for generic $\psi$. 
Comparing with \eqref{ACL}, we see that the principal quantum Hamiltonian reduction on the second factor gives the decomposition of $V^{k-1}( \mathfrak{sl}_n) \otimes \mathcal E^{\otimes n}$. In \cite{CL1} it is proven that in fact all the $\mathcal W^\phi(n, m)$ can be realized as a coset of a reduction of $A[\mathfrak{sl}_N, \psi]$ provided a certain existence and simplicity conjecture is true. There it is also conjectured that a certain relative semi-infinite Lie algebra cohomology of $\mathcal W^\phi(n -m , m) \otimes A[\mathfrak{sl}_m, 1- \psi]  \otimes \pi^{k-\ell}$ with $\pi^{k-\ell}$ a rank one Heisenberg vertex algebra is isomorphic as a vertex superalgebra to $\mathcal V^{\psi^{-1}}(n, m)$. This conjecture is also reduced to a certain existence and simplicity statement  \cite[Theorem 1.4]{CL1}. 
In the case of $m=1$, the kernel vertex algebra in this conjecture is replaced by $V_\Z\otimes \pi$ with $\pi$ a rank one non-degenerate Heisenberg vertex algebra. One major result of this work is that we not only prove the Conjecture for $m=1$, but especially understand how it provides a functor from 
$\mathcal W^\psi(n - 1, 1)$-modules to $\mathcal V^{\psi^{-1}}(n , 1)$-modules. Note that the   $\mathcal W^\psi(n - 1, 1)$-algebra is the subregular $\mathcal W$-algebra of $\mathfrak{sl}_n$ and  $\mathcal V^{\psi^{-1}}(n , 1)$ is the principal $\mathcal W$-superalgebra of $\mathfrak{sl}_{n|1}$ and in this case there is a very different proof of the duality that is rather useful for our understanding \cite{CGN}. Let us now explain our main results and especially the very nice properties of the functor from $\mathcal W^\psi(n - 1, 1)$-modules to $\mathcal V^{\psi^{-1}}(n , 1)$-modules.

\subsection{Results}
In this paper we employ two different methods to establish a correspondence between the representation theory of the subregular $\mathcal W$-algebra of $\mathfrak{sl}_n$ and that of the principal $\mathcal W$-superalgebra of $\mathfrak{sl}_{n|1}\simeq\mathfrak{sl}_{1|n}$. One method is a tensor categorical,   which is based on the vertex tensor category theory \cite{HLZ1}-\cite{HLZ8} together with the vertex operator superalgebra extension theory \cite{HKL, CKM}. This method only applies for categories of modules that are vertex tensor categories, as e.g.~for $C_2$-cofinite vertex operator algebras.  Alternatively, the other method is of cohomological nature and it does not require the existence of vertex tensor categories.
In order to state our main results on the correspondence of monoidal structures, we first briefly review correspondences of linear structures.

As mentioned earlier,
\begin{equation}
    \W^+:=\W^\psi(n-1, 1)=\W^k(\sll_n,\sub),\quad \W^-:=\mathcal{V}^{\psi^{-1}}(n, 1)\simeq\W^\ell(\sll_{1|n},\prin)
\end{equation}
contain affine $\mathfrak{gl}_1$ subalgebras, denoted by $\pi^+$ and $\pi^-$, respectively.
In \cite{CGN}, the first three authors of this paper introduced two diagonal coset constructions
\begin{equation}\label{intro:KS_and_inverse}
    \W^-\xrightarrow{\simeq}\Com(\Delta(\pi^+),\W^+\otimes V_\Z),\quad\W^+\xrightarrow{\simeq}\Com(\Delta(\pi^-),\W^-\otimes V_{\sqrt{-1}\Z}).
\end{equation}
Here $V_L$ stands for the lattice vertex superalgebra associated to an integral lattice $L$ and $\Delta(\pi^\pm)$ for certain diagonal Heisenberg subalgebras of $\W^\pm \otimes V_{\sqrt{\pm 1}\Z}$. When $n=2$, the former coincides with the Kazama--Suzuki coset construction firstly studied by \cite{DVPYZ,KaSu} at unitary levels, and the latter coincides with its inverse coset construction firstly studied by \cite{FST} at general levels (see also \cite{Ad1,Ad2,HM,Sato1, CLRW}).
Based on the same technique used in \cite{Sato1} (see also \cite{FST}), we consider the categories of $\W^\pm$-modules on which $\pi^\pm$ acts semisimply with the $\mathfrak{gl}_1$-weight decomposition (see \S\ref{Coset_Functor} for the definition)
\begin{equation}
    \W^\pm\Mod^{\sf wt}=\bigoplus_{[\lambda]\in\C/\Z}\W^\pm\Mod^{\sf wt}_{[\lambda]}
\end{equation}
and establish linear categorical equivalences
\begin{equation}\label{categorical equivalence: abelian}
    \W^+\Mod^{\sf wt}_{[\lambda]}\xrightarrow{\simeq}\W^-\Mod^{\sf wt}_{[\check{\lambda}]},\quad    \W^-\Mod^{\sf wt}_{[\check{\lambda}]}\xrightarrow{\simeq}\W^+\Mod^{\sf wt}_{[\lambda]}
\end{equation}
for appropriate pairs $(\lambda,\check{\lambda})$.
In addition, these equivalences are also valid if we replace $\W^\pm$ with their simple quotients (see \S\ref{sec:Kazama-Suzuki} for the precise statements). In the following, we lift this result to correspondences that include the monoidal structure of the categories. 

 \subsubsection{Correspondences via categorical methods}
\label{categorical method}
The simple subregular $\W$-algebra $\W_k(\sll_n,\sub)$ at the levels
\begin{align}\label{condition for r}
    k=-n+\frac{n+r}{n-1},\quad (r\in \Z_{\geq0},\ \mathrm{gcd}(n+r,r-1)=1)
\end{align}
is called exceptional. Its $C_2$-cofiniteness is established by Arakawa \cite{Ar2} and rationality was first proven for $n=3$ \cite{Ar1} and $n=4$ \cite{CL3} and then in general by Arakawa and van Ekeren \cite{AvE2}.
In this case, the Heisenberg subalgebra $\pi^+$ inside the subregular $\W$-algebra is extended to a lattice vertex algebra \cite{Mason} and thus the representation theory at these levels can be studied by the theory of \emph{simple current extensions} modulo the representation theory of the Heisenberg coset.   
Linshaw and one of the authors \cite{CL1} identified the Heisenberg coset with a principal $\W$-algebra 
\begin{equation}\label{intro: level-rank}
    \mathrm{Com}(\pi^+, \W_{-n+\frac{n+r}{n-1}}(\sll_n,\sub))\simeq \W_{-r+\frac{r+n}{r+1}}(\sll_r,\prin),
\end{equation}
which generalizes the level-rank duality for the $\sll_2$-parafermion algebra when $n=2$ \cite{ALY} as well as the cases $n=3$ \cite{ACL1} and $n=4$ \cite{CL3}. Here the right-hand side in \eqref{intro: level-rank} for the cases $r=0,1$ is interpreted as $\C$, see Section \ref{Fusion_Subregular}. The principal $\W$-algebra at these levels are $C_2$-cofinite and rational \cite{Ar2,Ar3} and the representation theory is completely understood \cite{Ar3}: the complete set of representatives of simple modules is in one-to-one correspondence with that of the simple affine vertex algebra $L_n(\sll_r)$ \cite{Ar3,FKW} and both of their fusion rules are the same \cite{AvE1, C1}. Hence it gives an isomorphism of their fusion rings
$$\mathcal{K}(L_n(\sll_r))\xrightarrow{\simeq} \mathcal{K}(\W_{-n+\frac{n+r}{n-1}}(\sll_n,\sub)),\quad L(\lambda)\mapsto \mathbf{L}_\mathcal{W}(\lambda),$$
where $L_n(\lambda)$ is the unique simple quotient of the previous $V^n(\lambda)$ and $\lambda$ runs through the set $P_+^n(r)$ of dominant weights of $\sll_r$ at level $n$.
Now, by \cite{CL3, CL1}, one can decompose the subregular $\W$-algebra into
$$\W_{-n+\frac{n+r}{n-1}}(\sll_n,\sub) \simeq \bigoplus_{i \in \Z_r} \mathbf{L}_\W(n\varpi_i) \otimes V_{\frac{ni}{\ssqrt{nr}}+\ssqrt{nr}\Z},$$
where $\varpi_i$ denotes the $i$-the fundamental weight, including the cases $r=0,1$ as $\C\otimes \C$ and $\C\otimes V_{\sqrt{n}\Z}$, respectively (see Remark \ref{r=0_1}). Note that the Kazama--Suzuki coset construction \eqref{intro:KS_and_inverse} implies a similar description for the simple principal $\W$-superalgebra. Namely, we have by Theorem \ref{decomposition for principal super}
\begin{align}\label{intro extension to superW}
\W_{-(n-1)+\frac{n-1}{n+r}}(\sll_{1|n},\prin)\simeq \bigoplus_{i\in \Z_r}\mathbf{L}_\mathcal{W}(n\varpi_i)\otimes V_{\frac{(n+r)i}{\ssqrt{(n+r)r}}+\ssqrt{(n+r)r}\Z},
\end{align}
including the cases $r=1,2$ as $\C\otimes \C$ and $\C\otimes V_{\ssqrt{n+1}\Z}$, respectively (see Remark \ref{r=0,1 for super}).

To deduce the classification of simple modules and the fusion ring of the simple subregular $\W$-algebra and principal $\W$-superalgebra, we view these (super)algebras as extensions of the principal $\W$-algebra tensored with lattice vertex (super)algebras. 
Based on earlier works \cite{HKL, CKL,CKM,KO} on simple current extensions, we consider the simple current extensions of the form
$$\mathcal{E}=\bigoplus_{a\in N/L} S_a\otimes V_{a+L},$$
where $N$ is an arbitrary lattice lying between a non-degenerate integral lattice $L$ and its dual lattice $L'$, and $\{V_{a+L}\,|\,a\in N/L\}$ (resp.~$\{S_a\,|\,a\in N/L\}$) is a group of simple currents for the lattice vertex superalgebra $V_L$ associated to $L$ (resp.~for a simple $C_2$-cofinite $\frac{1}{2}\Z$-graded vertex operator superalgebra $V$), satisfying some mild assumptions. 
Then the classification of simple modules and the description of fusion rings for the one are expressed by the other. In particular, the fusion rings are related in the following symmetric way:
\begin{equation}\label{duality by lattice}
\begin{split}
\mathcal{K}(\mathcal{E})&\simeq 
    \left(\mathcal{K}(V)\underset{\Z[N/L]}{\otimes}\Z[L'/L]\right)^{N/L},\quad \\
\mathcal{K}(V)&\simeq \left(\mathcal{K}(\mathcal{E})\underset{\Z[N'/L]}{\otimes}\Z[L'/L]\right)^{N'/L},
\end{split}
\end{equation}
see Theorem \ref{extension law} and \ref{Inverting_SCE} for details.
We note that this kind of duality in the rational setting between fusion rings for simple current extensions is essentially obtained by Yamada and Yamauchi \cite{YY} and  by Kanade, McRae and one of the authors \cite{CKM}. See also \cite{ADJR} in the case of parafermion vertex algebras.

By applying this general theory of simple current extensions and the level-rank duality \eqref{level-rank duality} of the fusion rings of $L_n(\sll_r)$ and $L_r(\sll_n)$ \cite{Fr,OS}, we obtain the following main theorem in the rational case:
\begin{thmA}\label{Main_Theorem_1}
Let $r\in\Z_{\geq0}$ such that $\mathrm{gcd}(n+r,r-1)=1$.\\
(1) The complete set of simple modules for $\W_k(\sll_n,\sub)$ at $k=-n+\frac{n+r}{n-1}$ is in one-to-one correspondence with that of $L_r(\sll_n)$, i.e., $P_+^r(n)$. Moreover, we have an isomorphism of fusion rings:
$$\mathcal{K}\left(\W_k(\sll_n,\sub)\right)\simeq \mathcal{K}(L_r(\sll_n)).$$
(2) The complete set of simple modules for $\W_\ell(\sll_{1|n},\prin)$ at $\ell=-(n-1)+\frac{n-1}{n+r}$ is in one-to-one correspondence with
$$\left\{(\lambda,a)\in P_+^r(n)\times \Z_{n(n+r)}\,\big|\, \Proj(\lambda)=a\in \Z_n\right\}\Bigm/\Z_n.$$
Moreover, we have an isomorphism of fusion rings:
$$\mathcal{K}(\W_\ell(\sll_{1|n},\prin))\simeq\left(\mathcal{K}(L_r(\sll_n))\underset{\Z[\Z_n]}{\otimes}\Z[\Z_{n(n+r)}]\right)^{\Z_n}.$$
\end{thmA}
See Proposition \ref{fusion ring of subregular} and Theorem \ref{the final fusion for sprin}, respectively.
We note that for (1), the classification result of simple modules together with the description of the fusion ring when $n$ is even is already obtained in \cite{AvE2}. For (2), the case of $n=2$ corresponds to the $\mathcal{N}=2$ super Virasoro algebra and the classification result is obtained by Adamovi\'{c} \cite{Ad1} and the description of fusion ring essentially coincides with the one in \cite{Ad2}. We also note that the description of fusion ring in the case $n=2$ as in (2) was inferred by Wakimoto \cite{Wak} using the Verlinde formula.

Since the coset constructions \eqref{intro:KS_and_inverse} in the rational case are indeed simple current extensions, the fusion rings for $\W_k(\sll_n,\sub)$ and $\W_\ell(\sll_{1|n},\prin)$ in Main theorem 1 are actually related by the isomorphisms \eqref{duality by lattice}. This manifests the compatibility of the monoidal structure under the correspondence \eqref{intro:KS_and_inverse}.

In subsection \ref{more} we comment that this procedure only requires the existence of a vertex tensor category structure on the Heisenberg coset. This is however only known in two series of non-rational subregular $\W$-algebras, namely for $\W_k(\sll_{n},\sub)$ for $k = -n+\frac{n}{n+1}$ and  $k = -n +\frac{n+1}{n}$ by \cite{ACGY, CL1}. The Heisenberg cosets are singlet algebras and certain vertex tensor subcategories are understood \cite{CMY2}. We thus seek a  second method that does not require the existence of vertex tensor category.

\subsubsection{Correspondences via relative semi-infinite cohomology and a kernel VOA}

One of the features of the coset construction \eqref{intro:KS_and_inverse} is that the difference between the two  $\W$-superalgebras $\W^{\pm}$ is the dressings by the Fock modules of the Heisenberg vertex algebra of rank one. 
In order to see this, let us decompose $\W^{\pm}$ as modules of tensor products of the common Heisenberg coset $\mathscr{C}_0$ i.e., \eqref{isom of cosets} with $m=1$ and the Heisenberg vertex algebra:
\begin{align}
    \W^+\simeq \bigoplus_{a\in\Z} \mathscr{C}_a\otimes \pi^+_a,\quad \W^-\simeq \bigoplus_{a\in\Z} \mathscr{C}_a\otimes \pi^-_a
\end{align}
Here $\mathscr{C}_a$ ($a\in\Z$) are $\mathscr{C}_0$-modules appearing as the highest weight vectors for $\pi^{\pm}$ of highest weight $a$.
We see that in order to go from $\W^+$ to $\W^-$ we just need to replace $\pi^+_a$ by $\pi^-_a$.
We are looking for functors that not only replace $\pi^+_a$ by $\pi^-_a$ and vice versa, but that also allow us to relate the monoidal structures of the representation categories of $\W^+$ and $\W^-$ in an efficient way. 
Here comes the relative semi-infinite cohomology into play.

The relative semi-infinite cohomology was introduced by Feigin \cite{Fe} and by I.~Frenkel, Garland and Zuckerman \cite{FGZ}. It is a vertex algebraic or infinite dimensional Lie algebraic analogue of the usual relative Lie algebra cohomology.
Let $\pi^\pm_\dag$ denote a Heisenberg vertex algebra generated by a field $H_\dag^\pm(z)$ which satisfies the OPE equal to the minus of that of the generator $H^\pm(z)$ of $\pi^\pm$, i.e.,
$$H_\dag^\pm(z)H_\dag^\pm(w)\sim -H^\pm(z)H^\pm(w).$$
Then $\pi^\pm\otimes \pi_\dag^\pm$ contains a commutative vertex algebra generated by a diagonal Heisenberg field $A(z)=H^\pm(z)+H^\pm_\dag(z)$.
We may consider the semi-infinite cohomology 
$\relcoh{p}(\pi^\pm_a\otimes \pi^\pm_{\dag,b})$ of $\mathfrak{gl}_1(\!(z)\!)$ with coefficients in $\pi^\pm_a\otimes \pi^\pm_{\dag,b}$ relative to $\mathfrak{gl}_1$, which satisfies
$$\relcoh{p}(\pi^\pm_a\otimes \pi^\pm_{\dag,b})=\delta_{p,0}\delta_{a+b,0}\C.$$
The non-zero cohomology class is spanned by the tensor product of the highest weight vectors. To obtain $\W^-$ from $\W^+$ or $\W^+$ from $\W^-$, we consider the direct sums:
$$\bigoplus_{a\in \Z}\pi^+_{\dag,-a}\otimes \pi^-_{a},\quad \bigoplus_{a\in \Z}\pi^-_{\dag,-a}\otimes \pi^+_{a}.$$
By changing the bases of the Heisenberg fields inside $\pi^\pm\otimes \pi^\pm_\dag$, we find that they admit vertex superalgebra structure:
\begin{align}\label{gluing objects}
V_\Z\otimes \pi^{\ssqrt{-1}\Z},\quad V_{\ssqrt{-1}\Z}\otimes \pi^{\Z}.
\end{align}
Indeed, $V_\Z$ in the first is the kernel VOA \eqref{kernel} for $m=1$. Then we may reconstruct $\W^+$ from $\W^-$ and vice versa in the following way:
$$\W^-\simeq \relcoh{0}\left(\W^+\otimes V_\Z\otimes \pi^{\ssqrt{-1}\Z}\right),\quad \W^+\simeq \relcoh{0}\left(\W^-\otimes V_{\ssqrt{-1}\Z}\otimes \pi^{\Z}\right),$$
see Proposition \ref{Comparison_VA}. Note, that the first isomorphism in particular proves Conjecture 1.2 of \cite{CL1} for the case $m=1$, that is $\mathfrak{gl}_1$.
We may generalize this construction for modules by replacing the second factor in \eqref{gluing objects} by their Fock modules with appropriate highest weights. This actually defines functors weight-wisely:
\begin{align}\label{categorical equivalence}
\Relcoh_{+,\epsilon\lambda}\colon \wtcat{+}_{[\lambda]}\rightarrow \wtcat{-}_{[\check{\lambda}]},\quad \Relcoh_{-,\epsilon\lambda}\colon \wtcat{-}_{[\check{\lambda}]}\rightarrow \wtcat{+}_{[\lambda]},
\end{align}
which turns out to be naturally isomorphic to the ones in \eqref{categorical equivalence: abelian}, see Theorem \ref{Comparison_Module}. Therefore, $\Relcoh_{+,\epsilon\lambda}$ and $\Relcoh_{-,\epsilon\lambda}$ themselves give an equivalence of categories as abelian categories. 
The upshot of this method is that the the way of comparing the monoidal structure is clear: we may use the intertwining operators among Fock modules to compare the intertwining operators among $\W^\pm$-modules related by the equivalence \eqref{categorical equivalence}. The second main theorem of this paper is the following:
\begin{thmB}
For $i=1,2,3$, let $M_i^+$ be an object in $\wtcat{+}_{[\lambda_i]}$ and $M_i^-$ in  $\wtcat{-}_{[\check{\lambda}_i]}$.\\
(1) The functors $\{\Relcoh_{\pm,\epsilon\lambda_i}\}$ induce isomorphisms between the superspaces of logarithmic intertwining operators:
\begin{align*}
\mathbf{H}_{\pm}\colon I_{\Wpm}\binom{M_3^\pm}{M_1^\pm\ M_2^\pm}\xrightarrow{\simeq} I_{\Wmp}\binom{\Relcoh_{\pm,\epsilon\lambda_3}(M_3^\pm)}{\Relcoh_{\pm,\epsilon\lambda_1}(M_1^\pm)\ \Relcoh_{\pm,\epsilon\lambda_2}(M_2^\pm)}.
\end{align*}
\\
(2) Suppose that $M_1^\pm$ and $M_2^\pm$ admit a fusion product $M_1^\pm\boxtimes M_2^\pm$. Then the corresponding $\W^\mp$-module $\Relcoh_{\pm,\epsilon(\lambda_1+\lambda_2)}(M_1^\pm\boxtimes M_2^\pm)$ equipped with the image of the canonical intertwining operator of type $\binom{M_1^\pm\boxtimes M_2^\pm}{M_1^\pm\ M_2^\pm}$ gives a fusion product of $\Relcoh_{\pm,\epsilon\lambda_1}(M_1^\pm)$ and $\Relcoh_{\pm,\epsilon\lambda_2}(M_2^\pm)$. 
In particular, we have a natural isomorphism
\begin{align*}
 \Relcoh_{\pm,\epsilon\lambda_1}(M_1^\pm)\boxtimes\Relcoh_{\pm,\epsilon\lambda_2}(M_2^\pm)\simeq\Relcoh_{\pm,\epsilon(\lambda_1+\lambda_2)}(M_1^\pm\boxtimes M_2^\pm).
\end{align*}
\end{thmB}
\noindent
See Theorem \ref{isom for intertwining operators} and Corollary \ref{SIcoh_Fusion} for more precise statements. See also Theorem \ref{check of monoidality in the rational case} for an explicit description in the rational case.

\subsection{Outlook}
 
In \cite{CL1, CL2} it has been conjectured that certain $\mathcal W$-(super)algebras whose levels are related by a relation of type \eqref{FF} are connected via tensoring with a certain kernel VOA and then taking a relative semi-infinite cohomology, see Sections 10 of \cite{CL1} and Section 1.2 of \cite{CL2}.
We have proven that in the simplest case, that is when the kernel VOA is a lattice vertex algebra times a Heisenberg vertex algebra, this indeed works. Moreover we have seen that this approach is very efficient in connecting representation categories of the involved $\mathcal W$-(super)algebras. It is thus a major future problem to lift our studies to higher rank cases, this means it is important to settle \cite[Conjecture 1.1 and 1.2]{CL1} and  \cite[Conjecture 1.2]{CL2}.

It is also interesting to study correspondences of correlation functions and conformal blocks and we expect that our relative semi-infinite cohomology approach will be quite useful. Recall that we benefitted from physics suggestions and so it is worth noting that physics suggests us an additional relation. 
The relation between correlation functions of the $H_3^+$-model (a non-compact version of the $\sll_2$ WZW-theory) and $\mathcal{N}=2$ super Liouville theory (the underlying vertex algebra is the $\mathcal{N}=2$ super Virasoro algebra) is called the supersymmetric Fateev--Zamolodchikov--Zamolodchikov duality in physics. It is understood using mirror symmetry \cite{HK} or the path integral \cite{CHR}. The latter method actually relates correlation functions of both theories to correlation functions in Liouville theory (whose underlying vertex algebra is the Virasoro algebra) and uses Liouville theories self-duality, i.e. Feigin--Frenkel duality of the Virasoro algebra.  This method generalizes to higher rank cases \cite{CH} and that generalization is actually inspired from \cite{CGN, CL1}.

Also note that the connection between characters of modules is already quite interesting. For example in the simplest example of $\sll_2$ at admissible level characters are formal distributions or expansions of meromorphic Jacobi forms \cite{CR, Ad5}, while characters of the $\mathcal{N}=2$ super Virasoro algebra are expressed in terms of Jacobi theta functions and mock Jacobi forms \cite{STT}. This correspondence has been studied in \cite{FSST, S3, KoS, CLRW}. From our perspective the character of a module that we obtain via $\Relcoh_{\pm,\epsilon\lambda}$ is  the Euler-Poincar\'e supercharacter of the relative complex as higher cohomologies vanish. 

It is usually a very difficult problem to establish the existence of a vertex tensor category structure on a category of modules of a vertex algebra $V$ that is not $C_2$-cofinite. However, Theorem 3.3.4 of \cite{CY} tells us that if the category of ordinary modules of $V$ is of finite length and $C_1$-cofinite, then it is a vertex tensor category. This Theorem applies for example to $L_k(\mathfrak{sl}_2)$ at admissible level \cite{CHY}, but almost all modules in this case are not ordinary and in fact are often not even lower bounded. The dual $\mathcal W$-superalgebra is the $\mathcal{N}=2$ super Virasoro algebra, and modules are always ordinary \cite{Sato1, CLRW}. It is a reasonable hope that they are in fact all $C_1$-cofinite and of finite length. 
Note that the duality between subregular $\mathcal W$-algebras of $\mathfrak{sl}_n$ and the principal $\mathcal W$-superalgebras of $\mathfrak{sl}_{n|1}$ somehow extends to $n=1$, namely an analogous relation between the $\beta\gamma$-system and the affine vertex superalgebra of $\mathfrak{gl}_{1|1}$ at non-degenerate level. For the latter, it is indeed true that the category of ordinary modules is $C_1$-cofinite and of finite length \cite{CMY3} and since the $\beta\gamma$-system times a pair of free fermions is a simple current extension of the affine vertex superalgebra of $\mathfrak{gl}_{1|1}$ at non degenerate level \cite{CR2} one obtains the category of $\beta\gamma$-modules  that includes relaxed-highest-weight modules. In fact, $\beta\gamma$ is the first example where one has a general vertex tensor category result including all relaxed-highest weight modules \cite{AW}, this result can thus be reproduced, but also $L_k(\sll_{2|1})$ for $k=1$ and $k=-1/2$, see \cite[Section 5]{CMY3}.

\subsection{Organization}
This paper is organized as follows. In Section \ref{Fusion_Lattice_Cosets} we first review and explore general results about the monoidal structure of categories of VOA modules and their super extensions, following \cite{CKM}. Based on purely categorical treatments of simple current extensions in Appendix \ref{Sec. BTC and Simple currents}, we prove the ring isomorphisms \eqref{duality by lattice} in Section \ref{Simple current extensions by Lattice}. After we review of the Feigin--Semikhatov duality \cite{CGN} in Section \ref{Feigin_Semikhatov_Duality}, we prove, in Section \ref{SuperW_rational}, Main Theorem 1 by using the ring isomorphisms \eqref{duality by lattice} and the level-rank duality discussed in Appendix \ref{Level_Rank_Duality}.
The weight-wise linear equivalences between module categories of the subregular $\W$-algebra and the principal $\W$-superalgebra are proved in Section \ref{sec:Kazama-Suzuki}.
Lastly, in Section 6, we introduce the relative semi-infinite cohomology functor and prove Main Theorem 2. We note that the purely categorical perspective of simple current extensions in Appendix \ref{Sec. BTC and Simple currents} is of independent interest.

\subsection*{Acknowledgements}

T.C. is supported by NSERC $\#$RES0048511.
N.G is supported by JSPS Overseas Research Fellowships Grant Number 112035.
S.N. is supported by the Program for Leading Graduate Schools, MEXT, Japan and by JSPS KAKENHI Grant Number 20J10147. 
This work was supported by World Premier International Research Center Initiative (WPI Initiative), MEXT, Japan.

\section{Fusion rules of lattice cosets}\label{Fusion_Lattice_Cosets}

\subsection{Vertex superalgebras and their modules}
We recall basics on vertex superalgebras and their modules, following \cite{HLZ1, CKM}. 
Recall that a vector superspace $V$ (over $\C$) is a $\Z_2$-graded vector space $V=V_{\bar{0}}\oplus V_{\bar{1}}$ where $\Z_2 = \{\bar{0}, \bar{1}\}$ and that an element $a \in V_{\bar{0}}$ (resp.\ $a \in V_{\bar{1}}$) is all even (resp.\ odd) or of parity $\overline{a}=\bar{0}$ (resp.\ $\overline{a}=\bar{1}$). 
The $\Z_2$-grading on $V$ induces a vector superspace structure on the space of $\C$-linear endomorphisms $\End V$ and thus $\End V$ is naturally a Lie superalgebra by $[a,b]=a b-(-1)^{\overline{\mathstrut a}\cdot\overline{\mathstrut b}}b a$ for parity homogeneous elements $a, b \in \End V$.

A vertex superalgebra is a vector superspace $V$ equipped with a non-zero vector $\mathbf{1}\in V_{\bar{0}}$, a linear map $\partial \in (\End V)_{\bar{0}}$ and a parity-preserving linear map
\begin{align*}
Y(\cdot,z)\colon V\rightarrow \End V [\![z,z^{-1}]\!],\quad a\mapsto Y(a,z)=a(z)=\sum_{n\in \Z}a_{(n)}z^{-n-1}
\end{align*}
such that $Y(a,z)b\in V(\!(z)\!)$ ($a,b\in V$) satisfying 
\begin{enumerate}
    \item $Y(\mathbf{1},z)a=a$ and $Y(a,z)\mathbf{1}\in a+V[\![z]\!]z$ for $a \in V$,
    \item $\partial\mathbf{1} = 0$ and $[\partial, Y(a, z)] = \partial_z Y(a, z)$ for $a \in V$,
    \item for $a, b \in V$, there exists $m \in \Z_{\geq 0 }$ such that $(z-w)^m[Y(a,z),Y(b,w)]=0$.
\end{enumerate}
It follows from the axioms that $[Y(a,z),Y(b,w)]=\sum_{n=0}^\infty Y(a_{(n)}b,w)\partial_w^n\delta(z-w)/n!$ where $\delta(z-w)=\sum_{n\in\Z}z^n w^{-n-1}$, which we write as 
\begin{align*}
Y(a,z)Y(b,w)\sim \sum_{n=0}^\infty \frac{Y(a_{(n)}b,w)}{(z-w)^{n+1}},
\end{align*}
called the OPE of $Y(a,z)$ and $Y(b,w)$. If $V = V_{\bar{0}}$, then it is called a vertex algebra.
An even element $\omega\in V_{\bar{0}}$ is called a conformal vector if the corresponding field $Y(\omega,z)=L(z)= \sum_{n\in \Z}L_n z^{-n-2}$ satisfies that $L_{-1} = \partial$, $L_0$ acts on $V$ semisimply, and
\begin{align*}
L(z)L(w) \sim \frac{\partial_w L(w)}{z-w} + \frac{2L(w)}{(z-w)^2} + \frac{c/2}{(z-w)^4}
\end{align*}
for some $c \in \C$, called the central charge of $\omega$. The $L_0$-grading $V=\oplus_{\Delta\in\C} V_\Delta$ is called the conformal grading and in this paper we consider the case $V$ is at most $\frac{1}{2}\Z$-graded.
A vertex operator superalgebra of CFT type is a pair $(V,\omega)$ such that $V$ is $\frac{1}{2}\Z_{\geq0}$-graded by $L_0$ and $V_0=\C \mathbf{1}$ and $\dim V_\Delta < \infty$ for $\Delta\in\frac{1}{2}\Z_{\geq0}$.

In the following, we will consider a negative-definite lattice vertex superalgebra e.g.\ $V_{\ssqrt{-1}\Z}$, which does not satisfy this condition, see \S \ref{lattice VOA} below. This motivates us to consider an additional grading, following \cite{HLZ1,Yan}: given a finitely generated abelian group $A$, we consider an additional $A$-grading on a vertex superalgebra $V=\oplus_{a\in A}V^a$ as a vector superspace such that 
$$Y(\cdot,z)\colon V^a\times V^b\rightarrow V^{a+b}(\!(z)\!),\quad (a,b\in A).$$
Then by a strongly $A$-graded vertex operator superalgebra of CFT type we mean a $\frac{1}{2}\Z$-graded vertex superalgebra $(V,\omega)$ with an additional $A$-grading 
$$V=\bigoplus_{a\in A} V^a=\bigoplus_{\Delta\in\frac{1}{2}\Z,a\in A} V_{\Delta}^a,\quad V_{\Delta}^a=V^a\cap V_{\Delta}$$
as a vector superspace, satisfying
\begin{enumerate}
    \item $V^a_\Delta$ is finite dimensional for all $(a,\Delta)$ and $V_0^0=\C\mathbf{1}$,
    \item $V_{\Delta}^0=0$ for $\Delta<0$ and $V_{\Delta+m}^a=0$ for sufficiently negative $m$ for each ($\Delta,a)$.
\end{enumerate}
Note that both of a lattice vertex superalgebra $V_L$ associated with a non-degenerate integral lattice $L$ with a natural strong $L$-grading and a $\frac{1}{2}\Z_{\geq0}$-graded vertex operator superalgebra with trivial strong $A$-grading satisfy this condition. In the following, we will consider modules for $V$ with a strong $A$-grading as above.

A weak $V$-module is a vector superspace $M$ equipped with a parity-preserving linear map
\begin{align*}
Y_M(\cdot,z) \colon V \rightarrow \End M [\![z, z^{-1}]\!],\quad
a \mapsto Y_M(a, z) = a^M(z) = \sum_{n \in \Z} a_{(n)}^M z^{-n-1}
\end{align*}
such that $Y_M(\cdot,z)\colon V\times M\rightarrow M(\!(z)\!)$ satisfying
\begin{enumerate}
    \item $Y_M(\mathbf{1},z)a = a$ for all $a\in M$,
    \item for all $a, b \in V$,
    \begin{align*}
&z_{0}^{-1}\delta\left(\frac{z_{1}-z_{2}}{z_{0}}\right)a^M(z_{1})b^M(z_{2})\\
&-(-1)^{\overline{\mathstrut a}\cdot\overline{\mathstrut b}}z_{0}^{-1}\delta\left(\frac{z_{2}-z_{1}}{-z_{0}}\right)b^M(z_{2})a^M(z_{1})=z_{2}^{-1}\delta\left(\frac{z_{1}-z_{0}}{z_{2}}\right)Y_{M}(a(z_{0})b,z_{1}),
\end{align*}
\end{enumerate}
where $\delta(z)=\sum_{n\in \Z}z^{n}$.
We note that the axioms imply $\partial_z Y_M(a, z) = Y_M( \partial a, z)$ for $a\in V$. We call a weak $V$-module $A$-gradable if it admits a grading as a vector superspace by an $A$-torsor $\widetilde{A}$, denoted by $M = \oplus_{a\in \widetilde{A}}M^a$, and $Y_M(\cdot,z)$ restricts to $V^a\times M^b\rightarrow M^{a+b}(\!(z)\!)$.
For $A$-gradable weak $V$-modules $M$ and $N$, a morphism from $M$ to $N$ is a linear map $f \colon M \rightarrow N$ such that
\begin{align}\label{eq: V-hom def}
f\left(a^{M}(z)v\right) = a^{N}(z)f(v),\quad(a \in V, v \in M)
\end{align}
which induces a morphism of $A$-torsors on gradings of $M$ and $N$.
In the following, the term $A$-gradable is dropped if $A$ is clear from context (especially, if $A$ is trivial).

Next, we recall several classes of modules. An $A$-gradable weak $V$-module $M$ is called an ($A$-gradable) generalized $V$-module if it admits a generalized $L_0^M$-eigenspace decomposition 
\begin{align}\label{decomposing module}
M = \bigoplus_{\Delta \in \C} M_\Delta=\bigoplus_{\Delta \in \C,a\in \widetilde{A}} M^a_\Delta, \quad M_\Delta = \left\{ v \in M \mid \exists N\geq0, (L_0-\Delta)^N v =0\right\}
\end{align}
A generalized $V$-module $M$ is called \emph{grading-restricted} if $\dim M^a_{\Delta} < \infty$ and $M^a_{\Delta + N}=0$ for $a\in \widetilde{A},\ \Delta \in \C$ and sufficiently negative $N \in \Z$;
\emph{ordinary} if $M$ is grading-restricted and \eqref{decomposing module} is an $L_0^M$-eigenspace decomposition.

Let $V\Mod_\C$ be the category of grading-restricted generalized $V$-modules.
It is naturally a supercategory \cite{BrE}, that is the Hom spaces are vector superspaces and the compositions of morphisms are compatible with the superstructure. (The symbol $\C$ indicated that the $L_0$-eigenvalues of objects are not restricted to $\mathbb{R}\subset \C$ whereas the later restriction is needed for monoidal structure as we will see in the next subsection.)
Recall that the underlying category $\underline{V\Mod}_\C$ is the subcategory of $V\Mod_\C$ whose objects are the same as $V\Mod_\C$ but whose morphisms consist of even homomorphisms. 
Although $\underline{V\Mod}_\C$ is naturally an abelian category, $V\Mod$ is merely a $\C$-linear additive supercategory since the kernel and cokernel objects for parity-inhomogeneous morphisms usually do not exists.
Then by a subquotient object in $V\Mod_\C$ we mean one in $\underline{V\Mod}_\C$ and by the notion of (semi)simplicity as well. 
Note that $V\Mod$ has an involutive autofunctor $\Pi$ which switches the superstructure of objects: $(\Pi M)_{\bar{i}} = M_{\overline{i+1}}$ ($\bar{i}\in \Z_2$). Therefore, we may recover all the properties of parity-homogeneous morphisms by composing $\Pi$ if necessary.

\subsection{Intertwining operators and fusion products}
We recall the notion of (logarithmic) intertwining operators, $P(x)$-tensor products and a braided monoidal structure on module categories, following \cite{HLZ1}-\cite{HLZ8} and \cite{CKM}. 
We restrict the category $V\Mod_\C$ to the full subcategory $V\Mod$ consisting of objects $M$ such that the grading \eqref{decomposing module} for $M$ is supported on $\mathbb{R}$ and the size $N$ of Jordan block for $L_0$ is uniformly bounded with respect to $\Delta\in\mathbb{R}$, following \cite[Assumption 3.9]{CKM}.
We assume that the grading by an $A$-torsor for each object in $V\Mod$ all comes from a grading with respect to a common abelian group containing $A$, or otherwise we restrict to a full subcategory satisfying this property.
For objects $M_i$ ($i=1,2,3$) in $V\Mod$, a parity-homogeneous (logarithmic) intertwining operator of type $\binom{M_3}{M_1 M_2}$ is a parity-homogeneous bilinear map
\begin{equation}\label{def of intertwiner}
\begin{array}{rcl}
\mathcal{Y}(\cdot,z) \colon  M_1 \times M_2 & \rightarrow & M_3\{z\}[\log z] \\
 (m_1, m_2) & \mapsto & \displaystyle \mathcal{Y}(m_1, z)m_2 =\sum_{k =0}^K \sum_{n \in \C} (m_1)_{(n;k)}^\mathcal{Y}m_2 z^{-n-1} (\log z)^k
\end{array}
\end{equation}
such that $\mathcal{Y}(\cdot, z)\colon M_1^a\times M_2^b \rightarrow M_3^{a+b}\{z\}[\log z]$ satisfying 
\begin{enumerate}
\item for all $m_i \in M_i$ ($i=1,2$), $(m_1)_{(n+N;k)}^\mathcal{Y}m_2 = 0$ for sufficiently large $N \in \N$,
\item $\mathcal{Y}(L_{-1}^{M_1}m_1,z)=\partial_z\mathcal{Y}(m_1,z)$,
\item for all homogeneous $a \in V$ and $m_1 \in M_1$.
\begin{align}\label{Jacobi for intertwining op}
\begin{split}
    &(-1)^{\overline{\mathstrut a}\cdot\overline{\mathstrut \mathcal{Y}}}z_0^{-1}\delta\left(\frac{z_1-z_2}{z_0}\right)a^{M_3}(z_1)\mathcal{Y}(m_1,z_2)\\
&-(-1)^{\overline{a}\cdot \overline{m}_1}z_{0}^{-1}\delta\left(\frac{z_2-z_1}{-z_0}\right)\mathcal{Y}(m_1,z_2)a^{M_2}(z_{1})=z_{2}^{-1}\delta\left(\frac{z_{1}-z_{0}}{z_{2}}\right)\mathcal{Y}\left(a^{M_1}(z_{0})m_1,z_{2}\right).
\end{split}
\end{align}
\end{enumerate}
A general intertwining operator of type $\binom{M_3}{M_1 M_2}$ is a sum of parity-homogeneous ones and the vector superspace consisting of them is denoted by  $I_V\binom{M_3}{M_1 M_2}$. 

In \cite{HLZ1}-\cite{HLZ8}, quite a similar notion called a $P(x)$-intertwining map was introduced. Essentially, it is obtained from an intertwining operator by evaluation of $z$ at a complex number in $x\in \C\backslash \{0\}$. To avoid the problem of convergence, in the definition of $P(x)$-intertwining map $I_x(\cdot,\cdot)$, \eqref{def of intertwiner} is replaced by a formal one
$$I_x(\cdot,\cdot)\colon  M_1 \times M_2  \rightarrow \overline{M}_3,$$
where $\overline{M}$ denote the algebraic completion $\overline{M} = \overline{M}_{\bar{0}} \oplus \overline{M}_{\bar{1}}$ with $\overline{M}_{\bar{i}} = \prod_{\Delta \in \C} M_{\Delta,\bar{i}}$.
By \cite[Proposition 3.15]{CKM}, intertwining operators and $P(x)$-intertwining maps are in one-to-one correspondence by evaluation of $z$ at $x$ and extension from the point $z=x$ by the differential equation (2) and we denote by $I_x\binom{M_3}{M_1\ M_2}$ the (super)space of $P(x)$-intertwining map of type $\binom{M_3}{M_1\ M_2}$.
For objects $M_1$ and $M_2$ in $V\Mod$, if the superfunctor from $V\Mod$ to the supercategory of vector superspaces over $\C$ given by $N\mapsto I_x\binom{N}{M_1\ M_2}$ is representable, then the representing object, which is unique up to isomorphisms, is called the $P(x)$-tensor product and denoted by $M_1\boxtimes_{P(x)} M_2$. It admits a canonical $P(x)$-intertwining map $I_x^{M_1,M_2}(\cdot,\cdot)\colon M_1\otimes M_2\rightarrow \overline{M_1\boxtimes_{P(z)} M_2}$ and thus a canonical intertwining operator 
$$\mathcal{Y}_{M_1\boxtimes_{P(x)} M_2}(\cdot,z)\colon M_1\times M_2\rightarrow \overline{M_1\boxtimes_{P(x)} M_2}\{z\}[\log z].$$
Note that the existence of $P(x)$-tensor product for some $x$ implies that for all the other values in $\C\backslash\{0\}$ and they are all isomorphic
$M_1\boxtimes_{P(x)}M_2\simeq M_1\boxtimes_{P(x')}M_2$ by an even morphism \cite[Corollary 3.36]{CKM}.

Now suppose the existence of $P(x)$-tensor product for all objects in $V\Mod$. 
It follows from the universality of representing objects that morphisms $f_i \colon M_i \rightarrow N_i$ ($i=1,2$) induce a unique morphism 
\begin{align*}
f_1 \boxtimes_{P(x)} f_2 \colon M_1 \boxtimes_{P(x)} M_2 \rightarrow N_1 \boxtimes_{P(x)} N_2,
\end{align*}
which gives rise to a superfunctor $\boxtimes_{P(x)}\colon V\Mod\times V\Mod\rightarrow V\Mod$. Moreover, the $P(x)$-intertwining map $M_1\otimes M_2\rightarrow M_2\boxtimes_{P(-x)}M_1$ defined by
$$(m_1,m_2)\mapsto (-1)^{\overline{m}_1\cdot \overline{m}_2}e^{z L_{-1}}\mathcal{Y}_{M_1\boxtimes_{P(x)} M_2}(m_1,z)m_2|_{z=e^{\log x+\pi\ssqrt{-1}}}$$
induces a unique morphism $\mathcal{R}_{P(x);M_1,M_2}^+\colon M_1\boxtimes_{P(x)}M_2\rightarrow M_2\boxtimes_{P(-x)}M_1$.
Then under several assumptions on convergence properties for iterated $P(x)$-tensor products among several values $x\in \C\backslash\{0\}$, the category $V\Mod$ equipped with $\boxtimes_{P(x)}$ and $\mathcal{R}^+$ has the structure of vertex tensor supercategory. It is first established in the even setting by Huang, Lepowsky and Zhang \cite{HLZ1}-\cite{HLZ8} and in the super setting by McRae, Kanade and one of the author \cite{CKM}. See \cite[\S 3.3]{CKM} for a detailed review on precise construction and the required assumptions. 

As a consequence, $V\Mod$ has a structure of additive braided monoidal supercategory with fusion (tensor) product $\boxtimes=\boxtimes_{P(1)}$ and braiding $\mathcal{R}_{M_1,M_2}\colon M_1\boxtimes M_2\mapsto M_2\boxtimes_{P(-1)}M_1\xrightarrow{\simeq} M_2\boxtimes M_1$, see Appendix A for more details on braided monoidal supercategory. We will denote the canonical intertwining operator as $\mathcal{Y}_{M_1\boxtimes M_2}(\cdot,z)=\mathcal{Y}_{M_1\boxtimes_{P(x)} M_2}(\cdot,z)|_{x=1}$.
Sufficient conditions for the required assumptions are obtained by Huang \cite[Theorem 4.13]{Hua2}, when $V=V_{\bar{0}}$ and $\Z_{\geq0}$-graded with trivial strong $A$-grading, and $V\Mod$ is restricted to the usual module category in this setting, that is, the full subcategory $V\Mod_{\bar{0}}\subset V\Mod$ consisting of objects $M$ such that $M=M_{\bar{0}}$. In particular, when $V=V_{\bar{0}}$, the $C_2$-cofiniteness condition, i.e., $\dim V/V_{(-2)}V<\infty$ is sufficient by \cite[Proposition 4.1, Theorem 4.13]{Hua2}.
In this paper, we consider the case for $V\Mod$ with $\frac{1}{2}\Z_{\geq0}$-grading but trivial strong $A$-grading or with lattice vertex superalgebras associated with non-degenerate integral lattices. The first one is an immediate consequence of \cite[Theorem 3.65]{CKM} and \cite[Proposition 4.1, Theorem 4.13]{Hua2}: 
\begin{proposition}\label{prop: V-Mod property}
Let $V$ be a $C_2$-cofinite, $\frac{1}{2}\Z_{\geq0}$-graded vertex operator superalgebra of CFT type. Then
$V\Mod$ is a $\C$-linear additive braided monoidal supercategory satisfying the following property (P): 
all Hom spaces are finite dimensional, all object have finite length, and the fusion product $\boxtimes$ is right exact.
\end{proposition}
\begin{proof}
We start with the subalgebra $V^\Z_{\bar{0}}=\bigoplus_{\Delta\in \Z_{\geq0}}V_{\Delta,\bar{0}}$,
which is a vertex operator algebra of CFT type characterized as the fixed-point by finite order automorphisms $\theta_V=e^{2\pi i L_0}$ and $P_V=\id_{V_{\bar{0}}}-\id_{V_{\bar{1}}}$.
Then $V^\Z_{\bar{0}}$ is also $C_2$-cofinite by \cite{Mi} and thus $V^\Z_{\bar{0}}\Mod_{\bar{0}}$ has a structure of $\C$-linear braided tensor category satisfying (P) by \cite[Theorem 3.24, 4.13]{Hua2} (moreover, it is finite as abelian category).
Note that $V$ itself is an object in $V_{\bar{0}}^\Z\Mod$ and that $V_{\bar{0}}^\Z\Mod$ is equivalent to the superization of $V^\Z_{\bar{0}}\Mod_{\bar{0}}$ as $\C$-linear additive braided monoidal supercategory, see Remark \ref{Superalgebra Object} for the definition of superization.
Since $V$ is an algebra object in $V^\Z_{\bar{0}}\Mod$, we may consider $V\Mod$ through the category $\Rep^0(V)$ of local $V$-modules in $V_{\bar{0}}^\Z\Mod$, see \S \ref{algebra object} for a review on this subject.
It follows from the construction that $\Rep^0 V$ is a $\C$-linear additive braided monoidal supercategory \cite{KO, CKM} induced by $V_{\bar{0}}^\Z\Mod$ satisfying the same property (P). Moreover, $\Rep^0 V$ is equivalent to $V\Mod$ as $\C$-linear additive supercategories by \cite{HKL, CKL} and in addition is the same as the one given in \cite{HLZ1}--\cite{HLZ8} by \cite[Theorem 3.65]{CKM}.
\end{proof}
In the following, we denote by $\irr(V)$ the complete set of representatives of simple objects in $V\Mod$ and by $\Pic(V)\subset \irr(V)$ the Picard group, i.e., the set consisting of simple currents (equivalently, simple invertible objects). The set $\Pic(V)$ forms a group by tensor product, see Appendix A for their properties. 

\subsection{Lattice vertex superalgebras}\label{lattice VOA}
Let us first fix some notation for Heisenberg vertex algebras and their Fock modules. Given a finite dimensional (even) vector space $\mathfrak{a}$ equipped with a bilinear form $(\cdot|\cdot)$, we denote by $\pi^\mathfrak{a}$ the associated Heisenberg vertex algebra, which is generated by the fields $a(z)$ ($a\in \mathfrak{a}$) with OPE 
$$a(z)b(w)\sim \frac{(a|b)}{(z-w)^2}.$$
When $(\cdot|\cdot)$ is non-degenerate, $\pi^\mathfrak{a}$ has a conformal vector of central charge $\dim \mathfrak{a}$ by the Sugawara construction. 
For $\lambda\in \mathfrak{a}^*$, we denote by $\pi^\mathfrak{a}_\lambda$ the Fock module of $\pi^\mathfrak{a}$ of highest weight $\lambda$. It is generated by a highest weight vector which we denote by $\ket{\lambda}$ satisfying 
$$a_{(n)}\ket{\lambda}=\delta_{n,0}\lambda(a)\ket{\lambda},\quad (a\in \mathfrak{a},\ n\geq0).$$
When $(\cdot|\cdot)$ is non-degenerate, we always identify $\mathfrak{a}$ with $\mathfrak{a}^*$ via $(\cdot|\cdot)$ and also use $\pi^\mathfrak{a}_{a}$ $(a\in \mathfrak{a})$. When $\mathfrak{a}$ is one dimensional and spanned by $A\in \mathfrak{a}$, we also use the notation $\pi^{A}=\pi^{\mathfrak{a}}$ and $\pi^{A}_{\lambda(A)}=\pi^{\mathfrak{a}}_{\lambda}$. 

Let $L$ be a non-degenerate integral lattice of finite rank and $L' = \{ a \in L_\Q:=L \otimes_\Z \Q \mid (a|L) \subset \Z\}$ its dual lattice.
Let $V_L = \bigoplus_{a \in L} \pi^{L_\C}_a$ denote the lattice vertex superalgebra associated to $L$.
It is naturally a strongly $L$-graded, $\frac{1}{2}\Z$-graded vertex operator superalgebra of CFT type where the conformal vector is the one in $\pi^{L_\C}$ by the Sugawara construction. 
We denote by $V_L\Mod$ whose objects are $L'$-graded has a family of ($L'$-graded) ordinary simple  $V_L$-modules $V_{a + L} = \bigoplus_{b \in L} \pi^{L_\C}_{a + b}$ ($a \in L'/L$). We summarize the property of $V_L\Mod$.
\begin{proposition}\label{prop: V_L property} Let $L$ be a non-degenerate integral lattice of finite rank.\\
(1) $V_L\Mod$ is a $\C$-linear finite semisimple braided tensor supercategory.\\
(2) $\irr(V_L)=\Pic(V_L)=\{ V_{a + L}\}_{a \in L'/L}$.\\
(3) The fusion product and the monodromy among $\irr(V_L)$ is given by 
$$V_{a + L}\boxtimes V_{b + L}\simeq V_{a + b + L},\quad \mathcal{M}_{V_{a+L},V_{b+L}}=e^{2 \pi i (a|b)}.$$
(4) $V_L\Mod$ is rigid and the fusion product $\boxtimes$ is bi-exact.
\begin{proof}
From the same proof in \cite[Theorem 3.16]{DLM} when $L$ is positive-definite, it follows that $V_L\Mod$ is semisimple and $\irr(V_L)=\{ V_{a + L}\}_{a \in L'/L}$. Then (1) follows from  \cite[Theorem 7.4]{Yan}. 
Then the assertions (2) and (3) follow from the same argument \cite{LL,Hua1,CKM} when $L$ is positive-definite. 
Then (4) follows from Lemma \ref{lem: exactness of simple currents}.
\end{proof}
\end{proposition}

\subsection{Simple current extensions by lattices}\label{Simple current extensions by Lattice}
Let $V$ be a simple, $C_2$-cofinite, $\frac{1}{2}\Z_{\geq0}$-graded vertex operator superalgebra of CFT type with trivial strong grading and $V_L$ be the lattice vertex superalgebra associated with $L$ be a non-degenerate integral lattice with strong $L$-grading. Here we consider the module category of simple current extensions of $V\otimes V_L$, following \cite{CKM} and Appendix A.

Since $V \otimes V_L$ is a simple, strongly $L$-graded, $C_2$-cofinite, $\frac{1}{2}\Z_{\geq0}$-graded vertex operator superalgebra of CFT type, we have the $\C$-linear additive braided monoidal supercategory $V\otimes V_L\Mod$ by the previous two subsections. It is naturally isomorphic to the Deligne product $(V\Mod)\otimes (V_L\Mod)$ by \cite[Theorem 5.5]{CKM2}:
\begin{align}\label{Deligne product}
\begin{array}{ccc}
V\otimes V_L\Mod&\simeq& (V\Mod)\otimes (V_L\Mod)\\
M&\mapsto& \displaystyle{\bigoplus_{a\in L'/L}} \Omega_{a}(M)\otimes V_{a+L}
\end{array}
\end{align}
where $\Omega_{a}(M):=\{m\in M\mid \forall \lambda\in L,\ n\geq0,\  h_{(n)}m=\delta_{n,0}\lambda(a)m, \}$. The equivalence implies the decomposition of the Picard groups:
\begin{align}\label{decomposition of Picard groupoid}
\Pic(V)\times \Pic(V_L)\simeq \Pic(V\otimes V_L),\quad (M,V_{a+L})\mapsto M\otimes V_{a+L}.
\end{align}
Let $N$ be a sublattice of $L'$ such that $L \subset N \subset L'$. 
We consider an algebra object $\mathcal{E}$ in $V\otimes V_L\Mod$ of the form
\begin{align}\label{lattice extension: easy}
\mathcal{E}=\bigoplus_{a \in N/L}\mathcal{E}_{a},\quad
\mathcal{E}_a = S_{a}\otimes V_{a+L},
\end{align}
where $\{S_a\}_{a\in N/L}$ is a subgroup of $\Pic(V)$ with $S_0 = V$. Then $\mathcal{E}$ is a categorical simple current extension, see \S \ref{Categorical simple current extensions}. By Proposition \ref{prop: V-Mod property} and \ref{prop: V_L property}, $V\otimes V_L\Mod$ satisfies Assumption \ref{assumption1}, \ref{assumption2} and \ref{assumption4}. As in \S \ref{Categorical simple current extensions}, we also suppose that $V\otimes V_L\Mod$ satisfies Assumption \ref{assumption3} and that $\mathcal{E}$ satisfies (S1) in \S \ref{Categorical simple current extensions} and equations
\begin{align*}
\theta_{\mathcal{E}}^2=\id,\quad
\theta_{\mathcal{E}_a}\theta_{\mathcal{E}_b}=\theta_{\mathcal{E}_{a+b}},\quad
a, b \in N/L,
\end{align*}
where $\theta_{M} = \mathrm{e}^{2\pi\sqrt{-1}L_0^M}$.
Then (S2) in \S \ref{Categorical simple current extensions} is automatically satisfied since $\irr(V_L)=\Pic(V_L)$. Thus, by \cite[Theorem 3.42]{CKM}, $\mathcal{E}$ is a simple, $C_2$-cofinite, strongly $L'$-graded, $\frac{1}{2}\Z_{\geq0}$-graded vertex operator superalgebra of CFT type and we can use all the results in Appendix A.
By Theorem \ref{thm: monodromy decomposition}, we have the decomposition
\begin{align*}
V\Mod=\bigoplus_{\phi\in (N/L)^\vee}V\Mod_\phi,
\end{align*}
where $(N/L)^\vee=\Hom_{\operatorname{Grp}}(N/L,\C^\times)$ and $V\Mod_\phi$ is the full subcategory of $V\Mod$ consisting of objects $M$ such that $\mathcal{M}_{S_{a},M}=\phi (a)\id_{S_{a}\boxtimes M}$ for $a\in N/L$. We write $\phi_M$ for $\phi$ if $M \in V\Mod_\phi$.
Similarly, 
\begin{align*}
V_L\Mod=\bigoplus_{\phi\in (N/L)^\vee}V_L\Mod_\phi.
\end{align*}
Since $V_L\Mod_\phi$ is semisimple, it follows that
\begin{align*}
\irr(V_L\Mod_\phi)=\left\{V_{a+L}\middle|\,b\in L'/L,\ e^{2\pi \sqrt{-1}(a|b)}=\phi(a)\ \mathrm{for}\ \mathrm{all}\ b\in N/L\right\}.
\end{align*}
The group homomorphism
\begin{align*}
\gamma \colon L'\rightarrow (N/L)^\vee,\quad
a \mapsto \left(\gamma_a \colon N/L \ni b \mapsto \gamma_a(b) = e^{2\pi \sqrt{-1}(a|b)} \in \C^\times \right).
\end{align*}
induces an isomorphism $L'/N' \simeq (N/L)^\vee$ and then $\irr(V_L\Mod_\phi)=\{V_{a+L}\mid a \in L'/L,\,\gamma_a = \phi\}$. Let $(V\otimes V_L \Mod)^0$ be the full subcategory of $V\otimes V_L \Mod$ whose objects $M$ satisfy that $\mathcal{M}_{\mathcal{E}, M} = \id_{\mathcal{E} \boxtimes M}$, or equivalently, that $M$ is local for $\mathcal{E}$. Then it follows from \eqref{Deligne product} that
\begin{align}
(V\otimes V_L \Mod)^0\simeq \bigoplus_{\phi\in (N/L)^\vee} (V\Mod)_\phi\otimes (V_L\Mod)_{\phi^{-1}}.
\end{align}
The following is a slight generalization to the irrational setting of \cite[Theorem 4.39, Theorem 4.41]{CKM} and of \cite[Theorem 3.7, Theorem 3.13]{YY}.
\begin{theorem}[cf.~\cite{CKM,YY}]\label{extension law}
Let $V,V_L$ and $\mathcal{E}$ as above.\\
(i) The complete set of isomorphism classes of simple objects in $\Rep^0(\mathcal{E})\simeq \mathcal{E}\Mod$ is 
\begin{align*}
\irr(\mathcal{E})\simeq \{(M,a)\in \irr(V)\times (L'/L)\mid\phi_M\phi_{V_{a+L}}=1\}\big/(N/L)
\end{align*}
by $(M,a)\mapsto\mathcal{F}(M\otimes V_{a+L})=\mathcal{E}\boxtimes_{V\otimes V_L}(M\otimes V_{a+L})$.
In particular, we have $|\irr(\mathcal{E})|=|\irr(V)|\cdot |N'/L|/|N/L|$ and 
$$\Pic(\mathcal{E})\simeq \{(M,a)\in \Pic(V)\times (L'/L)\mid\phi_M\phi_{V_{a+L}}=1\}\big/(N/L).$$\\
(ii) We have a ring isomorphism 
\begin{align}\label{isom1}
\K(\mathcal{E})\simeq \left(\mathcal{K}(V)\underset{\Z[N/L]}{\otimes}\Z[L'/L]\right)^{N/L}
\end{align}
where the tensor product over $\Z[N/L]$ is given by $[M\boxtimes S_a]\otimes b=[M]\otimes (-a+b)$ for $a\in N/L$.
\end{theorem}
\proof 
(i) is immediate from Corollary \ref{classification of irreducibles}. For (ii), note that 
$\mathcal{K}(V_L)\simeq \Z[L'/L]$, ($[V_{a+L}]\mapsto a$).
Since $\mathcal{K}(\mathcal{E})$ and $\mathcal{K}(V_L)$ are $(N/L)^\vee$-graded $\Z[N/L]$-algebras by Theorem \ref{fusion algebra}, the fusion algebra $\mathcal{K}\left((V\otimes V_L\Mod)^0\right)$ is a diagonal $N/L$-invariant subalgebra
\begin{align*}
\mathcal{K}\left((V\otimes V_L\Mod)^0\right)&\simeq \left(\mathcal{K}(V)\otimes\mathcal{K}(V_L)\right)^{N/L}\\
&\simeq\left(\mathcal{K}(V)\otimes\Z[L'/L]\right)^{N/L}.
\end{align*}
Hence, by Corollary \ref{Grothendieck groups for induction} the induction functor $\mathcal{E}\boxtimes_{V\otimes V_L}\cdot$ induces an isomorphism 
\begin{align*}
\mathcal{K}(\mathcal{E})\simeq \mathcal{K}\left((V\otimes V_L\Mod)^0\right)/\mathcal{J}
\end{align*}
where $\mathcal{J}$ is generated by 
$[M]\otimes b-[M\boxtimes S_a]\otimes (a+b)$ for an object $M$ in $V\Mod$, $a\in N/L$ and $b\in L'/L$. This implies \eqref{isom1}.
\endproof

Next, we consider an inverse statement of Theorem \ref{extension law}. Recall that $\Rep(\mathcal{E})$ is a $\C$-linear monoidal supercategory and the induction functor
\begin{align*}
\mathcal{F}(\cdot)=\mathcal{E}\boxtimes_{V\otimes V_L}\!(\cdot)\colon (V\Mod)\otimes (V_L\Mod)\rightarrow \Rep(\mathcal{E})
\end{align*}
is a $\C$-linear monoidal superfunctor. By (S2), the functors
\begin{align*}
&V\Mod\rightarrow \Rep(\mathcal{E}),\quad M\mapsto \mathcal{F}(M\otimes V_L),\\
&V_L\Mod\rightarrow \Rep(\mathcal{E}),\quad M\mapsto \mathcal{F}(V\otimes M),
\end{align*}
are embeddings so that we may consider $V\Mod$ and $V_L\Mod$ are $\C$-linear monoidal subcategories of $\Rep(\mathcal{E})$. From now on, we will write $V_{a+L}$ instead of $\mathcal{F}(V\boxtimes V_{a+L})$ as an object in $\Rep(\mathcal{E})$ for $a \in L'/L$ by abuse of notation.
\begin{lemma}\label{decomposition of simple objects}
Every simple object $M$ in $\Rep(\mathcal{E})$ is isomorphic to $\widetilde{M}\boxtimes_{\mathcal{E}}V_{a+L}$ for some $\widetilde{M}\in \irr(\mathcal{E})$ and $a\in L'/L$. Moreover, $N'/L$ acts simply transitively on the set $\{(\widetilde{M}, a) \in \irr(\mathcal{E}) \times L'/L \mid \widetilde{M}\boxtimes_{\mathcal{E}}V_{a+L} \in \irr(\Rep(\mathcal{E})) \}$ by
\begin{align*}
b \cdot (\widetilde{M},a)=(\widetilde{M}\boxtimes_{\mathcal{E}}V_{b+L},a-b),\quad
b \in N'/L.
\end{align*}
\end{lemma}
\proof
By Proposition \ref{induction} (ii), we have $M\simeq \mathcal{F}(\overline{M}\otimes V_{a+L})$ for some $\overline{M}\in \irr(V)$ and $a\in L'/L$. Since $\mathcal{F}$ is a monoidal superfunctor, we may decompose 
\begin{align}\label{decomposition of irreducibles}
\mathcal{F}(\overline{M}\otimes V_{a+L})\simeq \mathcal{F}(\overline{M}\otimes V_{b+L})\boxtimes_{\mathcal{E}}V_{a-b+L}
\end{align}
for any $b \in L'/L$. By using the isomorphism $\gamma \colon L'/N'\xrightarrow{\simeq}(N/L)^\vee$, we may take $b \in L'/L$ such that $\widetilde{M} \otimes V_{b+L}$ is a local $V\otimes V_L$-module with respect to $\{S_c\otimes V_{c+L}\}_{c\in N/L}$. Hence, $\mathcal{F}(\overline{M}\otimes V_{b+L})\in \irr(\mathcal{E})$. This proves the first part of the statement. The element $b\in L'/L$ in the decomposition \eqref{decomposition of irreducibles} such that $\overline{M}\otimes V_{b+L}$ is local as above are uniquely determined up to $N'/L$. This implies  the second part of the statement.
\endproof
Since objects of $\Rep(\mathcal{E})$ consist of pairs $(M, \mu_M)$ with $M\in (V\Mod)\otimes (V_L\Mod)$ and a morphism $\mu_M\colon \mathcal{E}\boxtimes_{V\otimes V_L} M\rightarrow M$ in $(V\Mod)\otimes (V_L\Mod)$, we have two families of monodromy actions 
\begin{align*}
\mathcal{M}_{N,\bullet}=\mathcal{M}_{N\otimes V_L,\bullet},\quad N\in V\Mod, \quad
\mathcal{M}_{V_{a+L},\bullet}=\mathcal{M}_{V\otimes V_{a+L},\bullet},\ a\in L'/L.
\end{align*}
Clearly, any object $M$ in the subcategory $V\Mod$ of $\Rep(\mathcal{E})$ has trivial monodromy with $V_{a+L}$, i.e., $\mathcal{M}_{V_{a+L}, \mathcal{F}(M\otimes V_L)}=\id$. Conversely, any simple object in $\Rep(\mathcal{E})$ satisfying this monodromy-free property lies in $V\Mod$. This implies the following slight generalization in the  irrational setting of \cite{CKM,YY} as Theorem \ref{extension law}.

\begin{theorem}[cf.~\cite{CKM,YY}]\label{Inverting_SCE}
Let $V,V_L$ and $\mathcal{E}$ as above.\\
(i) The set $\irr(V)$ has one-to-one correspondence to
$$\{(M,a)\in \irr(\mathcal{E})\times (L'/L)\mid\mathcal{M}_{V_{b+L},M\boxtimes_\mathcal{E}V_{a+L}}=\id,\ (\forall b\in N'/L)\}/(N'/L)$$
by $(M,a)\mapsto \overline{M}$ such that $\mathcal{F}(\overline{M}\otimes V_L)\simeq M\boxtimes_{\mathcal{E}} V_{a+L}$. In particular, we have $|\irr(V)|=|\irr(\mathcal{E})| \cdot|N/L|/|N'/L|$ and $\Pic(V)$ is naturally isomorphic to 
\begin{align*}
\{(M,a)\in \Pic(\mathcal{E})\times (L'/L)\mid\mathcal{M}_{V_{b+L},M\boxtimes_\mathcal{E}V_{a+L}}=\id,\ (\forall b\in N'/L)\}/(N'/L)
\end{align*}
(ii) We have an isomorphism of rings
\begin{align}\label{isom2}
\mathcal{K}(V)\simeq \left(\mathcal{K}(\mathcal{E})\underset{\Z[N'/L]}{\otimes}\Z[L'/L]\right)^{N'/L}
\end{align}
where the tensor product over $\Z[N'/L]$ is given by $[M\boxtimes_\mathcal{E}V_{a+L}]\otimes b=[M]\otimes (a+b)$ for $a\in N'/L$.
\end{theorem}
\proof
(i) is immediate from Lemma \ref{decomposition of simple objects}. We show (ii). Since every object in $V\Mod\otimes V_L\Mod$ has finite length, so does every object in $\Rep(\mathcal{E})$ and $\mathcal{E}\Mod$. Thus, we may take bases of $\mathcal{K}(\Rep(\mathcal{E}))$ and $\mathcal{K}(\mathcal{E})$ by complete sets of simple objects. By Lemma \ref{decomposition of simple objects}, we have a natural isomorphism
\begin{align*}
\mathcal{K}(\Rep(\mathcal{E}))\simeq \mathcal{K}(\mathcal{E})\underset{\Z[L'/N]}{\otimes} \mathcal{K}(V_L)\simeq \mathcal{K}(\mathcal{E})\underset{\Z[L'/N]}{\otimes} \Z[L'/L].
\end{align*}
Note that $\mathcal{K}(\Rep(\mathcal{E}))$ is an $(N'/L)^\vee$-graded ring by the monodromy action of $\{V_{a+L}\}_{a\in N'/L}$. Since the $(N'/L)$-invariant subring is spanned by $\irr(V)$, we obtain the assertion.
\endproof

\section{Feigin--Semikhatov Duality}\label{Feigin_Semikhatov_Duality}

In this section we review the Feigin--Semikhatov duality \cite{CGN,CL1} between the subregular $\W$-algebra and the principal $\W$-superalgebra conjectured by Feigin and Semikhatov in \cite{FS}.

\subsection{Affine vertex superalgebras and $\W$-superalgebras}\label{Affine_Notation}
Let $\mathfrak{g}$ be a finite dimensional Lie superalgebra equipped with a non-degenerate even supersymmetric invariant bilinear form $(\cdot|\cdot)$ and $\mathfrak{h}$ be its Cartan subalgebra. When $\mathfrak{g}$ is simple, we always normalize the form $(\cdot|\cdot)$ so that the highest even root of $\mathfrak{g}$ has square length $2$.
Let $\widehat{\mathfrak{g}} = \mathfrak{g}[t, t^{-1}] \oplus \C K$ be the affinization of $\mathfrak{g}$ defined by
\begin{align*}
[at^m, bt^n] = [a, b]t^{m+n}+m(a|b)\delta_{m+n,0}K,\quad
[K, at^m] = 0
\end{align*}
for $a, b\in\mathfrak{g}$ and $m, n\in\Z$. The induced $\widehat{\mathfrak{g}}$-module
\begin{align*}
V^k(\mathfrak{g})= U(\widehat{\mathfrak{g}})\underset{U(\mathfrak{g}[t]\oplus\C K)}{\otimes}\C_{k},
\end{align*}
where $\C_{k}$ is the $(1|0)$-dimensional $\mathfrak{g}[t]\oplus\C K$-module by $\mathfrak{g}[t] =0$ and $K = k\id$,
has a vertex superalgebra structure and is called the universal affine vertex superalgebra associated to $\mathfrak{g}$ at level $k$.
It is $\Z_{\geq0}$-graded conformal if $k+h^\vee \neq 0$ by the Sugawara construction, where $h^\vee$ is the dual Coxeter number of $\mathfrak{g}$. 
Similarly, for a dominant weight $\lambda \in P_+$ of $\g$, let $E_{\lambda, k}$ be the irreducible highest-weight representation of $\g \oplus \C K$ of highest-weight $\lambda$ on which $K$ acts by $k\id$. This lifts to a $\mathfrak{g}[t]\oplus\C K$-module by $\mathfrak{g}[t] =0$ and we obtain the induced $V^k(\g)$-module
\begin{align*}
V^k(\lambda)= U(\widehat{\mathfrak{g}})\underset{U(\mathfrak{g}[t]\oplus\C K)}{\otimes} E_{\lambda, k}.
\end{align*}

Suppose that $\mathfrak{g}$ is simple and take an even nilpotent element $f$ in $\mathfrak{g}_{\bar{0}}$ together with a good $\frac{1}{2}\Z$-grading $\Gamma$ of $\mathfrak{g}$ adapted to $f$ (see \cite{KRW} for the definition).
In \cite{KRW}, the universal $\W$-algebra $\mathcal{W}^k(\mathfrak{g},f;\Gamma)$ is defined by the generalized Drinfeld-Sokolov reduction of $V^k(\mathfrak{g})$ associated to $(\mathfrak{g}, f, \Gamma)$. 
When $k+h^\vee \neq 0$, it has a standard conformal vector $\omega_\Gamma$ and then is a $\frac{1}{2}\Z_{\geq0}$-graded vertex operator superalgebra of CFT type. Once we fix $\Gamma$, we abbreviate $\mathcal{W}^k(\mathfrak{g},f;\Gamma)$ as $\mathcal{W}^k(\mathfrak{g},f)$ and denote by $\W_k(\g,f)$ its unique simple quotient. In the next two subsection, we consider two $\W$-superalgebras which we consider in this paper.

\subsection{Subregular $\mathcal{W}$-algebras}\label{def of subreg}
Let $\mathfrak{g} = \mathfrak{sl}_n$, $f=\sub$ be a subregular nilpotent element in $\mathfrak{sl}_n$ and $\Gamma$  the Dynkin grading corresponding to the weighted Dynkin diagram
\begin{align*}
\quad\\
\setlength{\unitlength}{1mm}
\begin{picture}(0,0)(20,10)
\put(-38,9){$\mathfrak{g} = \mathfrak{sl}_{2m}:$}
\put(0,10){\circle{2}}
\put(-1,13){\footnotesize$1$}
\put(-1,5){\footnotesize$\alpha_1$}
\put(1,10.3){\line(1,0){4}}
\put(6.5,9.4){$\cdot$}
\put(8,9.4){$\cdot$}
\put(9.5,9.4){$\cdot$}
\put(12,10.3){\line(1,0){4}}
\put(17,10){\circle{2}}
\put(16,13){\footnotesize$1$}
\put(14.5,5){\footnotesize$\alpha_{m-1}$}
\put(18,10.3){\line(1,0){8}}
\put(27,10){\circle{2}}
\put(26,13){\footnotesize$0$}
\put(26,5){\footnotesize$\alpha_{m}$}
\put(28,10.3){\line(1,0){8}}
\put(37,10){\circle{2}}
\put(36,13){\footnotesize$1$}
\put(34,5){\footnotesize$\alpha_{m+1}$}
\put(38,10.3){\line(1,0){4}}
\put(43.5,9.4){$\cdot$}
\put(45,9.4){$\cdot$}
\put(46.5,9.4){$\cdot$}
\put(49,10.3){\line(1,0){4}}
\put(54,10){\circle{2}}
\put(53,13){\footnotesize$1$}
\put(53,5){\footnotesize$\alpha_{2m-1}$}
\put(56,9){,}
\end{picture}\\
\end{align*}
\begin{align*}
\quad\\
\setlength{\unitlength}{1mm}
\begin{picture}(0,0)(20,10)
\put(-38,9){$\mathfrak{g} = \mathfrak{sl}_{2m+1}:$}
\put(0,10){\circle{2}}
\put(-1,13){\footnotesize$1$}
\put(-1,5){\footnotesize$\alpha_1$}
\put(1,10.3){\line(1,0){4}}
\put(6.5,9.4){$\cdot$}
\put(8,9.4){$\cdot$}
\put(9.5,9.4){$\cdot$}
\put(12,10.3){\line(1,0){4}}
\put(17,10){\circle{2}}
\put(16,13){\footnotesize$1$}
\put(14.5,5){\footnotesize$\alpha_{m-1}$}
\put(18,10.3){\line(1,0){8}}
\put(27,10){\circle{2}}
\put(26,13){\footnotesize$\frac{1}{2}$}
\put(26,5){\footnotesize$\alpha_{m}$}
\put(28,10.3){\line(1,0){8}}
\put(37,10){\circle{2}}
\put(36,13){\footnotesize$\frac{1}{2}$}
\put(34,5){\footnotesize$\alpha_{m+1}$}
\put(38,10.3){\line(1,0){8}}
\put(47,10){\circle{2}}
\put(46,13){\footnotesize$1$}
\put(44,5){\footnotesize$\alpha_{m+2}$}
\put(48,10.3){\line(1,0){4}}
\put(53.5,9.4){$\cdot$}
\put(55,9.4){$\cdot$}
\put(56.5,9.4){$\cdot$}
\put(59,10.3){\line(1,0){4}}
\put(64,10){\circle{2}}
\put(63,13){\footnotesize$1$}
\put(63,5){\footnotesize$\alpha_{2m}$}
\end{picture}\\
\end{align*}
for $m\in\Z_{\geq1}$.
Then the associated $\W$-algebra $\mathcal{W}^k(\mathfrak{sl}_n,\sub)$ is called the subregular $\W$-algebra and 
$\Z_{\geq0}$-graded (resp.~$\frac{1}{2}\Z_{\geq0}$-graded) with respect to $\omega_{\mathrm{sub}}=\omega_\Gamma$ if $n=2m$ (resp.~$n=2m+1$).
As in \cite[\S4.1]{CGN}\footnote{In \cite{CGN}, another good grading $\Gamma_o$ corresponding to
\begin{align*}
\quad\\
\setlength{\unitlength}{1mm}
\begin{picture}(0,0)(20,10)
\put(-38,9){$\mathfrak{g} = \mathfrak{sl}_n:$}
\put(0,10){\circle{2}}
\put(-1,13){\footnotesize$0$}
\put(-1,5){\footnotesize$\alpha_1$}
\put(1,10.3){\line(1,0){8}}
\put(10,10){\circle{2}}
\put(9,13){\footnotesize$1$}
\put(9,5){\footnotesize$\alpha_2$}
\put(11,10.3){\line(1,0){6}}
\put(18.5,9.4){$\cdot$}
\put(20,9.4){$\cdot$}
\put(21.5,9.4){$\cdot$}
\put(24,10.3){\line(1,0){6}}
\put(31,10){\circle{2}}
\put(30,13){\footnotesize$1$}
\put(30,5){\footnotesize$\alpha_{n-1}$}
\put(33,9){,}
\end{picture}\\
\end{align*}
is used. Since $\mathcal{W}^k(\mathfrak{sl}_n,\sub,\Gamma)$ is isomorphic to $\mathcal{W}^k(\mathfrak{sl}_n,\sub;\Gamma_o)$ as a vertex algebra (see \cite{DSK} and \cite[Theorem 3.2.6.1]{AKM}), the element $\Hsub$ corresponds to $H_1$ in \cite[\S4.1]{CGN}.} (see \S\ref{Free_Field} for details), there exists a Heisenberg vertex subalgebra $\pi^{\Hsub}$ of $\mathcal{W}^k(\mathfrak{sl}_n,\sub)$
such that $H^+_{(0)}$ acts diagonally with integer eigenvalues and
\begin{equation}\label{H_plus}
    \Hsub(z)\Hsub(w)\sim\frac{\varepsilon_+}{(z-w)^2},\quad\varepsilon_+:=\frac{n-1}{n}(k+n)-1.
\end{equation}
The subregular $\W$-algebra is strongly $\Z$-graded by $H^+_{(0)}$ and, accordingly, we work with the module category whose objects are strongly $\C$-graded by $H^+_{(0)}$.

\subsection{Principal $\W$-superalgebras}\label{def of sprin}
Let $\mathfrak{g} = \mathfrak{sl}_{1|n}$, $f=\prin$ be a principal nilpotent element in $(\mathfrak{sl}_{1|n})_{\bar{0}}=\mathfrak{gl}_n$ and $\Gamma$ the good $\Z$-grading of $\mathfrak{sl}_{1|n}$ in \cite[\S3.3]{CGN}. By choosing a set $\{\beta_i\}_{i=0}^{n-1}$ of simple roots of $\mathfrak{sl}_{1|n}$ so that $(\beta_i|\beta_{i+1}) = -1$ and $\beta_0$ is a unique odd root, we can express the weighted Dynkin diagram for $\Gamma$ as 
\begin{align*}
\quad\\
\setlength{\unitlength}{1mm}
\begin{picture}(0,0)(20,10)
\put(-38,9){$\mathfrak{g} = \mathfrak{sl}_{1|n}$:}
\put(0,10){\circle{2}}
\put(-1.1,9.3){\footnotesize$\times$}
\put(-1,13){\footnotesize$0$}
\put(-1,5){\footnotesize$\beta_0$}
\put(1,10.3){\line(1,0){8}}
\put(10,10){\circle{2}}
\put(9,13){\footnotesize$1$}
\put(9,5){\footnotesize$\beta_1$}
\put(11,10.3){\line(1,0){6}}
\put(18.5,9.4){$\cdot$}
\put(20,9.4){$\cdot$}
\put(21.5,9.4){$\cdot$}
\put(24,10.3){\line(1,0){6}}
\put(31,10){\circle{2}}
\put(30,13){\footnotesize$1$}
\put(30,5){\footnotesize$\beta_{n-1}$}
\put(33,9){.}
\end{picture}\\
\end{align*}
By \cite[\S4.1]{CGN} (see \S\ref{Free_Field} for details), there exists a Heisenberg vertex subalgebra $\pi^{\Hsup}$ of $\mathcal{W}^k(\mathfrak{sl}_{1|n},\prin)$
such that $H^-_{(0)}$ acts diagonally with integer eigenvalues and
\begin{equation}\label{H_minus}
    \Hsup(z)\Hsup(w)\sim\frac{\varepsilon_-}{(z-w)^2},\quad\varepsilon_-:=-\frac{n}{n-1}(k+n-1)+1.
\end{equation}
In this paper we use $\omega_{\mathrm{sup}}:=\omega_\Gamma-\frac{1}{2}\partial H^-$ so that $\mathcal{W}^k(\mathfrak{sl}_{1|n},\prin)$ is $\frac{1}{2}\Z_{\geq0}$-graded (resp.~$\Z_{\geq0}$-graded) if $n=2m$ (resp.~$n=2m+1$).
Similarly to the subregular case, the principal $\W$-superalgebra is strongly $\Z$-graded by $H^-_{(0)}$ and we work with the module category whose objects are strongly $\C$-graded by $H^-_{(0)}$.

\subsection{Feigin--Semikhatov Duality}
Let $\kappa\in \C\setminus\{0,\frac{n}{n-1}\}$ and set 
\begin{equation}\label{def of Wpm}
    \Wplus:=\W^{\kappa-n}(\sll_{n},\sub),\quad \Wminus:=\W^{\frac{1}{\kappa}-(n-1)}(\sll_{1|n},\prin).
\end{equation}
By \eqref{H_plus} and \eqref{H_minus}, the condition $\kappa\neq\frac{n}{n-1}$ implies $\varepsilon_\pm\neq0$.
Then, by \cite{KRW,KW}, the Heisenberg vertex subalgebra $\pi^{\Hpm}$ acts semisimply on $\Wpm$ and thus we have a decomposition
\begin{equation}\label{support of W}
\Wpm\simeq\bigoplus_{\lambda\in\Z}\Omega^{\Hpm}_\lambda(\Wpm)\otimes\pi^{\Hpm}_{\lambda}
\end{equation}
as a $\Com(\pi^{\Hpm},\Wpm)\otimes\pi^{\Hpm}$-module, where
\begin{equation}
    \Omega_\lambda^{\Hpm}(\Wpm):=\{w\in \Wpm\mid H^\pm_{(n)}w=\delta_{n,0} \lambda w\text{ for }n\geq0\}.
\end{equation}

By using lattice vertex superalgebras $\mathcal{V}_{\pm}=V_{\Z\boson{\pm}}$ with $(\boson{\pm}|\boson{\pm})=\pm1$ and certain free field realizations (see \S\ref{Free_Field} below for details), we obtain the Kazama--Suzuki coset construction and its inverse \cite[Corollary 5.15]{CGN}:
\begin{equation}\label{KS_Duality}
{\sf KS}_\pm\colon\Wmp
\xrightarrow{\simeq}\mathrm{Com}(\pi^{\tHpm},\Wpm\otimes \mathcal{V}^{\pm}),
\end{equation}
where $\tHpm(z):=\mp\Hpm(z)+\boson{\pm}(z)$. We have
\begin{align}
\tHpm(z)\tHpm(w) \sim \frac{\widetilde{\varepsilon}_{\pm}}{(z-w)^2},\quad\widetilde{\varepsilon}_{\pm}:=\varepsilon_{\pm}\pm1=\pm\Big(\frac{n-1}{n}\kappa\Big)^{\pm1}.
\end{align}
The restriction of \eqref{KS_Duality} gives the Feigin--Semikhatov duality \cite{CGN,CL1}:
\begin{equation}\label{FS_duality}
\Com(\pi^{\Hmp},\Wmp)\xrightarrow{\simeq}
\Com(\pi^{\Hpm},\Wpm).
\end{equation}
\subsection{Free field realization revisited}\label{Free_Field}
Here we briefly review the free field realization of $\Wpm$ used in \cite{CGN} to construct the isomorphism \eqref{KS_Duality} for later use.

Let $\heisplus$ denote the Heisenberg vertex algebra generated by fields $\alpha_i(z)$ ($i=1,\dots,n-1$) corresponding to the simple roots of $\mathfrak{sl}_{n}$, which satisfy the OPEs
$$\alpha_i(z)\alpha_j(w)\sim \frac{\kappa(\alpha_i|\alpha_j)}{(z-w)^2}.$$
Then the Miura map for $\Wplus$ \cite{KW} induces an embedding 
\begin{align*}
p_+ \colon \Wplus \hookrightarrow \heisplus \otimes \mathcal{U}^+,
\end{align*}
which sends $\Hsub(z)$ to $\frac{1}{n}\sum_{i=1}^{n-1}(n-i)\alpha_i(z)-\phi_-(z)$. Here $\mathcal{U}^+$ denotes the vertex subalgebra of $V_{\Z\boson{+}\oplus \Z\boson{-}}$ generated by $\{\boson{+}(z),\boson{-}(z),\ket{m(\boson{+}+\boson{-})}(z)\,|\,m\in\Z\}$.
Similarly, let $\heisminus$ denote the Heisenberg vertex algebra generated by fields $\beta_i(z)$ ($i=0,\dots,n-1$) corresponding to the simple roots of $\mathfrak{sl}_{1|n}$ satisfying the OPEs
$$\beta_i(z)\beta_j(w)\sim \frac{\frac{1}{\kappa}(\beta_i|\beta_j)}{(z-w)^2}.$$
Then the Miura map for $\Wminus$ \cite{KW} induces an embedding 
\begin{align*}
p_- \colon\Wminus \hookrightarrow \heisminus \otimes \mathcal{U}^-,
\end{align*}
with $\mathcal{U}^-=V_{\Z\boson{+}}$, which sends $\Hsup(z)$ to $-\frac{1}{n-1}\sum_{i=0}^n(n-i)\beta_i(z)+\boson{+}(z)$.

We introduce vertex superalgebra embeddings $\widetilde{\sf{KS}}_\pm$ by
$$\widetilde{\sf{KS}}_+\colon \heisminus\otimes \mathcal{U}^-\rightarrow\heisplus\otimes \mathcal{U}^+\otimes \mathcal{V}^+,$$
\begin{align*}
    &\textstyle\beta_0(z)\mapsto
    \frac{1}{\kappa}(\boson{-}(z)+ \boson{+}_R(z)), \quad\beta_1(z)\mapsto -\frac{1}{\kappa}\alpha_1(z)+(\boson{+}_L(z)+ \boson{-}(z)),\\
    &\textstyle\beta_i(z)\mapsto -\frac{1}{\kappa}\alpha_i(z),\quad (i=2,\dots,n-1),\\
    &\boson{+}(z)\mapsto  \boson{+}_L(z)+\boson{-}(z)+\boson{+}_R(z),\quad\ket{m\boson{+}}\mapsto  \ket{m(\boson{+}+\boson{-})}\otimes \ket{m\boson{+}},
\end{align*}
where $\boson{+}_L(z)$ (resp.\ $\boson{+}_R(z)$) denotes the field $\boson{+}(z)$ in $\mathcal{U}^+$ (resp.\ $\mathcal{V}^+$), and by 
$$\widetilde{\sf{KS}}_-\colon \heisplus\otimes \mathcal{U}^+\rightarrow\heisminus\otimes \mathcal{U}^-\otimes \mathcal{V}^-,$$
\begin{align*}
    &\alpha_1(z)\mapsto  -\kappa\big(\beta_1(z)-\boson{+}(z)-\boson{-}(z)\big), \quad\alpha_i(z)\mapsto  -\kappa\beta_i(z),\quad (i=2,\dots,n-1),\\
    &\boson{\pm}(z)\mapsto  \mp\kappa\beta_0(z)+\boson{\pm}(z),\quad\ket{m(\boson{+}+\boson{-})}\mapsto  \ket{m\boson{+}}\otimes \ket{m \boson{-}}.
\end{align*}
Then, ${\sf KS}_\pm$ is obtained as a restriction of $\widetilde{\sf KS}_\pm$ by $p_\pm$ in the following way:
\begin{equation*}
\SelectTips{cm}{}
\xymatrix@W15pt@H11pt@R24pt@C25pt@M=5pt{
\Wminus \ar@<-0.3ex>@{^{(}->}[r]^-{{\sf KS}_+}\ar@<-0.6ex>@{^{(}->}[d]_-{p_-}&\Wplus\otimes \mathcal{V}^{+}\ar@<0.3ex>@{^{(}->}[d]^-{p_+\otimes \mathrm{id}}
&\Wplus\ar@<-0.3ex>@{^{(}->}[r]^-{{\sf KS}_-}\ar@<-0.6ex>@{^{(}->}[d]_-{p_+}&\Wminus\otimes \mathcal{V}^{-}\ar@<0.3ex>@{^{(}->}[d]^-{p_-\otimes\mathrm{id}}\\
\heisminus\otimes \mathcal{U}^- \ar@<-0.3ex>@{^{(}->}[r]^-{\widetilde{\sf KS}_+}&\heisplus\otimes \mathcal{U}^+\otimes \mathcal{V}^+,\ar@{}[lu]|{\circlearrowright}&\heisplus\otimes \mathcal{U}^+ \ar@<-0.3ex>@{^{(}->}[r]^-{\widetilde{\sf KS}_-}&\heisminus\otimes \mathcal{U}^-\otimes \mathcal{V}^-.\ar@{}[lu]|{\circlearrowright}
}
\end{equation*}
From this explicit form of ${\sf KS}_\pm$, the following lemma is obtained straightforwardly:
\begin{lemma}\label{Twisted_Emb}
\hspace{1cm}\\
(1) We have 
$\psi^{\pm}(z):={\sf KS}_\pm(\Hmp(z))=\widetilde{\varepsilon}_{\mp}\tHpm(z)\pm\boson{\pm}(z).$\\
(2) The composition
$g_\pm:={\sf KS}_\mp\circ {\sf KS}_\pm\colon
\Wpm\hookrightarrow
\Wpm\otimes \mathcal{V}^+\otimes \mathcal{V}^-$
satisfies
$$g_\pm(w)=w\otimes\ket{\lambda \boson{+}}\otimes\ket{\lambda \boson{-}},\quad (w\in\Omega_\lambda^{\Hpm}(\Wpm),\ \lambda\in \Z).$$
\end{lemma}
Note that Lemma \ref{Twisted_Emb} (1) implies 
$$g_\pm(\Hpm)(z)=\Hpm(z)+\varepsilon_\pm(\boson{+}(z)+\boson{-}(z)).$$
Hence, if we write a general element $v\in \Omega_\lambda^{\Hpm}(\Wpm)\otimes \pi^{\Hpm}_\lambda$ as 
$v=F(\Hpm)\overline{v}$ where $F(\Hpm)$ is a polynomial in the variables $\Hpm_{(-m)}$ ($m\in \Z_{>0}$) and $\overline{v}\in \Omega^{\Hpm}_\lambda(\Wpm)$, we have  
\begin{align}
g_\pm(v)=F\left(g_\pm(\Hpm)\right)\overline{v}\otimes \ket{\lambda \boson{+}}\otimes \ket{\lambda \boson{-}}
\end{align}
by Lemma \ref{Twisted_Emb} (2). Now, the formal Taylor expansion formula
$$F(x+y)={\sf exp}\left(y\frac{d}{dx}\right)F(x)$$
implies the following:
\begin{corollary}\label{Twist_Emb_VA}
Define an even linear operator $\mathcal{H}_\pm$ on $\W\pm\otimes\mathcal{V}_+\otimes\mathcal{V}_-$ by
\begin{equation}\label{Twist_Operator}
\mathcal{H}_\pm=\sum_{m=1}^{\infty}\frac{1}{m}\Hpm_{(m)}\otimes \left(\boson{+}_{(-m)}+\boson{-}_{(-m)}\right).
\end{equation}
Then, we have
$g_\pm(v)={\sf exp}(\mathcal{H}_\pm)(v\otimes\ket{\lambda \boson{+}}\otimes\ket{\lambda \boson{-}})$
for $v\in\Omega_\lambda^{\Hpm}(\Wpm)\otimes \pi^{\Hpm}_\lambda$.
\end{corollary}

\section{Fusion rules of $\mathcal{W}$-superalgebras: Rational cases}\label{SuperW_rational}
Here we describe the fusion rings of $\W_k(\sll_n,\sub)$ and $\W_\ell(\sll_{1|n},\prin)$ in the rational case, that is,
\begin{align}\label{notation of W-algebras}
    \subW:=\W_{-n+\frac{n+r}{n-1}}(\mathfrak{sl}_n,\sub),\quad \prinsW:=\W_{-(n-1)+\frac{n-1}{n+r}}(\mathfrak{sl}_{1|n},\prin)
\end{align}
for $n,r\in\Z_{\geq2}$ with $\mathrm{gcd}(n-1,r+1)=1$
in terms of that of the simple affine vertex algebra of type $A$.

\subsection{Fusion rules of affine vertex algebras}
Here we recall the fusion ring of the simple affine vertex algebra $L_n(\sll_r)$. 
Let $\mathfrak{h}$ denote the Cartan subalgebra of $\sll_r$,  $Q$ the root lattice, $P$ the weight lattice and $P_+$ the set of dominant integral weights. 
The quotient $P/Q$ is represented by the set of fundamental weights $\{\varpi_i\}_{i=1}^{r-1}$ and is isomorphic to $\Z_r$ as groups by $\varpi_i \mapsto i$, where $\varpi_0=0$. The sum $\overline{\rho}=\sum_{i=1}^{r-1}\varpi_i$ is called the Weyl vector.

Let $\mathfrak{h}_{\mathrm{aff}}$ denote the Cartan subalgebra of the affinization $\widehat{\sll}_r =\sll_r[t, t^{-1}] \oplus \C K$, which is $\mathfrak{h}_{\mathrm{aff}}=\mathfrak{h} \oplus \C K$, and $\{\Lambda_i\}_{i=0}^{r-1}$ the set of affine fundamental weights, $\widehat{P}$ the affine weight lattice, $\widehat{P}_+$ the set of dominant integral affine weights, and $\rho=\sum_{i=0}^{r-1}\Lambda_i$ the affine Weyl vector.
We identify $\mathfrak{h}_{\mathrm{aff}}^*$ with $\mathfrak{h}^*\oplus\C\Lambda_0$ and
write $\mathfrak{h}_{\mathrm{aff}}^* \ni \lambda \mapsto \overline{\lambda} \in \mathfrak{h}^*$ for the natural projection. Note that we have a group homomorphism
\begin{align}\label{eq:proj}
\Proj\colon\widehat{P} \twoheadrightarrow P \twoheadrightarrow P/Q \simeq \Z_r,\quad
\Lambda_i \mapsto i.
\end{align}

By \cite{FZ}, the set of simple $L_n(\mathfrak{sl}_r)$-modules $\irr(L_n(\mathfrak{sl}_r))$ is identified as \begin{equation}
    \widehat{P}_+^n(r):=\{\lambda\in \widehat{P}_+\,|\,\lambda(K)=n\}\xrightarrow{\simeq} \irr(L_n(\mathfrak{sl}_r)),\quad \lambda \mapsto L(\lambda).
\end{equation}
We denote by $N_{\lambda, \mu}^\nu$ the fusion rule 
\begin{align}\label{eq:sl-fusion}
L(\lambda)\boxtimes L(\mu)\simeq \bigoplus_{\nu\in \widehat{P}_+^n(r)}N_{\lambda,\mu}^{\nu}L(\nu),
\end{align}
which is calculated by the Kac--Walton formula, see e.g.~\cite[\S 16.2]{DFMS}. By \cite{Fu}, the group of simple currents is $\Pic(L_n(\sll_r))= \{L(n \Lambda_i)\}_{i\in \Z_r}$ and is isomorphic to $\Z_r$ by $L(n \Lambda_i) \mapsto i$. Thus we have a $\Z_r$-action on $\irr(L_n(\sll_r))$ by fusion product
\begin{align}\label{cyclic action}
L(n \Lambda_i)\boxtimes L(\mu)\simeq L(\sigma^i(\mu))
\end{align}
for $i\in\Z_r$ and $\mu \in \widehat{P}_+^n(r)$, where $\sigma$ is the cyclic permutation $\sigma(\Lambda_i)=\Lambda_{i+1}$.

\subsection{Fusion rules of principal $\mathcal{W}$-algebras}
The principal $\W$-algebra $\W^k(\sll_r):=\W^k(\sll_r,\prin)$ is the $\W$-algebra associated with $\sll_r$, the principal nilpotent element $\prin$ and the principal $\Z$-grading on $\mathfrak{sl}_r$. The simple quotient $\W_k(\sll_r)$ at the level
\begin{align*}
k=-r+\frac{r+n}{r+1},\quad (n \in \Z_{\geq0},\ \mathrm{gcd}(n-1, r+1) =1),
\end{align*}
which we denote by $\prinW$, is $C_2$-cofinite \cite{Ar2} and rational \cite{Ar3}.
In addition, we have $\irr(\prinW)=\{\prmod(\lambda)\mid \lambda \in \widehat{P}_+^n(r)\}$ where 
\begin{align}\label{eq:simple-prinW}
\mathbf{L}_\mathcal{W}(\lambda) = H^0_-\left(L(\overline{\lambda}-(k+r)\overline{\rho}+k\Lambda_0)\right),\quad
\left(k = -r+\frac{r+n}{r+1},\quad
\lambda \in \widehat{P}_+^n(r)\right),
\end{align}
and $H^0_-(?)$ denotes the ``$-$''-reduction functor introduced in \cite{FKW}. Note that $\prmod(\lambda)$ has the lowest conformal dimension
\begin{align}\label{conformal dimension of prinW}
h_\lambda^\mathcal{W} := \frac{(\lambda|\lambda+2\rho)}{2(k+r)}-(\lambda|\rho).
\end{align}
It follows from \cite{FKW, C1, AvE1} that the fusion rules of $\prinW$ are given by
\begin{align*}
\prmod(\lambda)\boxtimes \prmod(\mu)\simeq \bigoplus_{\nu\in \widehat{P}_+^n(r)}N_{\lambda,\mu}^{\nu} \prmod(\nu),
\end{align*}
where $N_{\lambda,\mu}^{\nu}$ is determined by \eqref{eq:sl-fusion}.
Thus we have an isomorphism of fusion rings
\begin{align}\label{fusion of prinW}
\mathcal{K}(L_n(\mathfrak{sl}_r)) \xrightarrow{\simeq} \mathcal{K}(\prinW),\quad
L(\lambda) \mapsto  \prmod(\lambda).
\end{align}
In particular, we have an ismorphism of groups
\begin{align*}
\mathrm{Pic}(\prinW) = \{\prmod(n \Lambda_i)\}_{i\in \Z_r}\simeq \Z_r,\quad
\prmod(n \Lambda_i) \mapsto i.
\end{align*}

\subsection{Fusion rules of subregular $\mathcal{W}$-algebras}\label{Fusion_Subregular}
We consider the subregular $\W$-algebra $\subW$ as in \eqref{notation of W-algebras}.
The norm of the Heisenberg field \eqref{H_plus} is $\varepsilon_+=\frac{r}{n}$.

\begin{lemma}[\cite{CL1}]
Let $n, r \in \Z_{\geq2}$ such that $\mathrm{gcd}(n-1, r+1) = 1$. Then
\begin{align}\label{eq:CL1-isom}
\Com\big(\pi^{\Hsub},\subW\big) \simeq \prinW.
\end{align}
\begin{proof}
Since \eqref{eq:CL1-isom} is proven for $r \geq 3$ \cite[Corollary 6.15]{CL1}, we show the case $r =2$.
We first prove that $C_0:=\Com\big(\pi^{\Hsub},\SubW(n,2)\big)$ is a conformal extension of $\PrinW(2,n)$. By \cite[Lemma 2.8]{AvEM}, it suffices to show that the asymptotic growth of $C_0$ is less than $1$. Since $\SubW(n,2)$ is of CFT type and $\pi^{\Hsub}$ acts semisimply on it, both $C_0$ and $\Com\big(C_0,\SubW(n,2)\big)$ are simple and the latter is isomorphic to a positive-definite lattice vertex algebra $V_L$ (cf.~\cite{LX}).
Then $C_0$ is rational by \cite[Theorem 4.12]{CKLR} and its asymptotic growth coincides with the effective central charge $c_{\text{eff}}(C_0)$ by  \cite[Proposition 2.4]{AvEM}.
Since the effective central charge of a simple $\mathcal W$-algebra at admissible level is given by 
\[
c_{\text{eff}}\left(\mathcal W_{-h^\vee + \frac{p}{q}}(\mathfrak g, f)\right) = \text{dim}(\mathfrak g^f) - \frac{h^\vee \text{dim}(\mathfrak g)}{pq},
\]
and the effective central charge of a unitary vertex algebra as e.g. a lattice vertex algebra of a positive definite lattice coincides with the central charge, we obtain 
\[
c_{\text{eff}}\left(\subWtwo\right)-c_{\text{eff}}(V_L) = \frac{n}{n+2} \geq c_{\text{eff}}\left(C_0\right).
\]
The last inequality follows from the fact that every simple $\subWtwo$-module decomposes into a direct sum of simple $C_0\otimes V_L$-modules. Hence the asymptotic growth of $C_0$ is less than $1$.

Next we decompose $\subWtwo$ as a $\prinWtwo\otimes \pi^{\Hsub}$-module. Since $\subWtwo$ is completely reducible, we have 
 $$\subWtwo=\bigoplus_{m\in \Z}C_m\otimes \pi^{\Hsub}_m,$$
 where $C_m$ is a certain simple $C_0$-module. By \cite[Theorem 9.1]{CL1}, $\subWtwo$ is weakly generated by $\prinWtwo\otimes \pi^{\Hsub}$ and the fields $G^\pm(z)$, which are the highest weight vectors in $C_{\pm1}\otimes \pi^{\Hsub}_{\pm1}$. By the equality of conformal weight of $G^+(z)$
 $$\frac{n}{2}=\Delta^{\subWtwo}(G^+(z))=\Delta^{\prinWtwo}(G^+(z))+\Delta^{\pi^{\Hsub}}(G^+(z))$$
and $\Delta^{\pi^{\Hsub}}(G^+(z))=n/4$, we have $\Delta^{\prinWtwo}(G^+(z))=n/4$, which is maximal among the conformal weights for the $(n+2,3)$-minimal series representations. Indeed it coincides uniquely with the conformal weight of the simple current $\mathbf{L}_{\W}(n\Lambda_1)$ of $\prinWtwo$. Recalling that $G^\pm(z)$ generates the lattice $\sqrt{2n}\Z$, we have a conformal embedding
$$\prinWtwo\otimes V_{\sqrt{2n}\Z}\oplus \mathbf{L}_{\W}(n\Lambda_1)\otimes V_{\frac{n}{\sqrt{2n}}+\sqrt{2n}\Z}\hookrightarrow \subWtwo,$$
which completes the weak generators. Therefore, the above embedding is surjective. This completes the proof.
\end{proof}
\end{lemma}

\begin{remark}\label{r=0_1}
It is straightforward to show that we have $\SubW(n,0)\simeq\C$ and $\SubW(2m,1)\simeq V_{\sqrt{2m}\Z}$ for $n\in\Z_{\geq2}$ and $m\in\Z_{\geq1}$.
\end{remark}

It follows from \cite[Theorem 9.4]{CL1} that
\begin{align}
&\Com\left(\prinW, \subW \right) \simeq V_{\ssqrt{nr}\Z},\\
\label{extension to subregular} &\subW \simeq \bigoplus_{i \in \Z_r} \prmod(n\Lambda_i) \otimes V_{\frac{ni}{\ssqrt{nr}}+\ssqrt{nr}\Z}.
\end{align}
In particular, $\subW$ is a simple current extension of $\prinW \otimes V_{\ssqrt{nr}\Z}$ and thus is $C_2$-cofinite and rational (see \cite[Corollary 5.19 (1)]{CGN}). Then, by using $\Proj$ given in \eqref{eq:proj} and Theorem \ref{extension law}, we obtain the following theorem.

\begin{theorem}\label{fusion rules of subregular}
There exists a one-to-one correspondence 
\begin{align*}
\irr(\subW)\simeq \big\{(\lambda,a)\in \widehat{P}_+^n(r)\times \Z_{nr}\,\big|\, \Proj(\lambda)=a\in \Z_r\big\}\big/\Z_r,
\end{align*}
where the $\Z_r$-action on $\widehat{P}_+^n(r)\times \Z_{nr}$ is defined by $m \cdot (\lambda, a) = (\sigma^m(\lambda), a + m n)$ for $m \in \Z_r$. Moreover, we have a $\prinW\otimes V_{\ssqrt{nr}\Z}$-module decomposition
\begin{equation}
\sbmod(\lambda,a)\simeq 
\bigoplus_{i\in\Z_r}\prmod(\sigma^{i}(\lambda))
\otimes V_{\frac{a+ni}{\ssqrt{nr}}+\ssqrt{nr}\Z}
\end{equation}
and the fusion product formula
\begin{align*}
\sbmod(\lambda,a)\boxtimes \sbmod(\mu,b)\simeq \bigoplus_{\nu\in \widehat{P}_+^n(r)} N_{\lambda\mu}^\nu \sbmod(\nu, a+b),
\end{align*}
where $\sbmod(\lambda,a)$ denotes the simple $\subW$-module corresponding to $(\lambda, a)$ and $N_{\lambda\mu}^\nu$ is the affine fusion rule given by \eqref{eq:sl-fusion}.
\end{theorem}
\proof
Using \eqref{conformal dimension of prinW}, it follows that $h_{n\Lambda_i}^\mathcal{W} = i n (r-i)/2r$ so that the monodromy operator $\mathcal{M}_{\prmod(n\Lambda_i),\prmod(\lambda)}$ is
\begin{align*}
\mathcal{M}_{\prmod(n\Lambda_i),\prmod(\lambda)}=\zeta_r^{-i\pi_{P/Q}(\lambda)},\quad \zeta_r=\mathrm{e}^{\frac{2\pi \ssqrt{-1}}{r}}.
\end{align*}
Thus the simple module $\mathbf{L}_\W(\lambda)\otimes V_{\frac{an}{rn}+\ssqrt{nr}\Z}$ is local with respect to the simple currents $\mathbf{L}_\W(n\Lambda_i)\otimes V_{\frac{ni}{nr}+\ssqrt{nr}\Z}$ if and only if 
$$\pi_{P/Q}(\lambda)=a\in \Z_r.$$
Here $a\in \Z_{nr}$ is regarded as an element of $\Z_r$ by the natural projection $\Z_{nr}\twoheadrightarrow \Z_r$. Thus the assertion follows from Corollary \ref{classification of irreducibles} and \ref{Grothendieck groups for induction}.
\endproof

\begin{corollary}
The modular $S$-matrix for $\subW$ is given by
\begin{equation*}
S^{\,\sf sb}_{(\lambda,a),(\mu,b)}=e^{\frac{2\pi\sqrt{-1}ab}{nr}}\sqrt{\frac{r}{n}}S^{\W}_{\lambda,\mu},
\end{equation*}
where $S^{\W}_{\lambda,\mu}$ is the modular $S$-matrix for $\prinW$.
\end{corollary}

Theorem \ref{fusion rules of subregular} together with \eqref{level-rank duality} implies the following isomorphism.

\begin{proposition}\label{fusion ring of subregular}
The one-to-one correspondence
$$\irr(L_r(\sll_n))\xrightarrow{\simeq}\irr(\subW);\  L(\lambda)\mapsto \sbmod(\lambda^t,\ell(\lambda))\quad \big(\lambda\in P^r_+(n)\big),$$
gives an isomorphism of fusion rings
\begin{equation}\label{Fusion_Subregular_Affine}
    \K(L_r(\sll_n))\simeq \K(\subW).
\end{equation}
\end{proposition}

\begin{remark}
In \cite{AvE2}, T.~Arakawa and J.~van Ekeren construct another ring isomorphism $\mathcal{K}(L_r(\mathfrak{sl}_n))\simeq \mathcal{K}(\subW)$ for even $n$. Let $v$ be coprime to $n+r$ and $\ell = - n + \frac{n+r}{v}$, then the fusion rings of $\mathcal{K}(L_r(\mathfrak{sl}_n))$ and $\mathcal{K}(L_\ell(\mathfrak{sl}_n))$ coincide and in fact for $v = n-1$ there is an equivalence of braided tensor categories between the category of ordinary modules of  
$L_\ell(\mathfrak{sl}_n)$ and those of $\subW$ \cite[Theorem 10.4]{ACF}.
We expect that these two parameterization of simple modules coincide.
\end{remark}
\subsection{Fusion rules of principal $\W$-superalgebras}\label{Fusion_SPrin}
We consider the principal $\W$-superalgebra $\prinsW$ as in \eqref{notation of W-algebras}.
The norm of the Heisenberg field \eqref{H_minus} is $\varepsilon_-=\frac{r}{n+r}$.

\begin{theorem}\label{decomposition for principal super}
There exists an isomorphism of vertex superalgebras 
\begin{align*}
\Com\left(\prinW,\prinsW\right)\simeq V_{\ssqrt{(n+r)r}\Z}.
\end{align*}
Moreover, we have
\begin{align}\label{extension to superW}
\prinsW\simeq \bigoplus_{i\in \Z_r}\mathbf{L}_\mathcal{W}(n\Lambda_i)\otimes V_{\frac{(n+r)i}{\ssqrt{(n+r)r}}+\ssqrt{(n+r)r}\Z}
\end{align}
as $\prinW\otimes V_{\ssqrt{(n+r)r}\Z}$-modules. In particular, $\prinsW$ is a simple current extension of $\prinW \otimes V_{\ssqrt{(n+r)r}\Z}$ and thus is $C_2$-cofinite and rational.
\begin{proof}
First of all, the generator of $\ssqrt{nr}\Z$ in \eqref{extension to subregular} is given by $n\Hsub$ and we have
\begin{align*}
\subW = \bigoplus_{i\in \Z_r}\bigoplus_{\lambda \in \Z}\prmod(n\Lambda_i)\otimes\pi^{\Hsub}_{r\lambda+i}.
\end{align*}
Then, since $(n+r)\psi^+ = n\Hsub + r \boson{+}$ by Lemma \ref{Twisted_Emb} (1), we have
\begin{align}\label{Kazama-Suzuki decomposition}
\begin{split}
\subW \otimes V_\Z
&= \bigoplus_{i\in \Z_r}\bigoplus_{\lambda, \mu\in \Z}\prmod(n\Lambda_i)\otimes\pi^{\Hsub}_{r\lambda +i}\otimes \pi^{\boson{+}}_\mu\\
&\simeq \bigoplus_{i\in \Z_r}\bigoplus_{\lambda, \mu\in \Z}\prmod(n\Lambda_i)\otimes\pi^{\tHsub}_{\mu-r\lambda- i}\otimes \pi^{(n+r)\psi^+}_{nr\lambda+ni+r\mu}
\end{split}
\end{align}
as $\prinW\otimes\pi^{\tHsub}\otimes \pi^{(n+r)\psi^+}$-modules. Thus, by \cite[Corollary 5.16]{CGN}, we obtain
\begin{align*}
\prinsW
&\simeq \Com\left(\pi^{\tHsub},\subW\otimes \mathcal{V}_+\right)\\
&\simeq \bigoplus_{i\in \Z_r}\bigoplus_{\lambda \in \Z}\prmod(n\Lambda_i) \otimes \pi^{(n+r)\psi^+}_{(n+r)(r \lambda +i)}.
\end{align*}
This implies the assertion since we have
$$\Com(\prinW,\prinsW)\simeq \bigoplus_{\lambda\in \Z}\pi^{(n+r)\psi^+}_{(n+r)r\lambda}\simeq V_{\ssqrt{(n+r)r}\Z}$$
by the characterization of lattice vertex superalgebras.
\end{proof}
\end{theorem}
\begin{remark}\label{r=0,1 for super}
As a counterpart to Remark \ref{r=0_1}, we have
$\PrinsW(n,0)\simeq\C$ and $\PrinsW(2m,1)\simeq V_{\sqrt{2m+1}\Z}$ for $n\in\Z_{\geq2}$ and $m\in\Z_{\geq1}$.
\end{remark}

By the same argument as in the proof of Theorem \ref{fusion rules of subregular}, the simple current extension \eqref{extension to superW} implies the following.
\begin{theorem}\label{fusion rules of superW}
There exists a one-to-one correspondence
\begin{align*}
\irr(\prinsW)\simeq \big\{(\lambda,a)\in \widehat{P}_+^n(r)\times \Z_{(n+r)r}\,\big|\, \Proj(\lambda)=a\in \Z_r\big\}\big/\Z_r,
\end{align*}
where the $\Z_r$-action on $\widehat{P}_+^n(r)\times \Z_{(n+r)r}$ is defined by $m \cdot (\lambda, a) = \big(\sigma^m(\lambda), a + m(n+r)\big)$ for $m \in \Z_r$. Moreover, we have a $\prinW\otimes V_{\ssqrt{(n+r)r}\Z}$-module decomposition
\begin{equation*}
\spmod(\lambda,a)\simeq 
\bigoplus_{i\in\Z_r}\prmod(\sigma^{i}(\lambda))
\otimes V_{\frac{a+(n+r)i}{\ssqrt{(n+r)r}}+\ssqrt{(n+r)r}\Z}
\end{equation*}
and the fusion product is given by 
\begin{align*}
\spmod(\lambda,a)\boxtimes \spmod(\mu,b)\simeq \bigoplus_{\nu\in \widehat{P}_+^n(r)} N_{\lambda\mu}^\nu \spmod(\nu, a+b),
\end{align*}
where $\spmod(\lambda,a)$ denotes the simple $\prinsW$-module corresponding to $(\lambda, a)$ and $N_{\lambda\mu}^\nu$ is the affine fusion rule given by \eqref{eq:sl-fusion}.
\end{theorem}
\begin{corollary}
We obtain the following.
\begin{enumerate}
    \item We have a ring isomorphism $$\K(\prinsW)\simeq\left(\K(L_n(\sll_{r}))\underset{\Z[\mathbb{Z}_{r}]}{\otimes}\Z[\Z_{(n+r)r}]\right)^{\mathbb{Z}_{r}}.$$
    \item The number of inequivalent simple $\prinsW$-modules is $\binom{n+r}{n}$.
    \item The group of simple currents of $\prinsW$-modules is
    \begin{equation*}
    \mathrm{Pic}(\prinsW)=\{\spmod(n\Lambda_0,ar)\}_{a\in \Z_{n+r}}\simeq\Z_{n+r},
    \end{equation*}
    where the isomorphism is given by $\spmod(n\Lambda_0,ar)\mapsto a$.
\end{enumerate}
\end{corollary}
\begin{corollary}
The formal character 
\begin{equation*}
\ch(\spmod(\lambda,a))(q,z)
=\tr_{\spmod(\lambda,a)}(q^{L_{0}-c/24}z^{\Hsup_{(0)}})
\end{equation*}
is given by
\begin{align*}
\frac{1}{\eta(q)^{r}}\sum_{\begin{subarray}c i\in \Z_r\\ w\in W\end{subarray}}\varepsilon(w)
q^{\frac{|(1+r)w\sigma^{i}(\lambda+\widehat{\rho})
-(n+r)(\Lambda_{0}+\widehat{\rho})|^{2}}{2(n+r)(1+r)}}\sum_{\mu\in (n+r)(i+r\Z)}q^{\frac{(a+\mu)^{2}}{2(n+r)r}}z^{\frac{a+\mu}{n+r}}
\end{align*}
and the modular $S$-matrix for $\prinsW$ is given by
\begin{equation*}
S^{\,\sf spr}_{(\lambda,a),(\mu,b)}=e^{\frac{2\pi\sqrt{-1}ab}{(n+r)r}}\sqrt{\frac{r}{n+r}}S^{\W}_{\lambda,\mu}.
\end{equation*}
\end{corollary}

Finally, we give another description of $\irr(\prinsW)$ by the Kazama--Suzuki coset construction. It follows from \eqref{extension to superW} and \eqref{Kazama-Suzuki decomposition} that 
$$\Com(\prinsW, \subW\otimes V_\Z)\simeq \bigoplus_{a\in \Z}\pi^{\widetilde{H}_1}_{(n+r)a}\simeq V_{\ssqrt{(n+r)n}\Z}$$
as vertex superalgebras. Hence, $\subW\otimes V_\Z$ decomposes into
\begin{align*}
&\bigoplus_{i\in \Z_r}\bigoplus_{\mu\in \Z_{n+r}}\prmod(n\Lambda_i)\otimes V_{\ssqrt{(n+r)r}\Z+\frac{i(n+r)+r\mu}{\ssqrt{(n+r)r}}}\otimes V_{\frac{n\mu}{\ssqrt{n(n+r)}}+\ssqrt{n(n+r)}\Z}\\
&\simeq \bigoplus_{\mu\in \Z_{n+r}}\spmod(n\Lambda_0,\mu r)\otimes V_{\frac{n\mu}{\ssqrt{n(n+r)}}+\ssqrt{n(n+r)}\Z}
\end{align*}
as a $\prinsW\otimes V_{\ssqrt{n(n+r)}\Z}$-module. Thus, $\subW\otimes V_\Z$ is an order $(n+r)$ simple current extension of $\prinsW\otimes V_{\ssqrt{n(n+r)}\Z}$. Since $V_\Z$ is a holomorphic vertex operator superalgebra, 
$$\subW\Mod\simeq \subW\Mod\otimes V_\Z\Mod,\quad M\mapsto M\otimes V_\Z$$
is an isomorphism of braided tensor supercategories. Hence Theorem \ref{Inverting_SCE} together with the isomorphism (\ref{Fusion_Subregular_Affine}) imply the following.
\begin{theorem}\label{the final fusion for sprin}
There exists a one-to-one correspondence
\begin{align*}
\irr(\prinsW)\simeq \big\{(\lambda,a)\in \widehat{P}_+^r(n)\times \Z_{(n+r)n}\,\big|\,
\Proj(\lambda)=-a\in \Z_n
\big\}\big/\Z_n
\end{align*}
where $\Z_n$ acts on $\widehat{P}_+^r(n)\times \Z_{(n+r)n}$ by $m \cdot (\lambda, a) = \big(\sigma^m(\lambda), a - m(n+r)\big)$, $(m \in \Z_n)$. 
Let $\mathbf{L}_{\mathcal{SW}}(\lambda,a)$ denote the simple $\prinsW$-module under this parameterization. Then we have an isomorphism of fusion rings
\begin{equation*}
    \K(\prinsW)\simeq\left(\K(L_r(\sll_n))\underset{\Z[\Z_n]}{\otimes}\Z[\Z_{n(n+r)}]\right)^{\Z_n};\ \mathbf{L}_{\mathcal{SW}}(\lambda,a)\mapsto L(\lambda)\otimes [a].
\end{equation*}
\end{theorem}

\subsection{On more correspondences}\label{more}

In this section, we used that the Heisenberg coset of our subregular $\subW$-algebra and our principal $\prinsW$-superalgebra are rational and $C_2$-cofinite. In particular the representation categories of the Heisenberg cosets are vertex tensor categories. This is in fact all we need in order to apply the simple current extension procedure. Note, that the theory of vertex algebra extensions applies also to infinite order extensions \cite{CMY1} and the existence of vertex tensor category on Fock modules for the Heisenberg vertex algebra is known as long as the Heisenberg weight is real \cite{CKLR}.
 There are only a few  cases of subregular $\W$-algebras where it is known that its Heisenberg coset has categories of modules that are vertex tensor categories. In all these case the singlet algebra $\mathcal M(p)$ \cite{Ad4, AM} appears as a coset. 

The first example are the $\mathcal B_p$-algebras for $p \in \mathbb Z_{\geq 2}$ of \cite{CRW}. The $\mathcal B_2$-algebra is just the vertex algebra of $\beta\gamma$-ghosts. This is not a subregular $\W$-algebra, but $\mathcal B_3$ is $L_{-4/3}(\sll_2)$ \cite{Ad3} and in general $\mathcal B_p$ is $\W_k(\sll_{p-1},\sub)$ for $k = -(p-1)^2/{p}$ \cite{ACGY}.  The $\mathcal B_p$-algebras have received attention since they are the chiral algebras of certain quantum field theories, called Argyres--Douglas theories \cite{BN, C2}, proven in  \cite{ACGY, ACKR}. 
The second example is the $\mathbb Z_n$-orbifold of $\beta\gamma$-ghosts, which is $\W_k(\sll_{n},\sub)$ for $k = -n +\frac{n+1}{n}$ \cite[Theorem 9.4]{CL1}.
The Heisenberg coset of the $\mathcal B_p$-algebra is the singlet algebra $\mathcal M(p)$ and the Heisenberg coset of $\beta\gamma$-ghosts and hence their $\mathbb Z_n$-orbifolds is the singlet algebra $\mathcal M(2)$. Existence of vertex tensor category is not completely understood, but at least it is established for all modules that appear as summands inside triplet algebra modules \cite{CMY2}. The $\mathcal B_p$-algebra as a simple current extension of a Heisenberg vertex algebra times the singlet algebra has already been investigated in \cite{ACKR}. 

It is in general a very difficult problem to get the existence of vertex tensor categories and so we will now study a procedure that does not require this.

\section{Beyond rational case: coset functors}\label{sec:Kazama-Suzuki}
In this section, we establish a linear equivalence of weight module categories for the subregular $\W$-algebra $\Wplus$ and the principal $\W$-superalgebra $\Wminus$ at arbitrary levels under the duality relation \eqref{def of Wpm}. This generalizes the known results for the $\mathcal{N}=2$ super Virasoro algebra and the affine vertex algebra associated to $\sll_2$ in \cite{FST,Sato1}, which corresponds to the case $n=2$.
\subsection{Coset functors}\label{Coset_Functor}
Recall that we have the category $\modcat{\pm}$ of strongly graded grading-restricted generalized $\Wpm$-modules. Here we introduce a full subcategory $\wtcat{\pm}$ of $\modcat{\pm}$ whose objects are semisimple as $\pi^{H^\pm}$-modules. For an  object $M\in Ob(\wtcat{\pm})$, let $\mathrm{Supp}(M)$ denote the set of $H^\pm_{(0)}$-eigenvalues of $M$. Since $\mathrm{Supp}(\Wpm)=\Z$ (see \S\ref{def of subreg}-\ref{def of sprin}), $\mathrm{Supp}(M)$ defines naturally a subset of $\C/\Z$. We may decompose $M$ into 
\begin{equation}\label{decomposition of module}
M=\bigoplus_{\lambda\in\mathrm{Supp}(M)} M_{\lambda}
\simeq \bigoplus_{\lambda\in\mathrm{Supp}(M)}
\Omega_{\lambda}^{H^\pm}(M)\otimes\pi_{\lambda}^{H^\pm}
\end{equation}
as $\mathrm{Com}(\pi^{H^\pm},\Wpm)\otimes \pi^{H^\pm}$-modules where
$$\Omega_{\lambda}^{H^\pm}(M):=\{w\in M\,|\, \forall m\geq0, H^\pm_{(m)}w=\delta_{m,0}\lambda w\}.$$
Note that for each $\lambda\in \C$, $M_{[\lambda]}:=\bigoplus_{\mu\in \lambda+\Z}M_\mu\subset M$ is a submodule by \eqref{support of W}. Thus we have a decomposition $M=\bigoplus_{[\lambda]\in \C/\Z}M_{[\lambda]}$ as a $\Wpm$-module. Motivated by this, we introduce full subcategories $\wtcat{\pm}_{[\lambda]}$ consisting of objects $M$ such that $\mathrm{Supp}(M)\subset \lambda+\Z$. Then we have a decomposition 
$$\wtcat{\pm} =\bigoplus_{[\lambda]\in\C/\Z}\wtcat{\pm}_{[\lambda]}.$$

For an object $M$ in $\wtcat{\pm}$, we have a decomposition
\begin{equation*}
M\otimes\mathcal{V}_{\pm}\simeq\bigoplus_{\xi\in\C}
\Omega_{\xi}^{\widetilde{H}^\pm}(M\otimes\mathcal{V}_{\pm})\otimes
\pi^{\widetilde{H}^\pm}_{\xi}
\end{equation*}
as a $\mathrm{Com}(\pi^{\widetilde{H}^\pm},\Wpm\otimes\mathcal{V}_{\pm})\otimes\pi^{\widetilde{H}^\pm}$-module.
Then the $\mathrm{Com}(\pi^{\widetilde{H}^\pm},\Wpm\otimes\mathcal{V}_\pm)$-module $ \Omega_{\xi}^{\widetilde{H}^\pm}(M\otimes\mathcal{V}_\pm)$ has a $\Wpm$-module structure through the isomorphism $\sf{KS}_\pm$ in \eqref{KS_Duality}, which we denote by $\coset^\pm_\xi(M)$.
For $M\in Ob(\wtcat{\pm}_{[\lambda]})$, it follows from 
\begin{align}\label{coset decomp}
\begin{split}
M\otimes \mathcal{V}_{\pm}
&\simeq\bigoplus_{\begin{subarray}c \mu\in \lambda+\Z\\ \nu\in \Z\end{subarray}}\Omega_\mu^{H^\pm}(M)\otimes \pi^{H^\pm}_\mu\otimes \pi^{\boson{\pm}}_\mu\\
&\simeq 
\bigoplus_{\begin{subarray}c \xi\in \mp\lambda+\Z\\ \nu\in \Z\end{subarray}}\Omega_{\nu\mp\xi}^{H^\pm}(M)\otimes \pi^{{\sf KS}_\pm(H^\mp)}_{\nu+\widetilde{\varepsilon}_{\mp}\xi}\otimes \pi^{\widetilde{H}^\pm}_\xi 
\end{split}
\end{align}
that $\coset^\pm_{\xi}(M)$ is non-zero only if $\xi\in \mp\lambda+\Z$ and in this case  
\begin{align}\label{coset functor}
\coset^+_{\xi}(M)\simeq\bigoplus_{\nu\in \Z}\Omega_{\nu-\xi}^{H^+}(M)\otimes \pi^{H^-}_{\nu+\widetilde{\varepsilon}_{-}\xi},\quad \coset^-_{\xi}(M)\simeq\bigoplus_{\nu\in \Z}\Omega_{\nu+\xi}^{H^-}(M)\otimes \pi^{H^+}_{\nu+\widetilde{\varepsilon}_{+}\xi}
\end{align}
Now, it is obvious that the assignments $\coset^\pm_{\xi}$ induce $\C$-linear superfunctors
$$\coset^+_{\xi}\colon\wtcat{+}_{[\lambda]}
\to\wtcat{-}_{[\widetilde{\varepsilon}_{-}\xi]},\quad \coset^-_{\xi}\colon
\wtcat{-}_{[\lambda]}
\to\wtcat{+}_{[\widetilde{\varepsilon}_{+}\xi]},$$
which are non-zero only if $\xi\in \mp\lambda+\Z$, respectively. We call $\coset^\pm_\xi$ \emph{coset functors}.
\subsection{Equivalence of categories}\label{sec: Equivalence of categories}
Here we establish an equivalence of categories $\Wpm\Mod$ weight-wisely:
\begin{theorem}\label{Quasi_Inv}
For $\lambda\in\C$, the two functors 
$$\coset^+_{-\lambda}\colon\wtcat{+}_{[\lambda]}
\to\wtcat{-}_{[-\widetilde{\varepsilon}_{-}\lambda]},\quad \coset^-_{-\widetilde{\varepsilon}_{-}\lambda}\colon\wtcat{-}_{[-\widetilde{\varepsilon}_{-}\lambda]}\to
\wtcat{+}_{[\lambda]}$$
are mutually quasi-inverse to each other and give an equivalence of categories.
\end{theorem}
The following lemma will play a key role in the proof:
\begin{lemma}\label{Twist_Emb_Mod}
For an object $M$ in $\wtcat{\pm}_{[\lambda]}$, the linear map
$g_{\pm,\lambda}^M\colon M\to M\otimes \mathcal{V}_+
\otimes\mathcal{V}_-$ defined by
$$g_{\pm,\lambda}^M(w)={\sf exp}(\mathcal{H}_\pm)
(w\otimes\ket{\mu \boson{+}}\otimes\ket{\mu \boson{-}}),\quad (w\in M_{\lambda+\mu},\ \mu\in\Z),$$
is an embedding of $\Wpm$-modules.
\end{lemma}
\begin{proof}
Since ${\sf exp}(\mathcal{H}_\pm)$ is invertible, $g_{\pm,\lambda}^M$ is injective. By Corollary \ref{Twist_Emb_VA}, $g_{\pm,\lambda}^M$ is already a $\pi^{H^\pm}$-homomorphism. Therefore, it remains to show 
\begin{equation*}
g_\pm(a)(z)g_{\pm,\lambda}^M(w)=g_{\pm,\lambda}^M(a(z)w)
\end{equation*}
for $\pi^{H^\pm}$-singular vectors $a\in \Omega_\nu^{H^\pm}(\Wpm)$ and $w\in \Omega_{\lambda+\mu}^{H^\pm}(M)$. For this, we decompose the structure map of $M$ 
$$Y_M(\cdot,z)\colon \left(\Omega_\nu^{H^\pm}(\Wpm)\otimes \pi^{H^\pm}_\nu\right)\otimes \left(\Omega_{\lambda+\mu}^{H^\pm}(M)\otimes \pi^{H^\pm}_{\lambda+\mu}\right)\rightarrow \left(\Omega_{\lambda+\mu+\nu}^{H^\pm}(M)\otimes \pi^{H^\pm}_{\lambda+\mu+\nu}\right)(\!(z)\!)$$
to a $\Com(\pi^{H^\pm},\Wpm)\otimes \pi^{H^\pm}$-intertwining operator $Y(\cdot,z)=I_1(\cdot,z)\otimes I_2(\cdot,z)$ so that 
$$I_2(\ket{\nu},z)\ket{\lambda+\mu}=z^{\nu(\lambda+\mu)/\varepsilon_\pm}\difexp{-\frac{\nu}{\varepsilon_\pm}H^\pm}\ket{\lambda+\mu+\nu}$$
where we denote
\begin{align}\label{eq:difexp}
   \difexp{h}={\sf exp}\left(\sum_{m=1}^\infty\frac{h_{(-m)}}{m}z^m\right)
\end{align}
for a Heisenberg field $h(z)=\sum_{m\in\Z}h_{(n)}z^{-n-1}$.
Then, by Lemma \ref{Twisted_Emb}, we have
\begin{align*}
&g_\pm(a)(z)g_{\pm,\lambda}^M(w)\\
&=Y_M(a\otimes \ket{\nu\boson{+}}\otimes \ket{\nu\boson{-}},z) (w\otimes \ket{\mu\boson{+}}\otimes \ket{\mu\boson{-}})\\
&=I_1(a,z)w\otimes I_2(\ket{\nu},z)\ket{\lambda+\mu}\otimes Y_{V_\Z}(\ket{\nu\boson{+}},z)\ket{\mu\boson{+}}\otimes Y_{V_{\ssqrt{-1}\Z}}(\ket{\nu\boson{-}},z)\ket{\mu\boson{-}}\\
&=I_1(a,z)w\otimes z^{\nu(\lambda+\mu)/\varepsilon_\pm}\difexp{-\frac{\nu}{\varepsilon_\pm}H^\pm}\ket{\lambda+\mu+\nu}\\
&\hspace{1cm} \otimes \difexp{\nu \boson{+}}\ket{(\mu+\nu)\boson{+}}\otimes 
\difexp{\nu \boson{-}}\ket{(\mu+\nu)\boson{-}}\\
&=I_1(a,z)w\otimes z^{\nu(\lambda+\mu)/\varepsilon_\pm}\difexp{-\frac{\nu}{\varepsilon_\pm}g_\pm(H^\pm)}\ket{\lambda+\mu+\nu}\otimes \ket{(\mu+\nu)\boson{+}}\otimes\ket{(\mu+\nu)\boson{-}}\\
&=g^M_{\pm,\lambda}(a(z)w).
\end{align*}
This completes the proof.
\end{proof}
\proof[Proof of Theorem \ref{Quasi_Inv}]
We show $\mathrm{id}\simeq \coset^-_{-\widetilde{\varepsilon}_{-}\lambda}\circ\coset^+_{-\lambda}$. 
For an object $M$ in $\wtcat{+}_{[\lambda]}$, it follows from \eqref{coset decomp} that the $\Wplus$-module $\coset^-_{-\widetilde{\varepsilon}_{-}\lambda}\circ\coset^+_{-\lambda}(M)$ is realized as the subspace of 
$$\bigoplus_{\mu\in\Z}M_{\lambda+\mu}\otimes \pi^{\boson{+}}_{\mu\boson{+}}\otimes \pi^{\boson{-}}_{\mu\boson{-}}\subset M\otimes \mathcal{V}^+\otimes \mathcal{V}^-,$$
consisting of highest weight vectors of $\pi^{\widetilde{H}^+}\otimes \pi^{\widetilde{H}^-}$, i.e., elements $w$ satisfying 
$$\widetilde{H}^+_{(m)}w=-\lambda\delta_{m,0}w,\quad \widetilde{H}^-_{(m)}w=-\widetilde{\varepsilon}_{-}\lambda\delta_{m,0}w,\quad (m\geq0).$$
Hence, the $\Wplus$-module  $\coset^-_{-\widetilde{\varepsilon}_{-}\lambda}\circ\coset^+_{-\lambda}(M)$ coincides with the image of 
$g_{+,\lambda}^M$ in Lemma \ref{Twist_Emb_Mod}.
Therefore, the family of $\Wplus$-homomorphisms
\begin{equation}\label{Natural_Isom_g}
\left\{g_{+,\lambda}^M\colon M\to\coset^-_{-\widetilde{\varepsilon}_{-}\lambda}\circ\coset^+_{-\lambda}(M)\,\big|\,M\in Ob(\wtcat{+}_{[\lambda]})\right\}
\end{equation}
gives a desired natural isomorphism. Similarly, one can show that
\begin{equation*}
\left\{g_{-,-\widetilde{\varepsilon}_{-}\lambda}^M\colon M\rightarrow \coset^+_{-\lambda}\circ\coset^-_{-\widetilde{\varepsilon}_{-}\lambda}(M)\,\big|\,M\in Ob(\wtcat{-}_{[-\widetilde{\varepsilon}_{-}\lambda]})\right\}
\end{equation*}
gives a natural isomorphism $\mathrm{id}\simeq \coset^+_{-\lambda}\circ\coset^-_{-\widetilde{\varepsilon}_{-}\lambda}$. This completes the proof.
\endproof
%

\section{Relative semi-infinite cohomology functor}
Here we reinterpret the coset functors in terms of the relative semi-infinite cohomology and then study its compatibility with monoidal structure on the module categories of $\Wpm$.
\subsection{Semi-infinite Cohomology of $\widehat{\mathfrak{gl}}_{1}$ 
relative to $\mathfrak{gl}_{1}$}

Let $\pi^A$ be a degenerate (i.e., commutative) Heisenberg vertex algebra generated by a field $A(z)$ and $\mathcal{E}$ be the $bc$-system vertex superalgebra generated by (odd) fields $\varphi(z), \varphi^*(z)$ satisfying the  OPEs
$$\varphi(z)\varphi^*(w)\sim \frac{1}{z-w},\quad \varphi(z)\varphi(w)\sim 0\sim \varphi^*(z)\varphi^*(w).$$
We introduce a cohomological grading $\mathcal{E}^\bullet= \bigoplus_{i\in\Z}\mathcal{E}^i$ on $\mathcal{E}$ as a strong $\Z$-grading given by $\varphi \in \mathcal{E}^{-1}$ and $\varphi^*\in \mathcal{E}^{1}$. 
For a $\pi^A$-module $M$, the $\pi^A\otimes \mathcal{E}$-module 
$C_\infty^\bullet(M):=M\otimes \mathcal{E}^\bullet$ forms a cochain complex with differential $d_{(0)}$ where $d(z) = A(z)\otimes \varphi^*(z)$.
The cohomology
\begin{align*}
H^{\bullet}_\infty(M) = H\left(C^\bullet_\infty(M),d_{(0)}\right)
\end{align*}
is called the semi-infinite cohomology of $\widehat{\mathfrak{gl}}_{1}$ with coefficients in $M$ \cite{Fe,FGZ}. In the following, we will use cohomologies obtained from a subcomplex
\begin{align*}
C^{\mathrm{\textsf{rel}},\bullet}_\infty(M)
:= \{ u \in C^\bullet_\infty(M) \mid A_{(0)}u = \varphi_{(0)}u =0\},
\end{align*}
which coincides with $M^{A_{(0)}}\otimes\mathcal{SF}$. Here $M^{A_{(0)}}$ is the space of $A_{(0)}$-invariants in $M$ and $\mathcal{SF}$ is a subalgebra of $\mathcal{E}$ generated by $\varphi(z)$ and $\partial \varphi^*(z)$. The cohomology 
\begin{align*}
\relcoh{\bullet}(M) = H\left(C^\mathrm{\textsf{rel},\bullet}_\infty(M),d_{(0)}\right)
\end{align*}
is called the \emph{relative semi-infinite cohomology} of $\widehat{\mathfrak{gl}}_{1}$ with coefficients in $M$ \cite{Fe,FGZ}. 
Note that if $V$ is a vertex superalgebra which contains $\pi^A$, then $\relcoh{0}(V)$ is naturally a vertex superalgebra and if $M$ is a $V$-module, then $\relcoh{p}(M)$ is naturally a $\relcoh{0}(V)$-module for each $p\in\Z$.

Consider the following spacial case: let $\pi^{B_{\pm}}$ be Heisenberg vertex algebras generated by fields $B_\pm(z)$ satisfying the OPEs
$$B^\pm(z)B^\pm(w)\sim \frac{\pm b}{(z-w)^2}$$
for some scalar $b\neq 0$. Then the Fock module $\pi^{B^+}_\lambda\otimes \pi^{B^-}_\mu$ is a module over the degenerate Heisenberg vertex algebra generated by $A(z)=B^+(z)+B^-(z)$. We will use the following fundamental result, which is a special case of \cite{FGZ, CFL}.
\begin{proposition}[{\cite{FGZ, CFL}}]\label{prop:CFL} 
For $p\in \Z$,
\begin{align*}
\relcoh{p}\left(\pi^{B^+}_\lambda\otimes\pi^{B^-}_\mu\right)=  \delta_{p,0}\delta_{\lambda+\mu,0}\C \left[\ket{\lambda}\otimes \ket{\mu}\right].
\end{align*}
\end{proposition}
\proof
We give a proof for the completeness of the paper. 
Since  
$C^{\mathrm{\textsf{rel}},\bullet}_\infty(\pi^{B^+}_\lambda\otimes\pi^{B^-}_\mu)$ is 0 if $\lambda+\mu\neq0$, we may assume $\mu=-\lambda$. Then the complex is isomorphic to the tensor product $\otimes_{n\geq1}\mathcal{C}[n]$ where each $\mathcal{C}[n]$ is isomorphic to the same complex $\mathcal{D}^\bullet=\C[x,y]\otimes \bigwedge^\bullet_{\psi,\psi^*}$ given by 
\begin{align}\label{complex}0\rightarrow \C[x,y]\psi\rightarrow \C[x,y]\oplus  \C[x,y]\psi\psi^*\rightarrow \C[x,y]\psi^*\rightarrow 0
\end{align}
with differential $D=\psi^*\partial/\partial y+x\partial/\partial \psi$ 
by change of variables
$$x\mapsto \frac{1}{2na}\left(B^+_{(-n)}-B^-_{(-n)}\right),\ y\mapsto B^+_{(-n)}+B^-_{(-n)},\ \psi\mapsto \varphi_{(-n)},\ \psi^*\mapsto  \varphi_{(-n)}^*.$$
It is straightforward to show $H^p(\mathcal{D}^\bullet)= \delta_{p,0}\C$ and then the assertion follows from the 
K$\ddot{\text{u}}$nneth formula.
\endproof
\subsection{Coset as relative semi-infinite cohomology}
We reinterpret the coset  functors in \S\ref{sec: Equivalence of categories} in terms of relative semi-infinite cohomology. For this, we begin with the $\W$-superalgebras $\Wpm$.

Fix a scalar $\epsilon$ satisfying $\epsilon^2=\varepsilon_-/\varepsilon_+$.
By using $\epsilon$, we introduce new sets of generating fields $\{s^\pm(z),t^\pm(z)\}$ of $\pi^{\boson{+}}\otimes \pi^{\boson{-}}$ by   
\begin{align*}
\left(
\begin{array}{c}
s^+ \\
t^+
\end{array}
\right)=
\left(
\begin{array}{cc}
-1 & -\epsilon^{-1} \\
1 & \epsilon
\end{array}
\right)
\left(
\begin{array}{c}
\boson{+} \\
\boson{-}
\end{array}
\right),\quad
\left(
\begin{array}{c}
s^-\\
t^-
\end{array}
\right)=
\left(
\begin{array}{cc}
-\epsilon & 1 \\
\epsilon^{-1} & -1
\end{array}
\right)
\left(
\begin{array}{c}
\boson{+} \\
\boson{-}
\end{array}
\right),
\end{align*}
which satisfy the OPEs
\begin{equation*}
s^\pm(z)s^\pm(w)\sim\frac{-\varepsilon_\pm}{(z-w)^2},\quad
t^\pm(z)t^\pm(w)\sim\frac{\varepsilon_{\mp}}{(z-w)^2},\quad
s^\pm(z)t^\pm(w)\sim 0.
\end{equation*}
Then the vertex superalgebra $\gluing{\pm} = V_{\ssqrt{\pm 1}\Z} \otimes \pi^{\boson{\mp}}$ decomposes into
$$\gluing{\pm}\simeq \bigoplus_{\mu\in \Z}\pi^{s^\pm}_{-\mu} \otimes \pi^{t^\pm}_{\mu}$$
as $\pi^{s^\pm}\otimes \pi^{t^\pm}$-modules. 
The tensor product $\Wpm\otimes\gluing{\pm}$ has a degenerate Heisenberg vertex subalgebra generated by a field $A^\pm(z)=H^\pm(z)+s^\pm(z)$. 
It follows from \eqref{support of W} and Proposition \ref{prop:CFL} that
\begin{align}\label{eq: shape of relcoh}
\relcoh{0}(\Wpm\otimes \gluing{\pm})\simeq \bigoplus_{\nu\in\Z}\Omega_\mu^{H^\pm}(\Wpm)\otimes \pi^{t^\pm}_\mu
\end{align}
as $\mathrm{Com}(\pi^{H^\pm},\Wpm)\otimes \pi^{t^\pm}$-modules. Since $\pi^{t^\pm}\simeq \pi^{H^\mp}$ by $t^\pm(z)\mapsto H^\mp(z)$, we have 
$\relcoh{0}(\Wpm\otimes \gluing{\pm})\simeq \Wmp$
as $\mathrm{Com}(\pi^{H^\mp},\Wmp)\otimes \pi^{H^\mp}$-modules by \eqref{FS_duality}.
Indeed, this is an isomorphism of vertex superalgebra through \eqref{KS_Duality} as follows:
\begin{proposition}\label{Comparison_VA}
The natural embedding 
$\Wpm\otimes\mathcal{V}^{\pm}\hookrightarrow 
C^\bullet_\infty(\Wpm\otimes\gluing{\pm})$ induces an isomorphism of vertex superalgebras
\begin{align*}
\eta_{\pm}\colon
\Com(\pi^{\widetilde{H}^\pm},\Wpm\otimes\mathcal{V}^{\pm})\xrightarrow{\simeq}
\relcoh{0}(\Wpm\otimes\gluing{\pm}).
\end{align*}
\begin{proof}
We show the case $\eta_+$.
It follows from 
$$\Wplus\otimes \mathcal{V}^+\subset \Wplus\otimes \mathcal{V}^+\otimes \pi^{\boson{-}}\otimes \mathcal{E}=C^\bullet_\infty(\Wplus\otimes \gluing{+})$$
and $A^+(z)=-(\widetilde{H}^+(z)+\epsilon^{-1}\boson{-}(z))$ that the embedding is restricted to  
$$\Com(\pi^{\widetilde{H}^+},\Wplus\otimes\mathcal{V}^{+})\hookrightarrow C^{\textsf{rel},\bullet}_\infty(\Wplus\otimes\gluing{+})$$
and that $\Com(\pi^{\widetilde{H}^+},\Wplus\otimes\mathcal{V}^+)$ is annihilated by the differential $d_{(0)}$. Thus, the embedding induces a homomorphism of vertex superalgebras
\begin{align*}
\eta_{+}\colon
\Com(\pi^{\widetilde{H}^+},\Wplus\otimes\mathcal{V}^+)\rightarrow
\relcoh{0}(\Wplus\otimes\gluing{+}).
\end{align*}
Under the isomorphism \eqref{eq: shape of relcoh}, $\eta_+$ is the identity on $\Omega_{\mu}^{H^+}(\Wplus)$ for each $\mu\in \Z$ and sends $H^-(z)$ to $t^+(z)$. Hence $\eta_+$ is an isomorphism. The case $\eta_-$ can be shown in the same way. This completes the proof.
\end{proof}
\end{proposition}
Next, we consider the relative semi-infinite cohomology for arbitrary $\Wpm$-modules in $\wtcat{\pm}$. For this purpose, we replace the vertex superalgebras $\gluing{\pm}$ with their modules:
\begin{align*}
\gluing{+}_\lambda = V_{\Z} \otimes \pi^{\boson{-}}_{\lambda},\quad
\gluing{-}_\lambda =V_{\ssqrt{-1}\Z}\otimes \pi^{\boson{+}}_{\lambda},\quad (\lambda\in\C).
\end{align*}
As a $\pi^{s^\pm}\otimes \pi^{t^\pm}$-module, $\gluing{\pm}_\lambda$ decomposes into 
$$\gluing{+}_\lambda\simeq \bigoplus_{\mu\in\Z}\pi^{s^+}_{-\mu+\epsilon^{-1}\lambda}\otimes \pi^{t^+}_{\mu+\epsilon\lambda},\quad
\gluing{-}_\lambda\simeq\bigoplus_{\mu\in\Z} \pi^{s^-}_{-\mu-\epsilon\lambda}\otimes \pi^{t^-}_{\mu+\epsilon^{-1}\lambda},$$
respectively. Then for a $\Wpm$-module $M$ in $\wtcat{\pm}$, the tensor product $M\otimes \gluing{\pm}_\lambda$ is a module of a degenerate Heisenberg vertex algebra $\pi^{A^\pm}$ generated by $A^\pm(z)=H^\pm(z)+s^\pm(z)$. The relative semi-infinite cohomology 
$H^{{\sf rel},p}_{\infty}(M\otimes \gluing{\pm}_\lambda)$ is naturally a $\Wmp$-module by the isomorphism
$$\Upsilon_\pm:=\eta_\pm\circ {\sf KS}_\pm \colon \Wmp\xrightarrow{\simeq} H^{{\sf rel},0}_{\infty}(\Wpm\otimes \gluing{\pm}_\lambda)$$
by \eqref{KS_Duality} and Proposition \ref{Comparison_VA}. 
non-zero only for $p=0$ by Proposition \ref{prop:CFL}. Therefore, this assignment induce exact $\C$-linear superfunctors
\begin{align*}
\Relcoh_{+,\lambda}\colon\wtcat{+}\to\wtcat{-},\quad
\Relcoh_{-,\lambda}\colon\wtcat{-}\to\wtcat{+},
\end{align*}
respectively. More precisely, for objects $M_+$ in $\wtcat{+}_{[\lambda]}$ and $M_-$ in $\wtcat{-}_{[\epsilon^2\lambda]}$ and $\xi\in \lambda+\Z$, we have 
$$\Relcoh_{+,\epsilon\xi}(M_+)\simeq \bigoplus_{\mu\in\Z}\Omega_{\xi+\mu}^{H^+}(M_+)\otimes\pi^{H^-}_{\epsilon^2\xi+\mu},\quad \Relcoh_{-,\epsilon\xi}(M_-)\simeq \bigoplus_{\mu\in\Z} \Omega_{\epsilon^2\xi+\mu}^{H^-}(M_-)\otimes\pi^{H^+}_{\xi+\mu}$$
as $\mathrm{Com}(\pi^{\widetilde{H}^\mp},\Wmp)\otimes \pi^{H^\mp}$-modules by Proposition \ref{prop:CFL}. It follows from \eqref{coset functor} that these decomposition coincide with $\coset^+_{-\xi}(M_+)$ and $\coset^-_{\epsilon^2\xi}(M_-)$, respectively. 
The following can be proven in the same way as Proposition \ref{Comparison_VA}.
\begin{theorem}\label{Comparison_Module}
For objects $M_+$ in $\wtcat{1}_{[\lambda]}$ and $M_-$ in $\wtcat{2}_{[\epsilon^2\lambda]}$,
the natural embeddings
$M_\pm\otimes \mathcal{V}^\pm\hookrightarrow C^\bullet_\infty(M\otimes\gluing{\pm}_{\epsilon\lambda})$ induce isomorphisms 
\begin{equation}\label{Natural_Isom_eta}
\eta_{+,\lambda}^M\colon\coset^+_{-\lambda}(M)
\xrightarrow{\simeq}\Relcoh_{+,\epsilon\lambda}(M),\quad
\eta_{-,\lambda}^M\colon\coset^-_{\epsilon^2\lambda}(M)
\xrightarrow{\simeq}\Relcoh_{-,\epsilon\lambda}(M)
\end{equation}
of $\Wmp$-modules, respectively. Moreover, the family $\{\eta_{\pm,\lambda}^\bullet\}$ gives natural isomorphisms of superfunctors 
\begin{align*}
&\coset^{+}_{-\lambda}\xrightarrow{\simeq}\Relcoh_{+,\epsilon\lambda}\colon \wtcat{+}_{[\lambda]}\rightarrow \wtcat{-}_{[\epsilon^2\lambda]},\\
&\coset^{-}_{\epsilon^2\lambda}
\xrightarrow{\simeq}\Relcoh_{-,\epsilon\lambda}\colon \wtcat{-}_{[\epsilon^2\lambda]}\rightarrow \wtcat{+}_{[\lambda]}.
\end{align*}
\end{theorem}
By composing the natural isomorphisms in \eqref{Natural_Isom_g} and Theorem \ref{Comparison_Module}, we obtain natural isomorphisms
\begin{align}\label{id to relcoh}
\Upsilon_{\pm,\lambda}^M \colon M\xrightarrow{\simeq}\Relcoh_{\mp,\epsilon\lambda}\circ\Relcoh_{\pm,\epsilon\lambda}(M)
\end{align}
for each object $M$ in $\wtcat{\pm}$.
More explicitly, first note that $\Relcoh_{\mp,\epsilon\lambda}\circ\Relcoh_{\pm,\epsilon\lambda}(M)$ is a subquotient of the total complex:
\begin{align}\label{total complex}
(M\otimes V_\Z\otimes \pi^{\boson{-}}_{\epsilon\lambda}\otimes \mathcal{E})\otimes V_{\ssqrt{-1}\Z}\otimes \pi^{\boson{+}}_{\epsilon\lambda}\otimes \mathcal{E}\simeq (M\otimes V_\Z\otimes V_{\ssqrt{-1}\Z})\otimes \pi^{\boson{+}}_{\epsilon \lambda}\otimes \pi^{\boson{-}}_{\epsilon\lambda}\otimes \mathcal{E}^{\otimes2}.    
\end{align}
Then $\Upsilon_{\pm,\lambda}^M$ is given by
\begin{align}\label{explicit form of upsilon}
\begin{split}
\Upsilon_{\pm,\lambda}^M(w)
&=\left[g_{+,\lambda}^M(w)\otimes \ket{\epsilon\lambda\boson{+}}\otimes \ket{-\epsilon\lambda\boson{-}}\otimes \mathbf{1}\right]\\
&=\left[\left(\exp(\mathcal{H}_\pm)(w)\otimes\ket{\mu\boson{+}_L}\otimes\ket{\mu\boson{-}_L} \right)\otimes \ket{\epsilon\lambda\boson{+}_R}\otimes \ket{-\epsilon\lambda\boson{-}_R}\otimes \mathbf{1}\right]
\end{split}
\end{align}
for $w\in M_{\lambda+\mu}$ ($\mu\in\Z$). 
Here we have replaced the notation $\boson{\pm}$ with $\boson{\pm}_L$ (reps.\ $\boson{\pm}_R$) standing for the elements in $V_\Z$ or $V_{\ssqrt{-1}\Z}$ (resp.\ $\pi^{\boson{+}}_{\epsilon \lambda}$ or $\pi^{\boson{-}}_{\epsilon\lambda}$) for clarity.
\subsection{Relation to fusion product}
Here we prove that the superfunctors $\Relcoh_{\pm,\lambda}$ respect the fusion products of the module categories $\wtcat{\pm}$.
More strongly, the relative semi-infinite cohomology functors induce isomorphisms between the spaces of (logarithmic) intertwining operators.

Let $\lambda_i\in\C$ ($i=1,2,3$) with $\lambda_3=\lambda_1+\lambda_2$ and $Y^{\pm}(\cdot,z)$ denote the $\pi^{\boson{\pm}}$-intertwining operator
$$Y^\pm(\cdot,z)\colon \pi^{\boson{\pm}}_{\epsilon\lambda_1}\otimes  \pi^{\boson{\pm}}_{\epsilon\lambda_2}\rightarrow \pi^{\boson{\pm}}_{\epsilon\lambda_3}\{z\}$$
satisfying
\begin{align*}
Y^{\pm}\left(\ket{\pm\epsilon\lambda_1\boson{\pm}},z\right)\ket{\pm\epsilon\lambda_2\boson{\pm}}=z^{\pm\epsilon^2\lambda_1\lambda_2}\difexp{\pm\epsilon\lambda_1\phi^\pm}\ket{\pm\epsilon\lambda_3\boson{\pm}},
\end{align*}
see \eqref{eq:difexp} for the definition of $E(\cdot,z)$. Take a $\Wplus$-module $M_i$ in $\wtcat{+}_{[\lambda_i]}$ for each $i=1,2,3$ and a $\Wplus$-intertwining operator $\mathcal{Y}(\cdot,z)$ of type $\binom{M_3}{M_1\ M_2}$.
Then the product 
$$\mathcal{Y}^{\sf rel}(\cdot,z):=\mathcal{Y}(\cdot,z)\otimes Y_{V_\Z}(\cdot,z)\otimes Y^-(\cdot,z)\otimes Y_{\mathcal{E}}(\cdot,z)$$
is a 
$C^\bullet_\infty(\Wplus\otimes\gluing{+})$-intertwining operator $$C^\bullet_\infty(M_1\otimes\gluing{+}_{\epsilon\lambda_1})\otimes C^\bullet_\infty(M_2\otimes\gluing{+}_{\epsilon\lambda_2})\rightarrow C^\bullet_\infty(M_3\otimes\gluing{+}_{\epsilon\lambda_3})\{z\}[\log z].$$
Recall that by \eqref{Jacobi for intertwining op} an intertwining operator $\mathcal{Y}$ satisfies the derivation property:
\begin{align*}
\mathcal{Y}(a_{(0)}m,z)=[a_{(0)},\mathcal{Y}(m,z)].
\end{align*}
It follows that $\mathcal{Y}^{\sf rel}(\cdot,z)$ can be restricted to
$$C^{\mathrm{\textsf{rel}},\bullet}_\infty(M_1\otimes\gluing{+}_{\epsilon\lambda_1})\otimes C^{\mathrm{\textsf{rel}},\bullet}_\infty(M_2\otimes\gluing{+}_{\epsilon\lambda_2})\rightarrow C^{\mathrm{\textsf{rel}},\bullet}_\infty(M_3\otimes\gluing{+}_{\epsilon\lambda_3})\{z\}[\log z]$$
and induces a $\Wminus$-intertwining operator
$$\mathbf{H}_{+}(\mathcal{Y})(\cdot,z)\colon \Relcoh_{+,\epsilon\lambda_1}(M_1)\otimes \Relcoh_{+,\epsilon\lambda_2}(M_2)\rightarrow \Relcoh_{+,\epsilon\lambda_3}(M_3)\{z\}[\log z].$$
Similarly, by using $Y^+(\cdot,z)$, we can construct from a $\Wminus$-intertwining operator $\mathcal{Y}(\cdot,z)$, a $\Wplus$-intertwining operator
$$\mathbf{H}_{-}(\mathcal{Y})(\cdot,z)\colon \Relcoh_{-,\epsilon\lambda_1}(M_1)\otimes \Relcoh_{-,\epsilon\lambda_2}(M_2)\rightarrow \Relcoh_{-,\epsilon\lambda_3}(M_3)\{z\}[\log z],$$
where $M_i$ ($i=1,2,3$) denotes a $\Wminus$-module in $\wtcat{-}_{[\epsilon^2\lambda_i]}$ by abuse of notation.
\begin{theorem}\label{isom for intertwining operators}
The linear maps
\begin{align*}
\mathbf{H}_{\pm}\colon I_{\Wpm}\binom{M_3}{M_1\ M_2}\rightarrow I_{\Wmp}\binom{\Relcoh_{\pm,\epsilon\lambda_3}(M_3)}{\Relcoh_{\pm,\epsilon\lambda_1}(M_1)\ \Relcoh_{\pm,\epsilon\lambda_2}(M_2)}
\end{align*}
are isomorphisms of vector superspaces.
\end{theorem}
\begin{proof}
Let us abbreviate the compositions $\mathbf{H}_{\pm\circ \mp}=\mathbf{H}_{\pm}\circ \mathbf{H}_{\mp}$ and $\Relcoh_{\pm\circ \mp}=\Relcoh_{\pm,\epsilon\lambda_p}\circ\Relcoh_{\mp,\epsilon\lambda_p}$.
Then to show that $\mathbf{H}_\pm$ are isomorphisms, it suffices to show that the compositions
\begin{align*}
I_{\Wpm}\binom{M_3}{M_1 M_2}
&\xrightarrow{\mathbf{H}_{\mp\circ \pm}}
I_{\Wpm} \binom{\Relcoh_{\mp\circ \pm}(M_3)}{\Relcoh_{\mp\circ \pm}(M_1)\ \Relcoh_{\mp\circ \pm}(M_2)}
\xrightarrow{\simeq}I_{\Wpm}\binom{M_3}{M_1\ M_2}
\end{align*}
are the identity, where the last isomorphism is induced by \eqref{id to relcoh}.
We show the claim for $\mathbf{H}_{-\circ +}$.
By taking a $\Wplus$-intertwining operator $\mathcal{Y}(\cdot,z)$ of type $\binom{M_3}{M_1\ M_2}$, we need to show that the diagram
\begin{equation}\label{diagram to check}
\SelectTips{cm}{}
\xymatrix@W20pt@H11pt@R26pt@C48pt{
M_1\otimes M_2 \ar[r]^-{\mathcal{Y}(\cdot,z)}\ar[d]_-{\Upsilon^{M_1}_{+,\lambda_1}\otimes \Upsilon^{M_2}_{+,\lambda_2}}
&M_3\{z\}[\log z]\ar[d]_-{\Upsilon^{M_3}_{+,\lambda_3}}\\
\Relcoh_{-\circ +}(M_1)\otimes \Relcoh_{-\circ +}(M_2)\ar[r]^-{\mathbf{H}_{-\circ+}(\mathcal{Y})(\cdot,z)}
&\Relcoh_{-\circ +}(M_3)\{z\}[\log z]}
\end{equation}
commutes. By using the decomposition \eqref{decomposition of module} of modules $M_i$ ($i=1,2$), we may take $w_i\in\Omega_{\lambda_i+\mu_i}^{H^+}(M_i)\otimes\pi^{H^+}_{\lambda_i+\mu_i}$ for some $\mu_i\in\Z$ without loss of generality. It follows from \eqref{explicit form of upsilon} that 
\begin{align*}
&\mathbf{H}_{-\circ+}(\mathcal{Y})\left(\Upsilon^{M_1}_{+,\lambda_1}(w_1),z)\right) \Upsilon^{M_2}_{+,\lambda_2}(w_2) \nonumber\\
&=\mathbf{H}_{-\circ+}(\mathcal{Y})
\left(
\Bigl[w_1\otimes\ket{\mu_1\boson{+}_L}\otimes\ket{\mu_1\boson{-}_L} \otimes \ket{\epsilon\lambda_1\boson{+}_R}\otimes \ket{-\epsilon_1\lambda\boson{-}_R}\otimes \mathbf{1}\Bigr]
,z
\right) \\
&\hspace{3cm}
\Bigl[w_2\otimes\ket{\mu_2\boson{+}_L}\otimes\ket{\mu_2\boson{-}_L} \otimes \ket{\epsilon\lambda_2\boson{+}_R}\otimes \ket{-\epsilon\lambda_2\boson{-}_R}\otimes \mathbf{1}\Bigr] \\
&=\Bigl[\mathcal{Y}(w_1,z)w_2\otimes Y_{V_\Z} (\ket{\mu_1\boson{+}_L},z)\ket{\mu_2\boson{+}_L}\otimes
Y_{V_{\ssqrt{-1}\Z}}(\ket{\mu_1\boson{-}_L},z)\ket{\mu_2\boson{-}_L} \\
&\hspace{3cm} Y^+(\ket{\epsilon\lambda_1\boson{+}_R},z)\ket{\epsilon\lambda_2\boson{+}_R}\otimes
Y^-(-\epsilon\lambda_1\boson{-}_R,z)\ket{-\epsilon\lambda_2\boson{-}_R}\otimes \mathbf{1}\Bigr] \\
&=
\Bigl[\mathcal{Y}(w_1,z)w_2\otimes \difexp{\mu_1(\boson{+}_L+\boson{-}_L)}\ket{\mu_3\boson{+}_L}\otimes \ket{\mu_3\boson{-}_L} \\
\label{}&\hspace{3cm} \difexp{\epsilon\lambda_1(\boson{+}_R-\boson{-}_R)}\ket{\epsilon\lambda_3\boson{+}_R}\otimes\ket{-\epsilon\lambda_3\boson{-}_R}\otimes \mathbf{1}\Bigr]
\end{align*}
where $\mu_3=\mu_1+\mu_2$.
By Proposition \ref{prop:CFL}, the Heisenberg field
$$t^+(z)+s^-(z)=(\boson{+}_L+\boson{-}_L)(z)-\epsilon(\boson{+}_R+\boson{-}_R)(z)$$
in \eqref{total complex} is a coboundary. Hence, the right-hand side in the last equality is equal to 
\begin{align*}
&\Bigl[\mathcal{Y}(w_1,z)w_2\otimes \difexp{(\lambda_1+\mu_1)(\boson{+}_L+\boson{-}_L)}
\ket{\mu_3\boson{+}_L}\otimes\ket{\mu_3\boson{-}_L}
\otimes\ket{\epsilon\lambda_3\boson{+}_R}\otimes\ket{-\epsilon\lambda_3\boson{-}_R}\Bigr].
\end{align*}
Then by decomposing the intertwining operator $\mathcal{Y}(\cdot,z)$ into a $\mathrm{Com}(\pi^{H^+},\Wplus)\otimes \pi^{H^+}$-intertwining operator, the argument in the proof of Lemma \ref{Twist_Emb_Mod} implies that the last cohomology class is equal to 
\begin{align*}
&\Bigl[\left(\exp(\mathcal{H}_+)\mathcal{Y}(w_1,z)w_2\otimes 
\ket{\mu_3\boson{+}_L}\otimes\ket{\mu_3\boson{-}_L}\right)
\otimes\ket{\epsilon\lambda_3\boson{+}_R}\otimes\ket{-\epsilon\lambda_3\boson{-}_R}\Bigr]\\
&=\Upsilon^{M_3}_{+,\lambda_3}\left(\mathcal{Y}(w_1,z)w_2\right).
\end{align*}
Therefore, we obtain the commutativity of the diagram \eqref{diagram to check}. The claim for $\mathbf{H}_{+\circ -}$ can be shown in the same way. This completes the proof.
\end{proof}

Suppose that there exist full subcategories $$\mathcal{S}^\pm=\bigoplus_{[\lambda]\in \C/\Z}\mathcal{S}^\pm_{[\lambda]}\subset\wtcat{\pm}$$ 
which have the vertex tensor supercategory structure \cite{HLZ1,CKM} and that the relative semi-infinite cohomology functors give an equivalence of categories 
\begin{equation}\label{equiv of cat}
    \Relcoh_{+,\epsilon\lambda}\colon \mathcal{S}^+_{[\lambda]}\xrightarrow{\simeq} \mathcal{S}^-_{[\epsilon^2\lambda]},\quad \Relcoh_{-,\epsilon\lambda}\colon \mathcal{S}^-_{[\epsilon^2\lambda]}\xrightarrow{\simeq} \mathcal{S}^+_{[\lambda]},\quad (\lambda\in\C).
\end{equation}
\begin{corollary}\label{SIcoh_Fusion}
For objects $M_i^\pm$ ($i=1,2$) in $\wtcat{+}_{[\lambda_i]}$ (resp.\ $\wtcat{-}_{[\epsilon^2\lambda_i]}$), 
the pair 
$(\Relcoh_{\pm,\epsilon(\lambda_1+\lambda_2)}(M_1^\pm\boxtimes M_2^\pm),\ \mathbf{H}_{\pm}(\mathcal{Y}_{M_1^\pm\boxtimes M_2^\pm}))$
induced by the canonical intertwining operator $\mathcal{Y}_{M_1^\pm\boxtimes M_2^\pm}(\cdot,z)$ of $M_1^\pm$ and $M_2^\pm$,
is the fusion product of $\Relcoh_{\pm,\epsilon\lambda_1}(M_1^\pm)$ and $\Relcoh_{\pm,\epsilon\lambda_2}(M_2^\pm)$. 
In particular, we have an isomorphism
\begin{equation}\label{tensor isom}
\Psi_{M_1,M_2}^\pm\colon \Relcoh_{\pm,\epsilon\lambda_1}(M_1^\pm)\boxtimes\Relcoh_{\pm,\epsilon\lambda_2}(M_2^\pm)\simeq\Relcoh_{\pm,\epsilon(\lambda_1+\lambda_2)}(M_1^\pm\boxtimes M_2^\pm).
\end{equation}
\end{corollary}
\proof
Since the proofs for $(M_1^\pm,M_2^\pm)$ are quite similar, we only show the case $(M_1^+,M_2^+)$. By the definition of fusion product, it suffices to show that given an arbitrary object $N$ in $\mathcal{S}^-$ and an intertwining operator 
\begin{align}\label{test intertwiner}
I(\cdot,z)\colon \Relcoh_{+,\epsilon\lambda_1}(M_1^+)\otimes \Relcoh_{+,\epsilon\lambda_2}(M_2^+)\rightarrow N\{z\}[\log z],
\end{align}
there exists a unique $\Wminus$-homomorphism 
$I^\sharp\colon \Relcoh_{+,\epsilon\lambda_1}(M_1^+)\boxtimes\Relcoh_{+,\epsilon\lambda_2}(M_2^+)\rightarrow N$ such that $I(\cdot,z)=I^\sharp\circ \mathbf{H}_{+}(\mathcal{Y}_{M_1^+\boxtimes M_2^+})$. For this, we may assume $N$ is an object in $\mathcal{S}^-_{[\epsilon^2]\lambda_3]}$ with $\lambda_3=\lambda_1+\lambda_2$. Then by \eqref{equiv of cat}, there exists an object $\widetilde{N}$ in $\mathcal{S}^+_{[\lambda_3]}$ such that $\Relcoh_{+,\epsilon\lambda_3}(\widetilde{N})\simeq N$. Then by Theorem \ref{isom for intertwining operators}, there exists a unique intertwining operator 
$$\widetilde{I}(\cdot,z)\colon M_1^+\otimes M_2^+\rightarrow \widetilde{N}\{z\}[\log z]$$
corresponding to \eqref{test intertwiner}. Hence there exists a unique $\Wplus$-homomorphism $\widetilde{I}^\sharp\colon M_1^+\boxtimes M_2^+\rightarrow \widetilde{N}$ such that $\widetilde{I}(\cdot,z)=\widetilde{I}^\sharp\circ \mathcal{Y}_{M_1^+\boxtimes M_2^+}(\cdot,z)$. Therefore, $I^\sharp:=\mathbf{H}_+(\widetilde{I}^\sharp)$ is the unique $\Wminus$-homomorphism satisfying $I(\cdot,z)=I^\sharp\circ \mathbf{H}_{+}(\mathcal{Y}_{M_1^+\boxtimes M_2^+})(\cdot,z)$ by \eqref{equiv of cat}.
\endproof
\begin{remark}
By using the associators of vertex tensor categories, one can prove that the triple of families of isomorphisms
\begin{itemize}
    \item $\Relcoh_{+,\epsilon\lambda}\colon \mathcal{S}^+_{[\lambda]}\xrightarrow{\simeq},  \mathcal{S}^-_{[\epsilon^2\lambda]}$,
    \item 
    $\Psi_{M,N}^+\colon \Relcoh_{+,\epsilon\lambda}(M)\boxtimes\Relcoh_{+,\epsilon\mu}(N)\simeq\Relcoh_{+,\epsilon(\lambda+\mu)}(M\boxtimes N)$
    \item 
    $\Upsilon^+\colon \Wminus\xrightarrow{\simeq} \Relcoh_{+,0}(\Wplus)$
\end{itemize}
satisfy the same axiom of tensor superfunctors, that is, the following commutative diagrams:
\begin{equation*}
\SelectTips{cm}{}
\xymatrix@W20pt@H11pt@R24pt@C36pt{
\left(\Relcoh_{+,\lambda}(M)\boxtimes\Relcoh_{+,\mu}(N)\right)\boxtimes\Relcoh_{+,\nu}(L)\ar[r]^-{\mathcal{A}^2}\ar[d]_-{\Psi^+_{M\boxtimes N}\,\boxtimes\,\id}
&\Relcoh_{+,\nu}(M)\boxtimes\left(\Relcoh_{+,\mu}(N)\boxtimes\Relcoh_{+,\nu}(L)\right)\ar[d]_-{\id\,\boxtimes\,\Psi^+_{N,L}}\\
\Relcoh_{+,\lambda+\mu}(M\boxtimes N)\boxtimes\Relcoh_{+,\nu}(L)\ar[d]_-{\Psi^+_{M\boxtimes N,L}}
&\Relcoh_{+,\lambda}(M)\boxtimes\Relcoh_{+,\mu+\nu}(N\boxtimes L)\ar[d]_-{\Psi^+_{M,N\boxtimes L}}\\
\Relcoh_{+,\lambda+\mu+\nu}\left((M\boxtimes N)\boxtimes L\right)\ar[r]_-{\Relcoh_{+,\lambda+\mu+\nu}(\mathcal{A}^1)}
&\Relcoh_{+,\lambda+\mu+\nu}\left(M\boxtimes(N\boxtimes L)\right),}
\end{equation*}
\begin{equation*}
\SelectTips{cm}{}
\xymatrix@W20pt@H11pt@R24pt@C36pt{
\Wminus\boxtimes \Relcoh_{+,\lambda}(M)\ar[r]^-{\ell^-}\ar[d]_-{\Upsilon^+\,\boxtimes\,\id}
&\Relcoh_{+,\lambda}(M)\\
\Relcoh_{+,0}(\Wplus)\boxtimes \Relcoh_{+,\lambda}(M)\ar[r]^-{\Psi^+_{\Wplus,M}}
&\Relcoh_{+,\lambda}(\Wplus\boxtimes M)\ar[u]_-{\Relcoh_{+,\lambda}(\ell^+)},
}
\end{equation*}
\begin{equation*}
\SelectTips{cm}{}
\xymatrix@W20pt@H11pt@R24pt@C36pt{
\Relcoh_{+,\lambda}(M)\boxtimes\Wminus \ar[r]^-{r^-}\ar[d]_-{\id\,\boxtimes\,\Upsilon^+}
&\Relcoh_{+,\lambda}(\Wplus)\\
\Relcoh_{+,\lambda}(M)\boxtimes \Relcoh_{+,0}(\Wplus)\ar[r]^-{\Psi^+_{M,\Wplus}}
&\Relcoh_{+,\lambda}(M\boxtimes\Wplus)\ar[u]_-{\Relcoh_{+,\lambda}(r^+)},
}
\end{equation*}
where $\mathcal{A}^\pm$ is the associativity isomorphism, $\ell^\pm$ (resp.\ $r^\pm$) is the left (resp.\ right) unit isomorphism on $\mathcal{S}^\pm$. (See also Appendix \ref{BTC and Simple currents}).
Therefore, $(\Relcoh_{+,\epsilon\bullet},\Psi_{\bullet,\bullet}^+,\Upsilon^+)$ is a lax tensor superfunctor. In this point of view, the decomposition $\mathcal{S}^+=\bigoplus_{[\lambda]\in \C/\Z}\mathcal{S}^+_{[\lambda]}$, which a priori is defined by the eigenvalues for $H^+_{(0)}$, should be interpreted as the decomposition in terms of the tensor autofunctor $\mathcal{G}=\mathrm{exp}(2\pi \ssqrt{-1}H^+_{(0)})$. Similarly, $(\Relcoh_{-,\epsilon\bullet},\Psi_{\bullet,\bullet}^-,\Upsilon^-)$ forms a lax tensor superfunctor in the same sense and these two functors are weight-wise quasi-inverse to each other. However, these two functors do not preserve the braided structure since we use Fock modules $\pi^{\boson{\pm}}_\lambda$ to define $\Relcoh_{\pm,\lambda}$, which give a nontrivial additional braiding themselves.
\end{remark}
\subsection{Rational case}
We present Corollary \ref{SIcoh_Fusion} more explicitly in the rational cases treated in \S\ref{SuperW_rational}.
Recalling the notation $\subW$ and $\prinsW$, we have
\begin{equation*}
\varepsilon_+=\frac{r}{n},\quad\varepsilon_-=\frac{r}{n+r},\quad\epsilon=\sqrt{\frac{\varepsilon_-}{\varepsilon_+}}=\sqrt{\frac{n}{n+r}}
\end{equation*}
and obtain the following proposition by comparing the branching rule.
\begin{proposition}\label{explicit form} Let $\xi\in\C$ and $\lambda\in\widehat{P}_+^n(r)$.\\
(1)  For $a\in\Z_{nr}$, we have
\begin{equation*}
\Relcoh_{+,\frac{\xi}{\sqrt{n(n+r)}}}\left(\sbmod(\lambda,a)\right)\simeq
\begin{cases}
\spmod\big(\lambda,\frac{(n+r)a-r\xi}{n}\big) & \text{ if }\xi\equiv a\text{ mod }n\Z,\\
0 & \text{ otherwise.}
\end{cases}
\end{equation*}\\
(2) For $a\in\Z_{(n+r)r}$, we have
\begin{equation*}
\Relcoh_{-,\frac{\xi}{\sqrt{n(n+r)}}}\left(\spmod(\lambda,a)\right)\simeq
\begin{cases}
\sbmod\big(\lambda,\frac{na+r\xi}{n+r}\big) &\text{ if }\xi\equiv a\text{ mod }(n+r)\Z,\\
0 & \text{ otherwise.}
\end{cases}
\end{equation*}
\end{proposition}
\begin{proof}
Since the proofs for (1) and (2) are similar, we only prove (i).
Since $\sbmod(\lambda,a)\otimes\widetilde{\mathcal{V}}^+_{\frac{\xi}{\sqrt{n(n+r)}}}$ decomposes into the direct sum
\begin{equation}
\bigoplus_{i\in\Z_r}\mathbf{L}_\W(\sigma^{i}(\lambda))\otimes\bigoplus_{m\in \Z}\pi^{\Hsub}_{\frac{a}{n}+i+rm}\otimes\bigoplus_{\mu\in \Z}\pi^{s^+}_{-\frac{\xi}{n}-\mu} \otimes \pi^{t^+}_{\frac{\xi}{n+r}+\mu},
\end{equation}
we obtain by Proposition \ref{prop:CFL}
\begin{equation*}
    \Relcoh_{+,\frac{\xi}{\sqrt{n(n+r)}}}\left(\sbmod(\lambda,a)\right)
    \simeq
    \begin{cases}
    \displaystyle
    \bigoplus_{i\in\Z_r}\mathbf{L}_\W(\sigma^{i}(\lambda))\otimes\bigoplus_{m\in \Z} \pi^{H^-}_{\frac{\xi}{n+r}+\frac{a-\xi}{n}+i+rm} & \text{ if }\xi\equiv a\text{ mod }n\Z,\\
    0 & \text{ otherwise}
    \end{cases}
\end{equation*}
as $\prinW\otimes\pi^{H^-}$-modules.
This decomposition and Theorem \ref{fusion rules of superW} imply (1).
\end{proof}
The following is Corollary \ref{SIcoh_Fusion} in the rational case, which can be deduced directly from the above proposition and the fusion rules obtained in \S\ref{SuperW_rational}.
\begin{theorem}\label{check of monoidality in the rational case} Let $\xi_1,\xi_2\in\C$ and $\lambda_1,\lambda_2\in\widehat{P}_+^n(r)$.\\
(1)
For $a_1,a_2\in\Z_{nr}$, we have
\begin{align*}
&\Relcoh_{+,\frac{\xi_1}{\sqrt{n(n+r)}}}\left(\sbmod(\lambda_1,a_1)\right)\boxtimes \Relcoh_{+,\frac{\xi_2}{\sqrt{n(n+r)}}}\left(\sbmod(\lambda_2,a_2)\right)\\
&\simeq
\begin{cases}
\Relcoh_{+,\frac{\xi_1+\xi_2}{\sqrt{n(n+r)}}}\left(\sbmod(\lambda_1,a_1)\boxtimes\sbmod(\lambda_2,a_2)\right)
& \text{ if }\xi_i\equiv a_i\text{ mod }n\Z\text{ for }i=1,2,\\
0 & \text{ otherwise.}
\end{cases}
\end{align*}\\
(2)
For $a_1,a_2\in\Z_{(n+r)r}$, we have
\begin{align*}
&\Relcoh_{-,\frac{\xi_1}{\sqrt{n(n+r)}}}\left(\spmod(\lambda_1,a_1)\right)\boxtimes \Relcoh_{-,\frac{\xi_2}{\sqrt{n(n+r)}}}\left(\spmod(\lambda_2,a_2)\right)\\
&\simeq
\begin{cases}
\Relcoh_{-,\frac{\xi_1+\xi_2}{\sqrt{n(n+r)}}}\left(\spmod(\lambda_1,a_1)\boxtimes\spmod(\lambda_2,a_2)\right) & \text{ if }\xi_i\equiv a_i\text{ mod }(n+r)\Z\text{ for }i=1,2,\\
0 & \text{ otherwise.}
\end{cases}
\end{align*}
\end{theorem}
\appendix

\section{Categorical aspects of simple currents}\label{Sec. BTC and Simple currents}
We consider the theory of simple currents of vertex operator algebras purely in the categorical manner, following \cite{CKL,CKM}.
We refer to a category enriched by the category of $\Z_2$-graded sets as a supercategory.
For a supercategory $\mathcal{C}$, we write $\underline{\mathcal{C}}$ for its underlying category, namely, the objects in $\underline{\mathcal{C}}$ are the same as $\mathcal{C}$ and the morphisms in $\underline{\mathcal{C}}$ are the even morphisms of $\mathcal{C}$. 
For an additive supercategory $\mathcal{C}$ such that $\underline{\mathcal{C}}$ is an \emph{abelian category}, by a subquotient object in $\mathcal{C}$ we mean one in the underlying category $\underline{\mathcal{C}}$. We use the notion of simplicity as well. 
We stress that we use this terminology even if $\mathcal{C}$ is an \emph{abelian supercategory}.
\subsection{Preliminaries}\label{BTC and Simple currents}
Let $\mathcal{C}$ be an essentially small, $\C$-linear, monoidal supercategory whose underlying category is abelian. We write 
$$\boxtimes\colon \mathcal{C}\times \mathcal{C}\rightarrow \mathcal{C},\quad (M,N)\mapsto M\boxtimes N$$
for the monoidal product superfunctor with unit object $1_\mathcal{C}$, by 
\begin{align*}
&l_{\bullet}\colon 1_\mathcal{C} \operatorname{\boxtimes} \bullet\xrightarrow{\simeq} \bullet,\quad (l_M\colon 1_\mathcal{C}\boxtimes M\xrightarrow{\simeq} M),\\
&r_{\bullet}\colon \bullet \operatorname{\boxtimes} 1_\mathcal{C}\xrightarrow{\simeq} \bullet,\quad (r_M\colon M\boxtimes  1_\mathcal{C}\xrightarrow{\simeq} M),\\
&\mathcal{A}_{\bullet,\bullet,\bullet}\colon \bullet \boxtimes(\bullet\boxtimes \bullet)\xrightarrow{\simeq} (\bullet\boxtimes \bullet)\boxtimes \bullet,\quad \left(\mathcal{A}_{M,N,L}\colon M\boxtimes (N\boxtimes L)\xrightarrow{\simeq} (M\boxtimes N)\boxtimes L\right),
\end{align*}
the structural natural isomorphisms of superfunctors satisfying the pentagon and triangle axioms, see \cite{BK,EGNO}. The last one is called the associativity isomorphism.
Here we use the convention that the parity of a parity-homogeneous morphism $f$ is denoted by $\bar{f}$ and the composition of two parity-homogeneous morphisms $(f_1,f_2)$, $(g_1,g_2)$ in $\mathcal{C}\times\mathcal{C}$ is given by $(-1)^{\bar{f}_2\bar{g}_1}(f_1g_1,f_2g_2)$.

For an object $M\in Ob(\mathcal{C})$, a right dual of $M$ is a triple $(M^*,e_M^R,i_M^R)$ of $M^*\in Ob(\mathcal{C})$ and even morphisms
\begin{align}\label{eq: right dual}
e_M^R\colon M^*\boxtimes M\rightarrow 1_\mathcal{C},\quad  i_M^R\colon 1_\mathcal{C}\rightarrow M\boxtimes M^*,
\end{align}
satisfying the rigidity axioms. A left dual of $M$ is similarly defined to be a triple $({}^*\!M,e_M^L,i_M^L)$ consisting of ${}^*\!M\in Ob(\mathcal{C})$ and even morphisms
\begin{align}\label{eq: left dual}
e_M^L\colon M\boxtimes {}^*\!M\rightarrow 1_\mathcal{C},\quad i_M^L\colon 1_\mathcal{C}\rightarrow {}^*\!M\boxtimes M .
\end{align}
By \cite[Proposition 2.10.15]{EGNO}, right (left) duals are unique up to even isomorphisms if they exist. The supercategory $\mathcal{C}$ is called rigid if every object in $\mathcal{C}$ has a right and left dual. We call an object $M$ of $\mathcal{C}$ an \emph{invertible object} if it admits a left and right dual such that the morphisms in \eqref{eq: right dual} and \eqref{eq: left dual} are isomorphisms.
\begin{lemma}\label{lem: exactness of simple currents}
For an invertible object $S$, the superfunctors
\begin{align*}
&S\boxtimes\bullet\colon \mathcal{C}\rightarrow \mathcal{C},\quad M\mapsto S\boxtimes M,\\
&\bullet\operatorname{\boxtimes} S\colon \mathcal{C}\rightarrow \mathcal{C},\quad M\mapsto  M \boxtimes S,
\end{align*}
are exact.
\end{lemma}
\proof
Since the proofs for $S\boxtimes \bullet$ and $\bullet\boxtimes S$ are similar, we prove only for $S\boxtimes\bullet$.
Note that any object lies in the image of $S\boxtimes\bullet$ since for any object $M\in Ob(\mathcal{C})$,
$$S\boxtimes(S^*\boxtimes M)\simeq (S\boxtimes S^*)\boxtimes M\simeq 1_\mathcal{C}\boxtimes M\simeq M.$$
Now, take a short exact sequence in $\mathcal{C}$
\begin{align*}
0\rightarrow M\rightarrow N\rightarrow L\rightarrow0.
\end{align*}
We show that the complex
\begin{align}\label{original exactness}
0\rightarrow S\boxtimes M\rightarrow S \boxtimes N\rightarrow S\boxtimes L\rightarrow0
\end{align}
is exact. For the left exactness of \eqref{original exactness}, it suffices to show that 
for any object $A\in Ob(\mathcal{C})$, the induced complex
\begin{align}\label{first case}
0\rightarrow \Hom_\mathcal{C}(A,S\boxtimes M)\rightarrow \Hom_\mathcal{C}(A,S\boxtimes N)\rightarrow  \Hom_\mathcal{C}(A,S\boxtimes L)
\end{align}
is exact. 
Indeed, by the left exactness of $\Hom_{\mathcal{C}}(A,\bullet)$, we have the exact sequence 
\begin{align}\label{first case; 1}
0\rightarrow \Hom_{\mathcal{C}}(A,M)\rightarrow \Hom_{\mathcal{C}}(A,N)\rightarrow \Hom_{\mathcal{C}}(A,L).
\end{align}
Then, by the functoriality of the following isomorphisms
\begin{align}\label{eq: identification of hom space}
\begin{split}
\Hom_\mathcal{C}(A,M&)\simeq \Hom_\mathcal{C}(A,1_\mathcal{C}\boxtimes M)\simeq \Hom_\mathcal{C}(A,({}^*\!S\boxtimes S)\boxtimes M)\\
&\simeq  \Hom_\mathcal{C}(A,{}^*\!S\boxtimes (S\boxtimes M))\simeq \Hom_\mathcal{C}(S\boxtimes A,S\boxtimes M),
\end{split}
\end{align}
(see \cite[Proposition 2.10.8]{EGNO}), the exactness of \eqref{first case; 1} implies that of 
$$0\rightarrow \Hom_{\mathcal{C}}(S\boxtimes A,S\boxtimes M)\rightarrow \Hom_{\mathcal{C}}(S\boxtimes A,S\boxtimes N)\rightarrow \Hom_{\mathcal{C}}(S\boxtimes A,S\boxtimes L).$$
Replacing $A$ by $S^*\boxtimes A$ in $S\boxtimes A$, we conclude that \eqref{first case} is exact.
We can prove the right exactness of \eqref{original exactness} in a similar way by showing the exactness of 
$$0\rightarrow \Hom_\mathcal{C}(S\boxtimes L,A)\rightarrow \Hom_\mathcal{C}(S\boxtimes N,A)\rightarrow \Hom_\mathcal{C}(S\boxtimes M,A)$$
for any object $A\in Ob(\mathcal{C})$ and thus we omit it. This completes the proof.
\endproof
We call a simple invertible object a \emph{simple current} following the terminology of the theory of vertex algebra.
By Lemma \ref{lem: exactness of simple currents}, a simple current exists if and only if the unit object $1_\mathcal{C}$ is simple. 
\begin{assumption}\label{assumption1}
The unit object $1_\mathcal{C}$ is simple.
\end{assumption}

\begin{definition}\label{Def:Picard}
We call the group of isomorphism classes of simple currents in $\mathcal{C}$, denoted by $\Pic(\mathcal{C})$, the Picard group of $\mathcal{C}$.
\end{definition}
A braided monoidal supercategory is a monoidal supercategory $\mathcal{C}$ equipped with a natural isomorphism of superfunctors, called the \emph{braiding},
$$\mathcal{R}_{\bullet,\bullet}\colon(\bullet\boxtimes\bullet)\xrightarrow{\simeq} ( \bullet\boxtimes \bullet)\circ \sigma,\quad (\mathcal{R}_{M,N}\colon M\boxtimes N\xrightarrow{\simeq}N\boxtimes M),$$
where $\sigma$ is the superfunctor
$$\sigma\colon \mathcal{C}\times \mathcal{C}\rightarrow \mathcal{C}\times \mathcal{C},\quad (M,N)\mapsto (N,M),$$
under which every parity-homogeneous morphism $(f,g)$ maps to $(-1)^{\bar{f}\bar{g}}(g,f)$. The natural isomorphism $\mathcal{R}_{\bullet,\bullet}$ is required to satisfy the hexagon identity. The natural isomorphism
$$\mathcal{M}_{\bullet,\bullet}:=\mathcal{R}_{\bullet,\bullet}^2\colon (\bullet\boxtimes \bullet) \xrightarrow{\simeq}  (\bullet\boxtimes \bullet),\quad (\mathcal{M}_{M,N}\colon M\boxtimes N\xrightarrow{\simeq} M\boxtimes N)$$
is called the \emph{monodromy}.
It is straightforward to check that, in a braided monoidal supercategory, the existence of a right dual implies that of a left dual and vice versa. Indeed, given a right dual $(M^*,e_M^R,i_M^R)$ of an object $M$, the triple $(M^*,e^R_M\circ \mathcal{R}_{M,M^*}^{-1}, \mathcal{R}_{M,M^*}\circ i^R_M)$ defines a left dual of $M$ and conversely, given a left dual $({}^*\!M,e_M^L,i_M^L)$, a triple $({}^*\!M,e^L_M\circ  \mathcal{R}_{M,{}^*\!M}, \mathcal{R}_{M,{}^*\!M}^{-1}\circ i^L_M)$ defines a right dual. 

One of the simplest examples of braided monoidal supercategory is the supercategory $\mathrm{SVect}_{\C}$ of vector superspaces over $\C$ of at most countable dimension equipped with the tensor product $M\boxtimes N=M\otimes_\C N$
and the braiding
\begin{align*}
R_{M,N}\colon M\otimes_\C N\rightarrow N\otimes_\C M,\quad m\otimes n\mapsto (-1)^{\bar{n}\bar{m}}n\otimes m.
\end{align*}
Obviously, the supercategory $\mathcal{C}$ is $\C$-linear and its underlying category $\underline{\mathcal{C}}$ is abelian.
Note that the supercategory $\mathrm{SVect}_{\C}$ has admits a natural parity reversing endofunctor $\Pi$, which exchanges the parity of objects.

In this paper, we make the following assumption on our monoidal supercategory $\mathcal{C}$ in consideration:
\begin{assumption}\label{assumption2}
Every object in $\mathcal{C}$ has a structure of a $\C$-vector superspace of at most countable dimension, the forgetful functor $\mathcal{C}\rightarrow \mathrm{SVect}_\C$ is a $\C$-linear exact faithful superfunctor, and there exists an involutive autofunctor $\Pi_{\mathcal{C}}$ of $\mathcal{C}$ which coincides with the parity reversing autofunctor $\Pi$ of $\mathrm{SVect}_\C$ through the forgetful functor. In addition, each $\C$-vector superspace of morphisms in $\mathcal{C}$ has a finite dimension.
\end{assumption}
Such a situation naturally appears when we consider module categories of suitable superalgebras over $\C$ like vertex operator superalgebras or Hopf superalgebras consisting of modules of at most countable dimension. 
The existence of the forgetful functor implies Diximir's lemma (or Schur's lemma) for $\mathcal{C}$.
\begin{lemma}\label{categorical Schur}
A simple object $M\in Ob(\mathcal{C})$ satisfies $\mathrm{End}_\mathcal{C}(M)_{\bar{0}}=\C \id_M$ and  $\mathrm{End}_\mathcal{C}(M)_{\bar{1}}=0$ or $\C\Pi$
where $\Pi$ is a certain linear map satisfying $\Pi^2=\mathrm{id}_M$. 
\end{lemma}
\proof
The first assertion follows from the argument in the purely even case, see e.g. \cite[\S 0.5]{Wal}. For the second one, we suppose $\End_\mathcal{C}(M)_{\bar{1}}\neq 0$ and take a nonzero morphism $f\in\End_\mathcal{C}(M)_{\bar{1}}$. We prove $\mathrm{End}_\mathcal{C}(M)_{\bar{1}}=\C f$. Since $f^2$ is an even isomorphism, we may assume that $f^2=\mathrm{id}_\mathcal{C}$ by rescaling. Then it suffices to show $g=\pm f$ for $g\in \mathrm{End}_\mathcal{C}(M)_{\bar{1}}$ such that $g^2=\mathrm{id}_\mathcal{C}$. Let $\alpha\in\C$ be the scalar such that $fg=\alpha \mathrm{id}_\mathcal{C}$. Since $\mathrm{id}_\mathcal{C}=f(fg)g=\alpha fg=\alpha^2 \id_{\mathcal{C}}$, we obtain $\alpha=\pm1$. Then, by $fg=\pm\mathrm{id}_\mathcal{C}$ and $(fg)(gf)=\id_{\mathcal{C}}$, we have $fg=gf$. Now, we have $(f+g)(f-g)=f^2-g^2=0$, which implies that $f+g=0$ or $f-g=0$. This completes the proof.   
\endproof 
In practice, the case $\mathrm{End}_\mathcal{C}(M)_{\bar{1}}=\C \Pi$ occurs when $\mathcal{C}$ is a module category of a $\C$-superalgebra $A=A_{\bar{0}}\oplus A_{\bar{1}}$ and an object  $M=M_{\bar{0}}\oplus M_{\bar{1}}$ satisfies $f\colon M_{\bar{0}}\simeq M_{\bar{1}}$ as $A_{\bar{0}}$-modules. In this case, an isomorphism $(f,f^{-1})\colon M=M_{\bar{0}}\oplus M_{\bar{1}}\simeq M_{\bar{1}}\oplus M_{\bar{0}}=M$ as $A_{\bar{0}}$-modules has odd parity and it might define an isomorphism as $A$-modules, see \cite[\S 3.1]{CW}.
\subsection{Block decomposition and Monodromy filtration by Simple currents}\label{Block decomposition and Monodromy filtration}
Let $\mathcal{C}$ be a monoidal supercategory as in \S2.1 equipped with a braided monoidal supercategory structure.
Here we study decompositions and filtrations on $\mathcal{C}$ induced by simple currents. Take a simple current $S\in Ob(\mathcal{C})$. We set 
$$S^{n}:=
\begin{cases}S\boxtimes (S\boxtimes(\cdots (S\boxtimes S)\cdots),& (n\geq1),\\
1_{\mathcal{C}},& (n=0),\\
S^*\boxtimes (S^*\boxtimes(\cdots (S^*\boxtimes S^*)\cdots),& (n\leq-1).
\end{cases}$$ 
Since simple currents are closed under taking left, (right), dual and monoidal product by \cite[Proposition 2.11.3]{EGNO}, $S^{n}$, ($n\in\Z$), are simple currents and satisfy
$$S^{n}\boxtimes S^{m}\simeq S^{n+m},\quad (S^{n})^*\simeq S^{-n}.$$
Then by \eqref{eq: identification of hom space}, we have 
$$\End_{\mathcal{C}}(M)\simeq \End_{\mathcal{C}}(S^{n}\boxtimes M),\quad (M\in Ob(\mathcal{C})).$$
\begin{lemma}\label{lem: identification of hom space}
The map 
\begin{align}\label{identification}
\End_{\mathcal{C}}(M)\rightarrow  \End_{\mathcal{C}}(S^{n}\boxtimes M),\quad f\mapsto \id_{S^{n}}\boxtimes f
\end{align}
is an isomorphism of $\C$-superalgebras.
\end{lemma}
\proof
It is straightforward to show that the map
\begin{align*}
\End_{\mathcal{C}}(S^{n}\boxtimes M)\rightarrow \End_{\mathcal{C}}\left(S^{-n}\boxtimes(S^{n}\boxtimes M)\right),\quad f\mapsto
\id_{S^{-n}}\boxtimes (\id_{S^{n}}\boxtimes f)
\end{align*}
is an inverse of \eqref{identification} under the natural isomorphisms
\begin{align*}
\begin{split}
\End_{\mathcal{C}}\left(S^{-n}\boxtimes(S^{n}\boxtimes M)\right)&\simeq
\End_{\mathcal{C}}\left((S^{-n}\boxtimes S^{n})\boxtimes M\right)\\
&\simeq\End_\mathcal{C}(1_\mathcal{C}\boxtimes M)\simeq \End_{\mathcal{C}}(M).
\end{split}
\end{align*}
\endproof
By Lemma \ref{lem: identification of hom space}, the monodromy 
$$\mathcal{M}_{S^{n},M}=\mathcal{R}_{M,S^{n}}\circ\mathcal{R}_{S^{n},M}\in \End_{\mathcal{C}}(S^{n}\boxtimes M)$$
defines a unique element $m_S(n)\in \End_{\mathcal{C}}(M)^{\bar{0}}$ satisfying 
$$\mathcal{M}_{S^{n},M}=\id_{S^{n}}\boxtimes\ m_S(n).$$
\begin{proposition}\label{eq: one-parameter subgroup}
The endomorphism $m_S(n)$ is invertible and moreover
\begin{align*}
m_S(n)=m_S(1)^n\quad (n\in \Z).
\end{align*}
\end{proposition}
\proof
We show the assertion for $n\geq1$ by induction on $n$. By the hexagon identity, the following diagram commutes.
\begin{equation*}
\SelectTips{cm}{}
\xymatrix@W10pt@H12pt@R10pt@C40pt{
(S\boxtimes S^{n-1})\boxtimes M\ar[r]^{\mathcal{R}_{S^{n},M}}\ar[d]_{\simeq}& M\boxtimes (S\boxtimes S^{n-1})\ar@/^18pt/[ddd]^\simeq\ar[r]^{\mathcal{R}_{M,S^{n}}}& (S\boxtimes S^{n-1})\boxtimes M\\
S\boxtimes(S^{n-1}\boxtimes M)\ar[d]_{\id\oboxtimes \mathcal{R}}&&S\boxtimes (S^{n-1}\boxtimes M)\ar[u]_\simeq\\
S\boxtimes(M\boxtimes S^{n-1})\ar[d]_{\simeq}&&S\boxtimes (M\boxtimes S^{n-1})\ar[u]_{\id\oboxtimes \mathcal{R}}\\
(S\boxtimes M)\boxtimes S^{n-1}\ar[r]_{\mathcal{R}\oboxtimes \id}&(M\boxtimes S)\boxtimes S^{n-1}\ar[r]_{\mathcal{R}\oboxtimes \id}\ar@/^18pt/[uuu]^\simeq&(S\boxtimes M)\boxtimes S^{n-1}\ar[u]_\simeq.
}
\end{equation*}
Here all the isomorphisms without symbols are associativity isomorphisms. Then the morphisms in the bottom are composed as 
\begin{align*}
(\mathcal{R}_{M,S}\boxtimes \id_{S^{n-1}})\circ (\mathcal{R}_{S,M}\boxtimes \id_{S^{n-1}})
&=\mathcal{M}_{S,M}\boxtimes \id_{S^{n-1}}\\
&=(\id_S\oboxtimes m_S(1))\boxtimes \id_{S^{n-1}}.
\end{align*}
Now, we may use the naturality of the associativity $\mathcal{A}_{\bullet,\bullet,\bullet}$ and the braiding $\mathcal{R}_{\bullet,\bullet}$ to conclude $\mathcal{M}_{S^{n},M}=\id_{S^{n}}\oboxtimes m_S(1)^n$ since
\begin{align*}
\mathcal{M}_{S^{n},M}
&=\mathcal{R}_{M,S^{n}}\circ \mathcal{R}_{S^{n},M}\\
&=\mathcal{A}_{S,S^{n-1},M}\circ (\id_{S}\oboxtimes \mathcal{R}_{S^{n-1},M})\circ \mathcal{A}_{S,M,S^{n-1}}^{-1}\\
&\hspace{1cm} \circ(\mathcal{M}_{S,M}\boxtimes \id_{S^{n-1}})\circ \mathcal{A}_{S,M,S^{n-1}}\circ
(\id_S\oboxtimes \mathcal{R}_{S^{n-1},M})\circ \mathcal{A}^{-1}_{S,S^{n-1},M}\\
&=(\id_{S^{n}}\oboxtimes m_S(1)) \circ\mathcal{A}_{S,S^{n-1},M}\circ (\id_{S}\oboxtimes \mathcal{R}_{S^{n-1},M})\circ \mathcal{A}_{S,M,S^{n-1}}^{-1}\\
&\hspace{2cm} \circ\mathcal{A}_{S,M,S^{n-1}}\circ
(\id_S\oboxtimes \mathcal{R}_{S^{n-1},M})\circ \mathcal{A}^{-1}_{S,S^{n-1},M}\\
&=(\id_{S^{n}}\oboxtimes m_S(1))\circ
\mathcal{A}_{S,S^{n-1},M}\circ (\id_{S}\oboxtimes \mathcal{M}_{S^{n-1},M})\circ \mathcal{A}^{-1}_{S,S^{n-1},M}\\
&=(\id_{S^{n}}\oboxtimes m_S(1))\circ
\mathcal{A}_{S,S^{n-1},M}\circ (\id_{S}\oboxtimes (\id_{S^{n-1}}\oboxtimes m_S(1)^{n-1}))\circ \mathcal{A}^{-1}_{S,S^{n-1},M}\\
&=\id_{S^{n}}\oboxtimes m_S(1)^n.
\end{align*}
Similarly, we can show $m_S(-n)=m_S(-1)^n$ for $n\geq1$. Thus it remains to prove
$$m_S(1)m_S(-1)=\id_{M}.$$
For this, note that by using the same argument as above, we obtain that 
$$\mathcal{M}_{S\boxtimes S^*,M}=\id_{S\boxtimes S^*}\boxtimes (m_S(1)m_S(-1)).$$
Then $S\boxtimes S^*\simeq 1_\mathcal{C}$ and $\mathcal{M}_{1_\mathcal{C},M}=\id_{1_\mathcal{C}\boxtimes M}$ (see \cite[Proposition XIII. 1.2]{Kassel}) implies the assertion. This completes the proof.
\endproof
\begin{proposition}\label{decomposition of object}
An object $M\in Ob(\mathcal{C})$ admits the generalized eigenspace decomposition of $m_S(1)$
$$M=\bigoplus_{\alpha\in \C^\times}M_\alpha,$$
where
$$M_\alpha=\bigcup_{n\in \Z_{\geq0}}M_\alpha[n],\quad M_\alpha[n]:=\Ker  (m_S(1)-\alpha)^n.$$ 
\end{proposition}
\proof
It is immediate that we have 
\begin{align}\label{eq: decomposition}
M\supset \sum_{\alpha\in \C^\times}M_\alpha=\bigoplus_{\alpha\in \C^\times}M_\alpha
\end{align}
and that $\{M_\alpha[n]\}_{n\in \Z_{\geq0}}$ defines a filtration $M_\alpha[p]\subset M_\alpha[q]$, ($p<q$). Thus it remains to prove the equality of \eqref{eq: decomposition}. For this, we use the multiplication of $m_S(1)$ on $\End_\mathcal{C}(M)$. Since the $\C$-vector superspace $\End_\mathcal{C}(M)$ is finite dimensional by Assumption \ref{assumption2}, the multiplication of $m_S(1)$ gives a generalized eigenspace decomposition 
\begin{align*}
&\End_\mathcal{C}(M)=\bigoplus_{\alpha \in \C^\times} \End_\mathcal{C}(M)_\alpha,\\
&\End_\mathcal{C}(M)_\alpha:=\left\{f\in  \End_\mathcal{C}(M)_\alpha\mid (m_S(1)-\alpha)^nf=0,\quad (\forall n\gg0)\right\}. 
\end{align*}
This gives the decomposition $\id_M=\sum_{\alpha}\pi_\alpha$. Thus we obtain 
$$M=\id_M M=\sum_{\alpha\in\C^\times}\pi_\alpha M\subset\sum_{\alpha\in \C^\times}M_\alpha.$$
This completes the proof.
\endproof 
We introduce full subcategories $\mathcal{C}_\alpha[n]\subset \mathcal{C}$, ($\alpha\in \C^\times, n\in \Z_{\geq0}$) defined by 
$$\mathcal{C}_{\alpha}[n]:=\{M\in Ob(\mathcal{C})\mid (m_S(1)-\alpha)^{n+1}\id_M=0\}.$$
For a fixed $\alpha\in \C^\times$, they give a filtration
\begin{align}\label{monodromy filtration}
\mathcal{C}_\alpha[0]\subset \mathcal{C}_\alpha[1]\subset \cdots\subset \mathcal{C}_\alpha:=\bigcup_{n \geq 0}\mathcal{C}_\alpha[n].
\end{align}
Then by Proposition \ref{decomposition of object}, we have the following block decomposition of $\mathcal{C}$:
\begin{align}\label{monodromy decomposition}
\mathcal{C}=\bigoplus_{\alpha\in \C^\times}\mathcal{C}_\alpha.
\end{align}
We call it the \emph{monodromy decomposition} of $\mathcal{C}$ by $S$ and \eqref{monodromy filtration} the \emph{monodromy filtration} of $\mathcal{C}$ by $S$.
\begin{proposition}\label{decomposition by a single current}
(i) Any object $M$ in $\mathcal{C}_\alpha[m]$ is an extension 
$$0\rightarrow M_1\rightarrow M\rightarrow M_2\rightarrow 0$$
for some $M_1\in Ob(\mathcal{C}_\alpha[0])$ and $ M_2\in Ob(\mathcal{C}_\alpha[m-1])$.\\
(ii) We have $$\boxtimes\colon \mathcal{C}_\alpha[m]\times \mathcal{C}_\beta[n]\rightarrow \mathcal{C}_{\alpha\beta}[m+n].$$
\end{proposition}
\proof
(i) is immediate from the definition of $\mathcal{C}_\alpha[m]$. We prove (ii). Take $M\in Ob(\mathcal{C}_{\alpha}[m])$ and $N\in Ob(\mathcal{C}_{\beta}[n])$. Let us write the monodromy operator $m_S(1)$ for $M$ as $m_{S,M}(1)$ for clarity. By using the same diagram in the proof of Proposition \ref{eq: one-parameter subgroup}, we obtain
$$\mathcal{M}_{S,M\boxtimes N}=\id_S\boxtimes(m_{S,M}(1)\boxtimes m_{S,N}(1)),$$ 
and thus $m_{S,M\boxtimes N}(1)=m_{S,M}(1)\boxtimes m_{S,N}(1)$.
Now the assertion holds since
\begin{align*}
&(m_{S,M\boxtimes N}(1)-\alpha \beta)^{m+n+1}\\
&=(m_{S,M}(1)\boxtimes m_{S,N}(1)-\alpha \beta)^{m+n+1}\\
&=\left((m_{S,M}(1)-\alpha)\boxtimes m_{S,N}(1)+\alpha \id_{M}\boxtimes (m_{S,N}(1)-\beta)\right)^{m+n+1}\\
&=\sum_{k=0}^{m+n+1}\binom{m+n+1}{k}\alpha^k(m_{S,M}(1)-\alpha)^{m+n+1-k}\boxtimes m_{S,N}(1)^{m+n+1-k}(m_{S,N}(1)-\beta)^k\\
&=0.
\end{align*}
\begin{corollary}
The full subcategory 
$$\mathcal{C}[0]:=\bigoplus_{\alpha\in \C^\times}\mathcal{C}_\alpha[0]$$
is a braided monoidal supercategory.
\end{corollary}
\begin{remark}
If the simple current $S$ is of finite order $S^n\simeq 1_{\mathcal{C}}$, then the decomposition \eqref{monodromy decomposition} holds without Assumption \ref{assumption2}. Indeed, by the complete reducibility of representations of finite abelian groups, $m_S(1)^n=1$ implies the block decomposition 
$$\mathcal{C}=\mathcal{C}[0]=\bigoplus_{\alpha \in \Z/n\Z}\mathcal{C}_\alpha[0],$$
where $\Z_n \hookrightarrow \C^\times$, $[1]\mapsto e^{2\pi i/n}$ (cf. \cite[Lemma 3.17]{CKL}).
\end{remark}

Next, we generalize \eqref{monodromy decomposition} to a simultaneous decomposition by simple currents $\{S_g\}_{g\in G}$ parametrized by a group $G$, i.e.,
$S_g\boxtimes S_h\simeq S_{gh}$, ($g,h\in G$).
Since $\mathcal{C}$ is braided, we have
$$S_{gh}\simeq S_{g}\boxtimes S_h\xrightarrow[\mathcal{R}_{S_g,S_h}]{\simeq}S_h\boxtimes S_g\simeq S_{hg},$$ 
and $G$ is necessarily abelian. As the functoriality of $\mathcal{M}_{\bullet,\bullet}$ implies that each monodromy operator $m_{S_g}(1)$ on $M\in Ob(\mathcal{C})$ lies in the center of $\End_{\mathcal{C}}(M)$, the abelian group $G$ acts on $M$ by $g\mapsto m_{S_g}(1)$, ($g\in G$). 
We write $G^\vee:=\Hom_{\operatorname{Grp}}(G,\C^\times)$ for the dual of $G$ and introduce full subcategories $\mathcal{C}_\phi\subset \mathcal{C}$, ($\phi\in G^\vee$),  by
$$\mathcal{C}_\phi:=\{M\in Ob(M)\mid \forall g\in G, (m_{S_g}(1)-\phi(g))^N=0, (\forall N\gg0)\}.$$
Then Proposition \ref{decomposition of object} and the proof of Proposition \ref{decomposition by a single current} implies the following immediately.
\begin{theorem}\label{thm: monodromy decomposition} The supercategory $\mathcal{C}$ admits a decomposition
$$\mathcal{C}=\bigoplus_{\phi\in G^\vee}\mathcal{C}_\phi$$
as additive supercategories and the monoidal product respects the decomposition, i.e.,
$\boxtimes \colon\mathcal{C}_\phi\times \mathcal{C}_\psi\rightarrow \mathcal{C}_{\phi\psi}$.
\end{theorem}
Finally, we consider the case that $G$ is finitely generated. In this case, by the fundamental theorem of finitely generated abelian group, we have 
$G\simeq G_{\operatorname{fin}}\times \Z^n$
for some finite abelian group $G_{\operatorname{fin}}$ and non-negative integer $n\geq0$. Then the dual group $G^\vee$ is 
$G^\vee\simeq G_{\operatorname{fin}}^\vee\times (\C^\times)^n$.
Let $\mathcal{C}_{\phi,\alpha}[p]\subset \mathcal{C}$, ($\phi\in G_{\operatorname{fin}}^\vee$, $\alpha\in (\C^\times)^n$, $p\in \Z_{\geq0}^n$), denote the full subcategory whose objects consist of $M\in Ob(\mathcal{C})$ such that 
$$m_{S_g}(1)=\phi(g),\ (\forall g\in G_{\operatorname{fin}}),\quad (m_{S,1_i}-\alpha_i)^{p_i}|_M=0,\ (1\leq \forall i \leq n),$$
where $1_i$, ($1\leq i\leq n$), denotes the generator of the $i$-th component of $\Z^n\subset G$. 
For each $(\phi,\alpha)\in G^\vee$, we have $\mathcal{C}_{\phi,\alpha}[p]\subset\mathcal{C}_{\phi,\alpha}[q]$ if $q-p\in \Z_{\geq0}^n$ and define
$$\mathcal{C}_{\phi,\alpha}:=\bigcup_{p\in\Z_{\geq0}^n}\mathcal{C}_{\phi,\alpha}[p].$$
Then we have the following:
\begin{theorem}\label{general monodromy decomposition}
Assume that the group $G$ is finitely generated. Then, \\
(i) the supercategory $\mathcal{C}$ admits a decomposition
\begin{align}\label{block decomp by G}
\mathcal{C}=\bigoplus_{(\phi,\alpha)\in G^\vee}\mathcal{C}_{\phi,\alpha}.
\end{align}
and the monoidal product respects the filtration, i.e.,
$$\boxtimes \colon\mathcal{C}_{\phi,\alpha}[p]\times \mathcal{C}_{\psi,\beta}[q]\rightarrow \mathcal{C}_{\phi\psi,\alpha\beta}[p+q],$$
(ii) the full subcategory 
\begin{align*}
\mathcal{C}[0]:=\bigoplus_{(\phi,\alpha)\in G^\vee}\mathcal{C}_{\phi,\alpha}[0]
\end{align*}
is naturally a braided monoidal subsupercategory,\\
(iii) every object in $\mathcal{C}_{\phi,\alpha}[p]$ is expressed as an extension of 
certain objects in $\mathcal{C}_{\phi,\alpha}[0]$ and objects in 
$\mathcal{C}_{\phi,\alpha}[(p_1,\cdots, p_j-1,\cdots,p_n)]$ for some $j$.
\end{theorem}
We call the decomposition \eqref{block decomp by G} the \emph{monodromy decompositions} by $G$ and the filtration $\{\mathcal{C}_{\phi,\alpha}[p]\}$ the \emph{monodromy filtration} by $G$.
\begin{remark}
By Theorem \ref{general monodromy decomposition} (iii), every simple object lies in $\mathcal{C}[0]$. Thus $\mathcal{C}=\mathcal{C}[0]$ holds if $\mathcal{C}$ is semisimple. Equivalently, $\mathcal{C}\neq \mathcal{C}[0]$ implies that the supercategory $\mathcal{C}$ is not semisimple.
\end{remark}
\subsection{Fusion rings}\label{Fusion Rings}
Let $\mathcal{C}$ be a braided monoidal supercategory as in \S 2.2 with Assumption \ref{assumption1} and \ref{assumption2}. Let $\K(\mathcal{C})$ denote its Grothendieck group, which is generated over $\Z$ by isomorphism classes $[M]$ of objects $M$ in $\mathcal{C}$.
Note that, if every object in $\mathcal{C}$ of finite length, then the set $\irr\mathcal{C}$ of simple objects of $\mathcal{C}$ gives a $\Z$-basis of $\K(\mathcal{C})$.
From now on, we assume the following condition:
\begin{assumption}\label{assumption3}
The bifunctor $\boxtimes\colon\mathcal{C}\times\mathcal{C}\to\mathcal{C}$ is biexact.
\end{assumption}
Then $\K(\mathcal{C})$ is a commutative ring by 
$$\K(\mathcal{C})\times \K(\mathcal{C})\rightarrow\K(\mathcal{C}),\quad ([M],[N])\rightarrow[M\boxtimes N]$$
and is called the fusion ring of $\mathcal{C}$.
The set of simple currents in $\mathcal{C}$, denoted by $\Pic \mathcal{C}$, is naturally an abelian group by $\boxtimes$ and we regard the group ring $\Z[\Pic \mathcal{C}]$ as a subring of $\mathcal{K}(\mathcal{C})$. 

Suppose that we have a set of simple currents $\{S_g\}_{g\in G}$ parametrized by a finitely generated abelian group $G$ as in \S\ref{Block decomposition and Monodromy filtration}. By Theorem \ref{general monodromy decomposition} (i), (ii), the fusion ring is $G^\vee$-graded
$$\K(\mathcal{C})=\bigoplus_{\xi\in G^\vee}\K(\mathcal{C}_\xi).$$
By Theorem \ref{general monodromy decomposition} (iii), the embedding $\mathcal{C}[0]\subset \mathcal{C}$ induces an isomorphism 
$\K(\mathcal{C}[0])\simeq \K(\mathcal{C})$.
Finally, we note that $\K(\mathcal{C})$ is a $\Z[G]$-algebra, where $\Z[G]$ denotes the group ring of $G$, by
$$\Z[G]\times \K(\mathcal{C})\rightarrow \K(\mathcal{C}),\quad (g,[M])\rightarrow [S_g\boxtimes M].$$
Before we remark a criterion for the $\Z[G]$-freeness of $\K(\mathcal{C})$, we recall that Lemma \ref{lem: exactness of simple currents} implies that $G$ acts on $\irr \mathcal{C}$ by 
$$G \times \irr \mathcal{C}\rightarrow \irr \mathcal{C},\quad (g,M)\rightarrow S_g\boxtimes M.$$
Then it is clear that if every object in $\mathcal{C}$ is of finite length, then
 $\K(\mathcal{C})$ is free over $\Z[G]$ if and only if the $G$-action on $\irr \mathcal{C}$ is free.
Thus we have proved the following proposition:
\begin{proposition}\label{fusion algebra}
For a set of simple currents $\{S_g\}_{g\in G}$ in $\mathcal{C}$ parametrized a finitely generated abelian group $G$, the fusion rings $\K(\mathcal{C}[0])$ and $\K(\mathcal{C})$ are $G^\vee$-graded $\Z[G]$-algebras. Moreover, the embedding $\mathcal{C}[0]\subset \mathcal{C}$ gives an isomorphism
$\K(\mathcal{C})\simeq \K(\mathcal{C}[0])$
as $G^\vee$-graded $\Z[G]$-algebras. If every object in $\mathcal{C}$ is of finite length, then $\K(\mathcal{C})$ is free over $\Z[G]$ if and only if the $G$-action on $\irr \mathcal{C}$ is free.
\end{proposition}
\subsection{Algebra objects and Induction functor}\label{algebra object}
Following \cite{CKM}, we review the notion of (unital, associative, commutative) algebra objects and their module objects in a braided monoidal supercategory, originally introduced by \cite{KO} in the purely even case. We note that the authors of \cite{CKM} deal with superizations of non-super categories (see Remark \ref{Superalgebra Object}), but the proofs apply to our setting. 

In this subsection, $\mathcal{C}$ denotes an essentially small, $\C$-linear, monoidal supercategory whose underlying category is abelian, satisfying the following assumption which is a weaker version of Assumption \ref{assumption3}:
\begin{remark}\label{needs only right exactness}
In this subsection, we may replace Assumption \ref{assumption3} by a weaker one, that is, the bifunctor $\boxtimes$ is right exact.
\end{remark}
An algebra object in $\mathcal{C}$ is a triple $(\mathcal{E},\mu_\mathcal{E},\iota)$, (or $\mathcal{E}$ for simplicity), consisting of $\mathcal{E}\in Ob(\mathcal{C})$ and even morphisms  $\mu_\mathcal{E}\in \Hom_\mathcal{C}(\mathcal{E},\mathcal{E})_{\bar{0}}$ and $\iota\in \Hom_\mathcal{C}(1_\mathcal{C},\mathcal{E})_{\bar{0}}$ satisfying 
\begin{itemize}
    \item The morphism $\iota\colon 1_\mathcal{C}\rightarrow \mathcal{E}$ is injective
\end{itemize}
and the following commutative diagrams: 
\begin{enumerate}
\item[(A1)]Unity
\begin{equation*}
\SelectTips{cm}{}
\xymatrix@W15pt@H11pt@R10pt@C30pt{
1_\mathcal{C}\boxtimes \E\ar[r]^-{\iota\boxtimes \id_\E}\ar[rd]_-{l_\E}&\E\boxtimes \E\ar[d]^{\mu_\E}\\
&\E.
}
\end{equation*}
\item[(A2)]Associativity
\begin{equation*}
\SelectTips{cm}{}
\xymatrix@W15pt@H11pt@R18pt@C34pt{
\E\boxtimes (\E\boxtimes \E)\ar[r]^-{\id_{\E}\boxtimes \mu_{\E}}\ar[d]_-{\mathcal{A}_{\E,\E,\E}}&\E\boxtimes \E\ar[r]^{\mu_\E}&\E\\
(\E\boxtimes \E)\boxtimes \E\ar[r]^-{\mu_\E\boxtimes \id_{\E}}&\E\boxtimes\E \ar[ur]_-{\mu_\E}.
}
\end{equation*}
\item[(A3)]Commutativity
\begin{equation*}
\SelectTips{cm}{}
\xymatrix@W15pt@H11pt@R10pt@C30pt{
\E\boxtimes \E\ar[r]^-{\mathcal{R}_{\E,\E}}\ar[rd]_{\mu_\E}&\E\boxtimes \E\ar[d]^{\mu_\E}\\
&\E.
}
\end{equation*}
\end{enumerate}
\begin{remark}\label{Superalgebra Object}${}$\\
(1) A typical example of our braided monoidal supercategory is the superization $\mathcal{SC}$ of a braided abelian monoidal category $\mathcal{C}$, whose objects are pairs $M=(M_{\bar{0}},M_{\bar{1}})$ of $M_{\bar{i}}\in Ob(\mathcal{C})$. This induces a $\Z_2$-graded structure on the set of morphisms. In this case, an algebra object in $\mathcal{SC}$ is called a superalgebra object in $\mathcal{C}$, see \cite{CKM}.\\
(2) For an application to extensions of vertex superalgebras, the condition 
$$\Hom_\mathcal{C}(1_\mathcal{C},\E)\simeq\End_{\mathcal{C}}(1_\mathcal{C})$$
is often assumed so that the extended vertex superalgebra is of CFT type.
\end{remark}

An $\E$-module is a pair $(M,\mu_M)$, (or $M$ for simplicity), consisting of $M\in Ob(\mathcal{C})$ and an even morphism $\mu_M\in \Hom_\mathcal{C}(\E\boxtimes M,M)_{\bar{0}}$ satisfying the following commutative diagrams:
\begin{enumerate}
\item[(M1)]Unity
\begin{equation*}
\SelectTips{cm}{}
\xymatrix@W15pt@H11pt@R12pt@C30pt{
1_\mathcal{C}\boxtimes M\ar[r]^-{\iota\boxtimes \id_M}\ar[rd]_-{l_M}&\E\boxtimes M\ar[d]^{\mu_M}\\
&M.
}
\end{equation*}
\item[(M2)]Associativity
\begin{equation*}
\SelectTips{cm}{}
\xymatrix@W15pt@H11pt@R18pt@C32pt{
\E\boxtimes (\E\boxtimes M)\ar[r]^-{\id_{\E}\boxtimes \mu_{M}}\ar[d]_-{\mathcal{A}_{\E,\E,M}}&\E\boxtimes M\ar[r]^{\mu_M}&M\\
(\E\boxtimes \E)\boxtimes M\ar[r]^-{\mu_\E\boxtimes \id_{M}}&\E\boxtimes M \ar[ur]_-{\mu_M}.
}
\end{equation*}
\end{enumerate}
An $\E$-module $(M,\mu_M)$ is called \emph{local} if it further satisfies the following commutative diagram:
\begin{enumerate}
\item[(M3)] Locality
\begin{equation*}
\SelectTips{cm}{}
\xymatrix@W15pt@H11pt@R10pt@C29pt{
\E\boxtimes M\ar[r]^-{\mathcal{M}_{\E,M}}\ar[rd]_-{\mu_M}&\E\boxtimes M\ar[d]^{\mu_M}\\
&M.
}
\end{equation*}
\end{enumerate}
A morphism of $\E$-modules from $(M,\mu_M)$ to $(N,\mu_N)$ is a morphism $f\in \Hom_\mathcal{C}(M,N)$ satisfying the following commutative diagram:
\begin{equation*}
\SelectTips{cm}{}
\xymatrix@W15pt@H11pt@R12pt@C30pt{
\E\boxtimes M\ar[r]^-{\id_\E\boxtimes f}\ar[d]_-{\mu_M}&\E\boxtimes N\ar[d]^{\mu_N}\\
M\ar[r]^-{f}&N.
}
\end{equation*}
Let $\Rep(\E)$ denote the supercategory of $\E$-modules with morphisms of $\E$-modules, and $\Rep^0(\E)$ the full subcategory of $\Rep(\E)$ consisting of local $\E$-modules.

Although $\Rep(\mathcal{E})$ is just a $\C$-linear additive supercategory, the underlying category $\underline{\Rep(\mathcal{E})}$ is an abelian category. Furthermore, by the existence of the involutive autofunctor $\Pi_\mathcal{C}$, every parity-homogeneous morphism in $\Rep(\E)$ admits kernel and cokernel objects. The same is true for $\Rep^0(\mathcal{E})$.
Each $\Rep(\mathcal{E})$ and $\Rep^0(\mathcal{E})$ admits a natural monoidal structure in the following way \cite[\S 2]{CKM}. Consider the two compositions
\begin{align*}
&\xi_1\colon \E\boxtimes (M\boxtimes N)\xrightarrow{\mathcal{A}_{\E,\E,M}}(\E\boxtimes M)\boxtimes N\xrightarrow{\mu_M\boxtimes \id_N}M\boxtimes N,\\
&\xi_2\colon \E\boxtimes (M\boxtimes N)\xrightarrow{\mathcal{A}_{\E,M,N}} (\E\boxtimes M)\boxtimes N\xrightarrow{\mathcal{R}_{\E,M}\boxtimes N}(M\boxtimes \E)\boxtimes N\\
&\hspace{4cm}\xrightarrow{\mathcal{A}^{-1}_{M,\E,N}}M\boxtimes (\E\boxtimes N)\xrightarrow{\id_M\boxtimes \mu_N}M\boxtimes N.
\end{align*}
Then $M\boxtimes_\E N$ is defined by $M\boxtimes_\E N:=\mathrm{Coker}(\xi_1-\xi_2)$,
which is an object of $\mathcal{C}$. Let $\eta_{M,N}$ denote the canonical surjection $M\boxtimes N\rightarrow M\boxtimes_\E N$.
The $\E$-module structure 
$\mu_{M\boxtimes_\E N}\colon \E\boxtimes (M\boxtimes_\E N)\rightarrow M\boxtimes_\E N$ 
is the unique even morphism, which makes the diagram
\begin{equation*}
\SelectTips{cm}{}
\xymatrix@W15pt@H11pt@R12pt@C30pt{
\E\boxtimes (M\otimes N)\ar[r]^-{\xi_i}\ar[d]_-{\id_\E\boxtimes \eta_{M,N}}&M\boxtimes N \ar[d]^{\eta_{M,N}}\\
\E\boxtimes (M\boxtimes_\E N)\ar[r]^-{\mu_{M\boxtimes N}}&M\boxtimes_\E N.
}
\end{equation*}
commutes for $i=1,2$. The associativity 
$\mathcal{A}^\E_{\bullet,\bullet,\bullet}\colon \bullet\boxtimes_\E(\bullet \boxtimes_\E \bullet)\simeq (\bullet\boxtimes_\E \bullet)\boxtimes_{\E}\bullet$
is given by the family of unique even morphisms $\mathcal{A}^\E_{M,N,L}\colon M\boxtimes_\E(N \boxtimes_\E L)\simeq (M\boxtimes_\E N)\boxtimes_{\E}L$ for $M,N,L\in Ob(\Rep(\E))$, which make the following diagram commute:
\begin{equation*}
\SelectTips{cm}{}
\xymatrix@W15pt@H11pt@R12pt@C30pt{
M\boxtimes (N\otimes L)\ar[r]^-{\mathcal{A}_{M,N,L}}\ar[d]_-{\id_M\boxtimes \eta_{N,L}}&(M\boxtimes N)\boxtimes L \ar[d]^{\eta_{M,N}}\\
M\boxtimes (N\boxtimes_\E L)\ar[d]_-{\eta_{M,N\boxtimes_\E L}}&(M\boxtimes_\E N)\boxtimes L\ar[d]^-{\eta_{M\boxtimes_\E N,L}}\\
M\boxtimes_\E (N\boxtimes_\E L)\ar[r]^{\mathcal{A}^\E_{M,N,L}}&(M\boxtimes_\E N)\boxtimes_\E L.
}
\end{equation*}
The unit object of $\Rep(\E)$ and $\Rep^0(\mathcal{E})$ is $\E$ equipped with even natural morphisms 
$$l^\E_\bullet\colon \E\boxtimes_\E \bullet \simeq \bullet,\quad r^\E_\bullet\colon \bullet \boxtimes_\E \E\simeq \bullet$$
given by the family of unique even morphisms $l^\E_M\colon \E\boxtimes_\E M\simeq M$ and $r^\E_M\colon M\boxtimes_\E \E\simeq M$, which make the following diagrams commute:
\begin{equation*}
\SelectTips{cm}{}
\xymatrix@W15pt@H11pt@R12pt@C30pt{
\E\boxtimes M \ar[r]^-{\mu_M}\ar[d]_-{\eta_{\E,N}}&M\ar@{=}[d]
&M\boxtimes \E\ar[r]^-{\mathcal{R}_{M,\E}^{-1}}\ar[d]_-{\eta_{M,\E}}&\E\boxtimes M\ar[r]^-{\mu_M}&M\ar@{=}[d]\\
\E\boxtimes_\E M\ar[r]^-{l^\E_M}&M&M\boxtimes_\E \E\ar[rr]^-{r^\E_M}&&M.
}
\end{equation*}
The braiding $\mathcal{R}_{\bullet,\bullet}$ on $\mathcal{C}$ induces a braiding $\mathcal{R}^{\E}_{\bullet,\bullet}$ on $\Rep^0(\E)$, which is a family of unique even morphisms $\mathcal{R}^\E_{M,N}\colon M\boxtimes_\E N\simeq N\boxtimes_\E M$, which make the diagram
\begin{equation*}
\SelectTips{cm}{}
\xymatrix@W15pt@H11pt@R12pt@C30pt{
M\boxtimes N\ar[r]^-{\mathcal{R}_{M,N}}\ar[d]_-{\eta_{M,N}}&N\boxtimes M \ar[d]^{\eta_{N,M}}\\
M\boxtimes_\E N\ar[r]^-{\mathcal{R}^\E_{M,N}}&N\boxtimes_\E M
}
\end{equation*}
commute. To summarize, we obtain the following.
\begin{theorem}[{\cite[Theorem 2.53]{CKM}}]
The supercategory $\Rep(\E)$ (resp.\ $\Rep^0(\E)$) is naturally a $\C$-linear additive (resp.\ braided) monoidal supercategory such that the underlying category $\underline{\Rep(\E)}$, (resp.\ $\underline{\Rep^0(\E)}$) is an abelian category.
\end{theorem}
For $M\in Ob(\mathcal{C})$, we define an $\mathcal{E}$-module $(\mathcal{F}(M), \mu_{\mathcal{F}(M)})$ by 
\begin{align*}
&\F(M):=\E\boxtimes M,\\
&\mu_{\mathcal{F}(M)}\colon \E\boxtimes (\E\boxtimes M)\simeq (\E\boxtimes \E)\boxtimes M\xrightarrow{\mu_\E\boxtimes \id_M}\E\boxtimes M.
\end{align*}
We also define a superfunctor $\mathcal{F}\colon \mathcal{C}\rightarrow\Rep(\mathcal{E})$ by $M\mapsto (\mathcal{F}(M), \mu_{\mathcal{F}(M)})$ and $\F(f):=\id_\E\boxtimes f\in \Hom_{\Rep(\mathcal{E})}(\F(M),\F(N))$ for $f\in \Hom_{\mathcal{C}}(M,N)$.
The superfunctor $\mathcal{F}$ is called the induction functor and enjoys the following property.
\begin{theorem}[{\cite[Theorem 2.59]{CKM}}]
The induction functor $\mathcal{F}\colon \mathcal{C}\rightarrow\Rep(\mathcal{E})$ is a $\C$-linear, additive, strong monoidal superfunctor.
\end{theorem}
Let $\mathcal{C}^0$ denote the full subcategory of $\mathcal{C}$ consisting of objects $M\in Ob(\mathcal{C})$ such that 
$\mathcal{M}_{\E,M}=\id_{\E,M}$. We call $\mathcal{C}^0$ the category of $\mathcal{E}$-local objects. 
\begin{theorem}[\cite{CKM}]\label{locality statement}
We have the following.
\begin{enumerate}
\item The supercategory $\mathcal{C}^0$ is a braided monoidal subsupercategory of $\mathcal{C}$.
\item For $M\in Ob(\mathcal{C})$, the object $\F(M)$ lies in $\Rep^0(\E)$ if and only if $M\in Ob(\mathcal{C}^0)$.
\item The restriction of $\F$ to $\mathcal{C}^0$ gives a braided monoidal superfunctor
$$\F\colon \mathcal{C}^0\rightarrow \Rep^0(\E).$$
\end{enumerate}
\end{theorem}
\proof
By \cite[Theorem 2.67]{CKM}, $\mathcal{C}^0$ is a $\C$-linear additive braided monoidal supercategory, Thus to show (1), it remains to show that $\mathcal{C}^0$ is closed under kernel and cokernel, which immediately follows from the exactness of $\mathcal{E}\boxtimes\bullet$ in Assumption \ref{assumption3}. (2) is \cite[Proposition 2.65]{CKM} and (3) is \cite[Theorem 2.67]{CKM}.
\endproof
At last, the induction functor $\mathcal{F}$ is related to the forgetful functor
$$\G\colon \Rep(\E)\rightarrow \mathcal{C},\quad (M,\mu_M)\rightarrow M.$$
by the Frobenius reciprocity:
\begin{proposition}[{\cite[Lemma 2.61]{CKM}}]\label{Frobenius reciprocity}
The superfunctor $\G\colon \Rep(\E)\rightarrow \mathcal{C}$ is right adjoint to the superfunctor $\F\colon \mathcal{C}\rightarrow \Rep(\E)$, that is, we have a natural isomorphism
\begin{align*}
\Hom_{\Rep(\E)}(\F(N),M)\simeq \Hom_\mathcal{C}(N,\G(M)),
\end{align*}
for $M\in Ob(\Rep(\E))$ and $N\in Ob(\mathcal{C})$. 
More explicitly, for $f\in  \Hom_\mathcal{C}(N,\G(M))$, the corresponding morphism of $\E$-modules is given by
\begin{align*}
\F(N)=\E\boxtimes N\rightarrow M,\quad a\boxtimes m\mapsto \mu_M(a\boxtimes f(m)).
\end{align*}
\end{proposition}
\subsection{Categorical simple current extensions}\label{Categorical simple current extensions}
Let $\mathcal{C}$ be a braided monoidal supercategory as in \S \ref{Block decomposition and Monodromy filtration}, satisfying Assumption \ref{assumption1}--\ref{assumption3} and the following assumption.
\begin{assumption}\label{assumption4}
Every object in $\mathcal{C}$ has finite length.
\end{assumption}
Let $\E$ be an algebra object in $\mathcal{C}$ of the form
$$\E=\bigoplus_{g\in G} S_g,$$
where $\{S_g\}_{g\in G}\subset \Pic\mathcal{C}$ is a set of simple currents parametrized by a finite abelian group $G$ with 
$S_e=1_\mathcal{C}$. Here $e\in G$ denotes the unit of $G$. We call $\mathcal{E}$ a categorical simple current extension of $1_\mathcal{C}$.
In the rest of this subsection, we assume 
\begin{enumerate}
\item[(S1)]
the product $\mu_\E$ restricts to a non-zero morphism $S_g\boxtimes S_h\rightarrow S_{gh}$ for $g,h\in G$,
\item[(S2)]
the action of $G$ on $\irr\mathcal{C}$ is fixed-point free.
\end{enumerate}
Note that these assumptions imply $S_g\simeq S_h$ if and only if $g=h$.
\begin{proposition}\label{induction} Suppose (S1) and (S2).
\begin{enumerate}
\item If $M$ is a simple object in $\mathcal{C}$, so is $\mathcal{F}(M)$ in $\mathrm{Rep}(\mathcal{E})$.
\item For a simple object $M$ in $\mathrm{Rep}(\mathcal{E})$, there exists a simple object $N$ in $\mathcal{C}$ such that $M\simeq \mathcal{F}(N)$ in $\mathrm{Rep}(\mathcal{E})$.
\item For simple objects $M$ and $N$ in $\mathcal{C}$, we have $\mathcal{F}(M)\simeq \mathcal{F}(N)$ in $\mathrm{Rep}(\mathcal{E})$ if and only if we have $M\simeq S_g\boxtimes N$ for some $g\in G$.
\item The supercategory $\mathcal{C}$ is semisimple if and only if $\mathrm{Rep}(\mathcal{E})$ is semisimple. In this case, $\mathcal{C}^0\subset \mathcal{C}$ and $\mathrm{Rep}^0(\mathcal{E})\subset \mathrm{Rep}(\mathcal{E})$ are also semisimple.
\end{enumerate}
\end{proposition}
\proof
Although these statements are well-known in the theory of vertex algebras, (see e.g., \cite[Proposition 4.5]{CKM}), we include a proof for the completeness of the paper.  
First we prove (1). Let $M$ be a simple object in $\mathcal{C}$. Then $\mathcal{F}(M)=\oplus_{g\in G} S_g\boxtimes M$ and all summands are pairwise non-isomorphic by (S2). Let $\mathcal{N}$ be a nonzero subobject of $\mathcal{F}(M)$ in $\mathrm{Rep}(\mathcal{E})$.
Since $\mathcal{N}$ is a semisimple object in $\mathcal{C}$, it has a simple subobject which is isomorphic to $S_g\boxtimes M$ for some $g\in G$. By (S1), the structure morpihsm $\mu_{\mathcal{F}(M)}$ restricts to an isomorphism $S_h\boxtimes (S_g\boxtimes M)\xrightarrow{\simeq}S_{hg}\boxtimes M$ for any $h\in G$. Since $\mathcal{N}$ is closed under the $\mathcal{E}$-action, it contains $\sum_{h\in G}\mu_{\mathcal{F}(M)}(S_h\boxtimes(S_g\boxtimes M))=\oplus_{h\in G}S_{hg}\boxtimes M=\mathcal{F}(M)$. This proves (1). Then (2) and (3) follow from (1) and Proposition \ref{Frobenius reciprocity}.
Finally, we show (4). 
Assume that $\mathrm{Rep}(\mathcal{E})$ is semisimple and take $N\in Ob(\mathcal{C})$. Since $\mathcal{F}(N)$ is semisimple, we have
$$\mathcal{F}(N)\simeq \bigoplus_{i\in I} \mathcal{N}_i$$
for some simple objects $\mathcal{N}_i$ in $\mathrm{Rep}(\mathcal{E})$ indexed by a finite set $I$. Then by (2), we may replace $\mathcal{N}_i$ by $\mathcal{F}(N_i)$ for some simple objects $N_i$ in $\mathcal{C}$. Thus,
\begin{align*}
M\subset \mathcal{F}(N)\simeq \bigoplus_{i\in I}\mathcal{F}(N_i)= \bigoplus_{\begin{subarray}c i\in I\\ g\in G\end{subarray}} S_g\boxtimes N_i.
\end{align*}
Since $S_g\boxtimes N_i$ are all simple objects in $\mathcal{C}$, $N$ is semisimple.
To prove the converse, assume that $\mathcal{C}$ is semisimple. Since every object in $\mathcal{C}$ has finite length, so does every object in $\mathrm{Rep}(\mathcal{E})$. Thus to show that $\mathrm{Rep}(\mathcal{E})$ is semisimple, it suffices to show the splitting of any short exact sequence 
\begin{align}\label{short exact sequence}
0\rightarrow \mathcal{N}_1\rightarrow \mathcal{N}\rightarrow \mathcal{N}_2\rightarrow 0
\end{align}
in $\mathrm{Rep}(\mathcal{E})$ where $\mathcal{N}_1$, $\mathcal{N}_2$ are simple. By (2), we may assume $\mathcal{N}_i=\mathcal{F}(N_i)$ for some simple object $N_i\in Ob(\mathcal{C})$. If $\mathcal{F}(N_1)\not\simeq \mathcal{F}(N_2)$, then $S_g\boxtimes N_1\not\simeq S_h \boxtimes N_2$ for all $g,h\in G$ in $\mathcal{C}$. This implies the splitting of \eqref{short exact sequence} in $\mathrm{Rep}(\mathcal{E})$. Thus we may assume $\mathcal{F}(N_1)\simeq \mathcal{F}(N_2)$ and moreover $N:=N_1=N_2$ from the beginning. Since $G$ is a finite abelian group, it is isomorphic to some direct product of cyclic groups $\prod_{i} \mathbb{Z}_{n_i}$. Then it suffices to show the splitting in the case of $G=\mathbb{Z}_n$ for some $n\in \mathbb{Z}_{>0}$. Let $S_p$ denote the simple current corresponding to $p\in \mathbb{Z}_n$. Since the space of intertwining operators
$I\binom{S_{p+1}\boxtimes N}{S_{1}\ S_{p}\boxtimes N}$ is one dimensional, we may take its basis by the intertwining operator
$$S_{1}\boxtimes (S_{p}\boxtimes N)\simeq (S_{1}\boxtimes S_{p})\boxtimes N\simeq S_{p+1}\boxtimes N$$
used in the definition of $\mathcal{F}(N)$. On the other hand, the restriction of the $\mathcal{E}$-module structure of $\mathcal{N}$ gives an intertwining operator $S_1\boxtimes\mathcal{N}\rightarrow\mathcal{N}$.
Along with the decomposition $\mathcal{N}\simeq \mathcal{F}(N)\oplus \mathcal{F}(N)$ in $\mathcal{C}$, the intertwining operator is expressed by the matrix:
$$K:=\left(\begin{array}{c|c}E&A\\\hline 0&E\end{array}\right)$$
where $E=\sum_{i\in {\mathbb{Z}_n}} E_{i+1,i}$ and $A=\sum_{i\in \mathbb{Z}_n} a_i E_{i+1,i}$ for some $a_i\in \mathbb{C}$. Since $S_1^{n}\simeq 1_{\mathcal{C}}$, we have $\sum_i a_i=0$. Then it is straightforward to check that $K$ is conjugate to the matrix $K$ with $a_i=0$ for all $i$. This implies that we may take a decomposition $\mathcal{N}=\mathcal{F}(N)\oplus \mathcal{F}(N)$ in $\mathcal{C}$ which is preserved by the action of $S_1\subset \mathcal{E}$. Since the other action of $S_p\subset \mathcal{E}$ is obtained from the action of $S_{1}$ via iteration, the above decomposition of $\mathcal{N}$ is actually the decomposition as an $\mathcal{E}$-module. The remaining statements in (4) are now obvious. This completes the proof.
\endproof
In particular, we have the following.
\begin{corollary}\label{classification of irreducibles}
The induction functor $\mathcal{F}\colon \mathcal{C}\rightarrow \mathrm{Rep}(\mathcal{E})$ induces the following natural identifications:
\begin{enumerate}
\item $\mathrm{Irr}(\mathrm{Rep}(\mathcal{E}))\simeq \mathrm{Irr}(\mathcal{C})/G$ and $\mathrm{Pic}(\mathrm{Rep}(\mathcal{E}))\simeq\mathrm{Pic}(\mathcal{C})/G$;
\item $\mathrm{Irr}(\mathrm{Rep}^0(\mathcal{E}))\simeq \mathrm{Irr}(\mathcal{C}^0)/G$ and $\mathrm{Pic}(\mathrm{Rep}^0(\mathcal{E}))\simeq\mathrm{Pic}(\mathcal{C}^0)/G$.
\end{enumerate}
\end{corollary}
Since the induction functor is a monoidal superfunctor, we may write the fusion ring of $\mathrm{Rep}(\mathcal{E})$ in terms of $\mathcal{C}$.
\begin{corollary}\label{Grothendieck groups for induction}
Suppose (S1) and (S2). If the superfunctor $\boxtimes_{\mathcal{E}}$ is bi-exact, then the induction functor $\mathcal{F}\colon \mathcal{C}\rightarrow \mathrm{Rep}(\mathcal{E})$ induces the following isomorphism of rings:\begin{enumerate}
\item $\mathcal{K}(\mathrm{Rep}(\mathcal{E}))\simeq \mathcal{K}(\mathcal{C})/\mathcal{I}$
 where $\mathcal{I}=\langle [M]-[S_g\boxtimes M]\mid g\in G, M\in Ob(\mathcal{C})\rangle$;
\item $\mathcal{K}(\mathrm{Rep}^0(\mathcal{E}))\simeq \mathcal{K}(\mathcal{C}^0)/\mathcal{I}^0$
where $\mathcal{I}^0=\langle [M]-[S_g\boxtimes M]\mid g\in G, M\in Ob(\mathcal{C}^0)\rangle$.
\end{enumerate}
\end{corollary}
\proof 
(1) follow from Proposition \ref{induction}. (2) follows from (1) and Theorem \ref{locality statement}.
\endproof
\section{Proof of Proposition \ref{fusion ring of subregular}}\label{Level_Rank_Duality}

We recall the level-rank duality between $L_m(\mathfrak{sl}_n)$ and $L_n(\mathfrak{sl}_m)$ for $n,m\geq2$,  \cite{Fr,OS}. The isomorphism $\C^{nm}\simeq \C^n\otimes \C^m$ induces an embedding of Lie algebras $\mathfrak{sl}_n\oplus \mathfrak{sl}_m\hookrightarrow \mathfrak{sl}_{nm}$ and thus $L_m(\mathfrak{sl}_n)\otimes L_n(\mathfrak{sl}_m)\hookrightarrow L_1(\mathfrak{sl}_{nm})$. It is a conformal embedding and gives a finite decomposition
$$L_1(\mathfrak{sl}_{nm})\simeq \bigoplus_{\begin{subarray}c \lambda\in \widehat{P}_+^m(n)\\ \pi_{P/Q}(\lambda)=0\end{subarray}}L_m(\lambda)\otimes L_n(\lambda^t)$$
as $L_m(\mathfrak{sl}_n)\otimes L_n(\mathfrak{sl}_m)$-modules where $\lambda\mapsto \lambda^t$ denotes the transpose. More precisely, let $C_{n,m}$ denote the set of Young diagrams lying in the $n\times m$ rectangle. Then we have an embedding $\widehat{P}_+^m(n)\hookrightarrow C_{n,m}$, ($\lambda=\sum_{i\in \Z_n}a_i\Lambda_{i}\mapsto \sqcup_{i=1}^{n-1}a_i R_i$) where $R_i$ denote the column of boxes of height $i$. ($R_n$ is identified with the empty set.) Then the transpose $\lambda\mapsto \lambda^t$ is just the transpose of Young diagrams, e.g.,
\begin{align*}
\setlength{\unitlength}{1mm}
\begin{picture}(0, 0)(20,7)
\put(-33,3){$\widehat{P}_+^5(3)\ni\Lambda_0+\Lambda_1+3\Lambda_2=$}
\put(12,0){\line(0,1){03}}
\put(12,0){\line(1,0){9}}
\put(12,3){\line(1,0){12}}
\put(12,6){\line(1,0){12}}
\put(12,0){\line(0,1){6}}
\put(15,0){\line(0,1){6}}
\put(18,0){\line(0,1){6}}
\put(21,0){\line(0,1){6}}
\put(24,3){\line(0,1){3}}
\put(30,3){$\xrightarrow{t}$}
\put(38,-3){\line(0,1){12}}
\put(41,-3){\line(0,1){12}}
\put(44,0){\line(0,1){9}}
\put(38,9){\line(1,0){6}}
\put(38,6){\line(1,0){6}}
\put(38,3){\line(1,0){6}}
\put(38,0){\line(1,0){6}}
\put(38,-3){\line(1,0){3}}
\put(48,3){$=\Lambda_0+\Lambda_3+\Lambda_4\in \widehat{P}_+^3(5)$}
\put(85,0){.}
\end{picture}
\quad\\ \quad\\
\end{align*}
Note that $\pi_{P/Q}(\lambda)=\ell(\lambda)\in \Z_n$ where $\ell(\lambda)$ denotes the number of boxes of the Young diagram of $\lambda$. 
By the Frenkel--Kac construction, the natural embedding $Q(\mathfrak{sl}_{nm})\hookrightarrow\Z^{nm}$ gives rise to a vertex algebra embedding $L_1(\mathfrak{sl}_{nm})\simeq V_{Q(\mathfrak{sl}_{nm})}\hookrightarrow V_{\Z^{nm}}$. Then $\Com(L_1(\mathfrak{sl}_{nm}),V_{\Z^{nm}})\simeq V_{\sqrt{nm}\Z}$ with $\sqrt{nm}\Z\hookrightarrow \Z^{nm};\ a\sqrt{nm}\mapsto (a,a,\ldots,a)$ and we have
$$V_{\Z^{nm}}\simeq \bigoplus_{a\in \Z_{nm}}L_{nm}(\Lambda_a)\otimes V_{\frac{a}{\sqrt{nm}}+\sqrt{nm}\Z}$$
as $L_1(\mathfrak{sl}_{nm})\otimes V_{\sqrt{nm}\Z}$-modules. Now, by the branching law of $L_1(\Lambda_a)$ as an $L_m(\mathfrak{sl}_n)\otimes L_n(\mathfrak{sl}_m)$-module \cite[Theorem 4.1]{OS}, we obtain 
\begin{align}\label{pre level-rank duality} 
V_{\Z^{nm}}\simeq \bigoplus_{a\in \Z_{nm}}\left(\bigoplus_{\substack{\lambda\in \widehat{P}_+^m(n)\\\Proj(\lambda)=a}}L_m(\lambda)\otimes L_n(\sigma^{\frac{a-\ell(\lambda)}{n}}(\lambda^t))\right)\otimes  V_{\frac{a}{\sqrt{nm}}+\sqrt{nm}\Z}
\end{align}
as $L_m(\mathfrak{sl}_n)\otimes L_n(\mathfrak{sl}_m)\otimes V_{\sqrt{nm}\Z}$-modules, where we identify $a\in\Z_{nm}$ with its image in $\Z_m$ under the natural projection $\Z_{nm}\to\Z_{m}$.
Thus, $\mathcal{E}_{m,n}:=\Com(L_m(\mathfrak{sl}_n),V_{\Z^{nm}})$ is an order $m$ simple current extension of $L_n(\mathfrak{sl}_m)\otimes V_{\sqrt{nm}\Z}$ and we have
$$\Com(L_m(\mathfrak{sl}_n),V_{\Z^{nm}})\simeq \bigoplus_{a\in \Z_{m}} L_n(n\Lambda_a)\otimes V_{\frac{an}{\sqrt{nm}\Z}+\sqrt{nm}\Z}.$$
Therefore, by Theorem \ref{extension law}, we have
\begin{align*}
\mathrm{Irr}(\E_{m,n})= \left\{\mathbf{M}(\lambda,a)\,\Bigg| \begin{array}{c}(\lambda,a)\in \widehat{P}_+^n(m)\times \Z_{nm}/\Z_m\\\Proj(\lambda)=a\in \Z_m\end{array}\right\},
\end{align*}
where $\Z_m$ acts on $\widehat{P}_+^n(m)\times \Z_{nm}$ by $r \cdot (\lambda, a) = (\sigma^r(\lambda), a + r n)$, ($r \in \Z_m$), and
$$\mathcal{K}(\mathcal{E}_{m,n})\simeq (\mathcal{K}(L_n(\mathfrak{sl}_m))\otimes_{\Z[\Z_m]}\Z[\Z_{nm}])^{\Z_m},\quad \mathbf{M}(\lambda,a)\mapsto L_n(\lambda)\otimes [a].$$
Now, \eqref{pre level-rank duality} implies the decomposition
$$V_{\Z^{nm}}\simeq \bigoplus_{\lambda\in \widehat{P}_+^m(n)}L_m(\lambda)\otimes \mathbf{M}(\lambda^t,\ell(\lambda))$$
as $L_m(\mathfrak{sl}_n)\otimes \mathcal{E}_{m,n}$-modules. This gives an one-to-one correspondence
$$\irr(L_m(\mathfrak{sl}_n))\rightarrow \irr(\mathcal{E}_{m,n}),\quad L_m(\lambda)\mapsto \mathbf{M}(\lambda^t,\ell(\lambda))$$
which implies a braided-reverse equivalence of braided tensor categories between $L_m(\mathfrak{sl}_n)\Mod$ and $\mathcal{E}_{m,n}\Mod$ by \cite{CKM2} and, in particular, an isomorphism
\begin{align}\label{level-rank duality}
\mathcal{K}(L_m(\mathfrak{sl}_n))\simeq \left(\mathcal{K}(L_n(\mathfrak{sl}_m))\underset{\Z[\Z_m]}{\otimes}\Z[\Z_{nm}]\right)^{\Z_m},\quad L_m(\lambda)\mapsto L_n(\lambda^t)\otimes [\ell(\lambda)].
\end{align}

\end{document}